\documentclass[11pt]{amsart}
\usepackage{pb-diagram,amssymb,epic,amsthm,amsmath, amscd,eepic,verbatim,epsfig, hyperref}
\usepackage{graphicx}
\numberwithin{equation}{section}
\numberwithin{figure}{section}

\def\hhref#1{\href{http://arxiv.org/abs/#1}{arXiv:#1}} 

\newtheorem{theorem}{Theorem}[subsection]

\newtheorem{proposition}[theorem]{Proposition}
\newtheorem{lemma}[theorem]{Lemma}

\newtheorem{appthm}{Theorem}[section]
\newtheorem{applem}[appthm]{Lemma}
\newtheorem{appcor}[appthm]{Corollary}
\newtheorem{appprop}[appthm]{Proposition}

\newtheorem{lem}[theorem]{}
\newtheorem{theorem1}{Theorem}
\newtheorem*{thm*}{Theorem}
\theoremstyle{definition}
\newtheorem{definition}[theorem]{Definition}
\newtheorem*{definition*}{Definition}
\theoremstyle{remark}

\newtheorem{example}[theorem]{Example}
\newcommand{\blem}{\begin{lem} \rm}
\newcommand{\elem}{\end{lem}}
\newcommand\A{\mathcal{A}}
\newcommand\M{\mathcal{M}}

\newcommand\mC{\mathcal{C}}
\renewcommand\M{\mathcal{M}}

\newcommand{\bH}{\mathbb{H}}

\newcommand{\F}{\mathcal{F}}

\newcommand{\R}{\mathbb
{R}}

\newcommand{\RR}{\mathcal{R}}
\newcommand{\I}{\mathcal{I}}
\newcommand{\C}{\mathbb{C}}

\newcommand{\bB}{\mathcal{B}}
\newcommand{\sS}{\mathcal{S}}

\newcommand{\on}{\operatorname}

\newcommand{\dist}{\on{dist}}
\newcommand\bra[1]{ < \kern-.7ex {#1} \kern-.7ex >} 

\newcommand{\Fun}{\on{Fun}}

\newcommand{\Hom}{ \on{Hom}}

\renewcommand{\ker}{ \on{ker}}
\newcommand{\im}{ \on{im}}

\newcommand{\codim}{\on{codim}}

\newcommand{\ssm}{\kern-.5ex \smallsetminus \kern-.5ex}

\newcommand\dirac{/\kern-1.2ex\partial} 
\newcommand\qu{/\kern-.7ex/} 
\newcommand\lqu{\backslash \kern-.7ex \backslash} 

\newcommand\dr{r_+ \kern-.7ex - \kern-.7ex r_-}
 
\newcommand{\labell}\label

\newcommand{\lra}{\longrightarrow}

\newcommand{\ol}{\overline}

\newcommand\lam{\lambda}

\newcommand\eps{\epsilon}

\newcommand\cE{\mathcal{E}}

\newcommand\vol{\on{vol}}

\newcommand\ul{\underline}

\renewcommand\Im{\on{Im}}

\newcommand\rrho{r_\rho}
\newcommand\bdefn{\begin{definition}}
\newcommand\edefn{\end{definition}}
\newcommand\bea{\begin{eqnarray*}}
\newcommand\eea{\end{eqnarray*}}
\newcommand\bcv{\left[ \begin{array}{r} }
\newcommand\ecv{\end{array} \right] }

\newcommand\bma{\left[ \begin{array} }
\newcommand\ema{\end{array} \right]}
\newcommand\ben{\begin{enumerate}}
\newcommand\een{\end{enumerate}}
\newcommand\bex{\begin{example}}
\newcommand\bsj{\left\{ \begin{array}{rrr} }
\newcommand\esj{\end{array} \right\}}

\newcommand\eex{\end{example}}

\newcommand\Don{ {\on{Don}}}
\newcommand\Fuk{{\on{Fuk}}}

\newcommand\Id{{\on{Id}}}

\newcommand\sx{*\kern-.5ex_X}
\newcommand{\ainfty}{{$A_\infty$\ }}
\renewcommand{\M}{\mathcal{M}}
\newcommand\cI{\mathcal{I}}
\def\mathunderaccent#1{\let\theaccent#1\mathpalette\putaccentunder}
\def\putaccentunder#1#2{\oalign{$#1#2$\crcr\hidewidth \vbox
to.2ex{\hbox{$#1\theaccent{}$}\vss}\hidewidth}}

\newcommand{\cR}{\mathcal{R}}
\newcommand{\cG}{\mathcal{G}}
\newcommand{\B}{\mathcal{B}}
\newcommand\E{\mathcal{E}}

\newcommand\ulu{\underline{u}}
\newcommand\ulM{\underline{M}}
\newcommand{\cP}{\mathcal{P}}
\newcommand\ulxi{\underline{\xi}}
\newcommand\uleta{\underline{\eta}}
\newcommand\ulY{\underline{Y}}
\newcommand\ulK{\underline{K}}
\newcommand\ulJ{\underline{J}}
\newcommand\ulx{\underline{x}}
\newcommand\ulH{\underline{H}}
\newcommand\ulv{\underline{v}}
\newcommand{\dvol}{\on{dvol}}
\newcommand{\bS}{\mathcal{S}}
\newcommand{\txi}{\tilde{\xi}}
\newcommand{\delbar}{\overline{\partial}}
\newcommand{\half}{\frac{1}{2}}

\begin{document}

\title{Gluing pseudoholomorphic quilted disks.}
\author{Sikimeti Ma'u}
\maketitle

\begin{abstract}
We construct families of quilted surfaces parametrized by the Stasheff multiplihedra, and define moduli spaces of pseudoholomorphic quilted disks using the holomorphic quilts theory of Wehrheim and Woodward.  We prove a gluing theorem for isolated, regular pseudoholomorphic quilted disks. This analytical result is a fundamental ingredient for the construction of \ainfty functors associated to Lagrangian correspondences.
\end{abstract}


\section{Introduction}

\subsection{The setting} The Fukaya category of a symplectic manifold $(M,\omega)$ is an algebraic construct built out of certain Lagrangian submanifolds and their associated Floer chain groups. Roughly speaking the objects of $\Fuk(M)$ are Lagrangian submanifolds, while morphisms are elements of the Floer chain groups, $CF(L, L^\prime)$. The \ainfty structure of the Fukaya category comes from a whole sequence of higher compositions, 
\[
\mu^n : CF(L_{n-1}, L_n)\otimes \ldots \otimes CF(L_0, L_1) \to CF(L_0, L_d)
\]
for $n \geq 1$, which satisfy a sequence of quadratic relations called the \ainfty associativity relations.
The maps $\mu^n$ are defined geometrically by counts of pseudoholomorphic $(n+1)$-gons, and are natural extensions of the differential $\partial$ and Donaldson product in Lagrangian Floer homology, which count pseudoholomorphic bigons and triangles, respectively.  The main difference in the analytical theory when $n \geq 3$ is that the complex structure on the domain of the pseudoholomorphic maps is allowed to vary,  parametrized by the space of disks with $n+1$ distinct points on the boundary modulo complex isomorphism.  This parameter space is identified with $\M_{0,n+1}^\R$, a component of the real locus of the Deligne-Mumford space $\M_{0,n+1}$.  Its stable compactification realizes the Stasheff  associahedron $K_n$, and this is where the intrinsic \ainfty structure comes from. 

In this paper we consider a similar set-up that uses holomorphic {\em quilts}, introduced by Wehrheim and Woodward \cite{ww1, ww_1a}.  For $d \geq 1$, we construct families of quilts parametrized by moduli spaces of pointed {quilted disks}, and study moduli spaces of pseudoholomorphic quilted disks. A quilted disk with $d+1$ marked points on the boundary is a configuration $(D, C, z_0, \ldots, z_d)$ where $D$ is the unit disk in $\C$, $(z_0, \ldots, z_d)$ is a cyclically ordered tuple of distinct points in $\partial D$, and $C \subset D$ is an inner circle such that $C \cap \partial D = \{z_0\}$. The moduli space of $(d+1)$-pointed quilted disks, denoted $\RR^{d,0}$, consists of such tuples modulo complex automorphisms of $D$. These moduli spaces were studied in \cite{multiplihedra}\footnote{In \cite{multiplihedra} the moduli spaces $\RR^{d,0}$ were called $M_{d,1}$.}, and their Grothendieck-Knudsen type compactifications realize the {\em multiplihedra}, a family of polytopes that appeared alongside the associahedra in the work of Stasheff in the '60s \cite{stasheff}.  The multiplihedra are intrinsically connected to \ainfty maps between \ainfty spaces, just as the associahedra are intrinsically connected to \ainfty spaces. 
\begin{figure}[ht]
\center{\includegraphics[width=2.7in,height=2.3in]{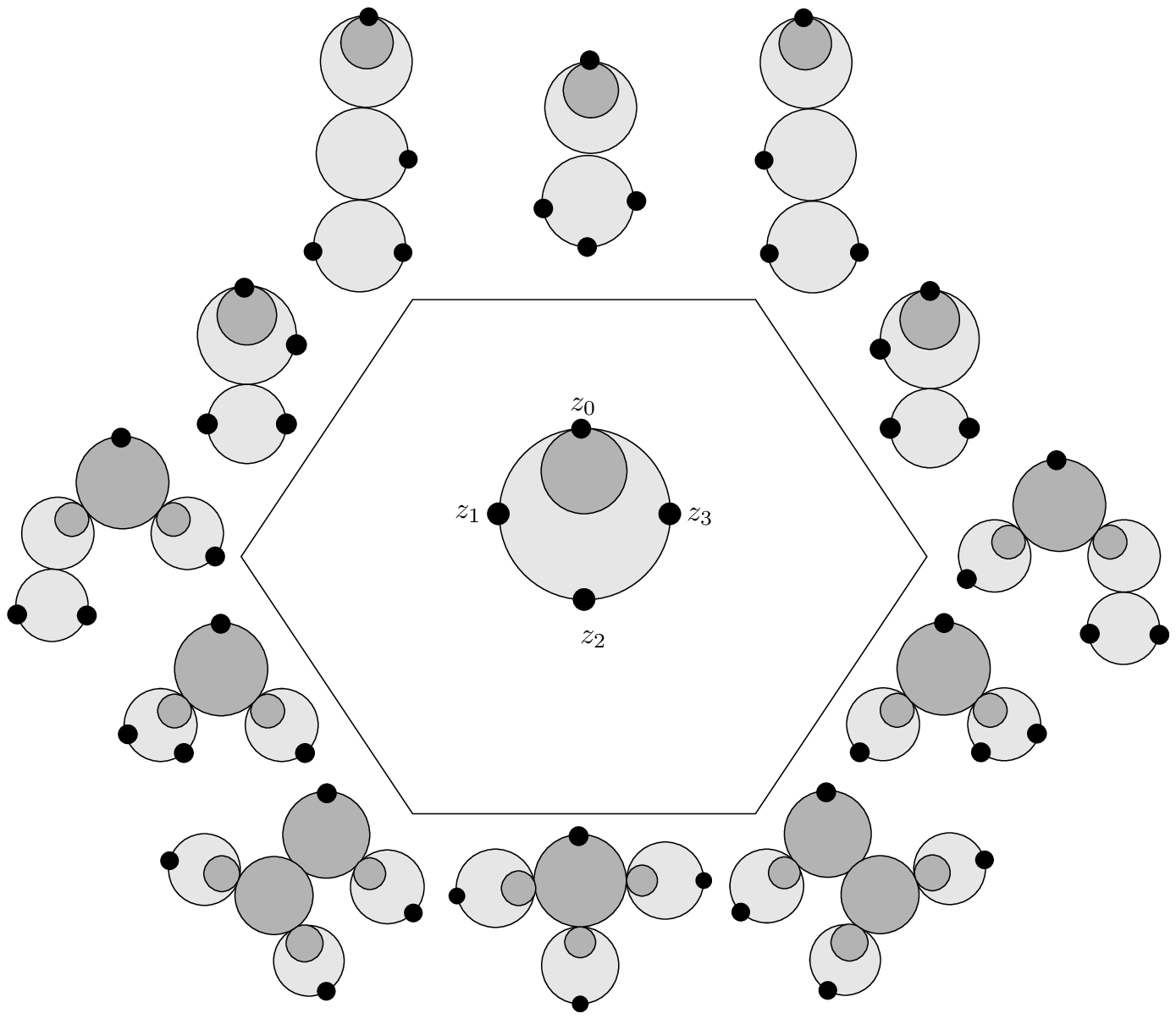}}
\caption{The third multiplihedron, $J_3 = \ol{\RR}^{3,0}$.}\label{mult_fig}
\end{figure}

\subsection{Functorial motivation} \label{motivation}  
A Lagrangian correspondence from $(M, \omega_M)$ to $(N, \omega_N)$ is a Lagrangian submanifold $L$ of the product symplectic manifold $M^-\times N:= (M\times N, (-\omega_M, \omega_N))$. Lagrangian correspondences are a notion of morphism between symplectic manifolds, so one would hope for topological invariants associated to symplectic manifolds to exhibit some sort of functoriality with respect to them.  Lagrangian correspondences aren't morphisms ``on the nose", since they can't always be composed\footnote{There is a geometric notion of composition of Lagrangian correspondences, following Weinstein, but the resulting object may well be singular.}. Following \cite{ww1}, a {\em generalized} Lagrangian correspondence from $M_A$ to $M_B$ is a sequence of Lagrangian correspondences, that starts with $M_A$ and ends in $M_B$ (passing through any number of intermediate symplectic manifolds in between).  A generalized Lagrangian correspondence is represented as a directed sequence
\[
\ul{L} = \{(M_A, \omega_A) \overset{L_{01}}{\longrightarrow} (M_1, \omega_1) \overset{L_{12}}{\longrightarrow} (M_2, \omega_2) \overset{L_{23}}{\longrightarrow} \ldots (M_r, \omega_r)\overset{L_{r, r+1}}{\longrightarrow} (M_B, \omega_B)\}.
\] 
{Generalized} Lagrangian correspondences can be composed by concatenation.  We denote concatenation by $\#$, i.e. $\ul{L}\# \ul{L}^\prime := M_A \overset{\ul{L}}{\longrightarrow} M_B \overset{\ul{L}^\prime}{\longrightarrow} M_C$.

Quilted Lagrangian Floer theory allows one to define Floer homology $HF(\ul{L}, \ul{L}^\prime)$ for pairs $\ul{L}, \ul{L}^\prime$ of admissible generalized Lagrangian correspondences from $M_A$ to $M_B$ \cite{ww1,ww_1b}.  Thus a pair  of symplectic manifolds $(M_A, \omega_A), (M_B, \omega_B)$ has an associated enlarged Donaldson-Fukaya category, $\Don^\#(M_A,M_B)$, whose objects are sequences of admissible Lagrangian correspondences from $M_A$ to $M_B$, and $\Hom(\ul{L}_0, \ul{L}_1) := HF(\ul{L}_0, \ul{L}_1)$.  When $M_A=\{pt\}$, i.e., the zero dimensional symplectic manifold, we say that a generalized Lagrangian correspondence from $\{pt\}$ to $M_B$ is a {\em generalized Lagrangian submanifold} of $M_B$.\footnote{An {actual} Lagrangian submanifold is a sequence with exactly one arrow,
$L = \{pt \overset{L}{\longrightarrow} (M_B, \omega_B)\}.$}
We write $\Don^\#(M_B) := \Don^\#(pt, M_B)$ for the corresponding enlarged Donaldson-Fukaya category. 

It was shown in \cite{ww1} that a generalized Lagrangian correspondence between $M_A$ and $M_B$ determines a functor
 \[
 \Phi(\ul{L}_{AB}): \Don^\#(M_A) \to \Don^\#(M_B),
 \] 
and furthermore that the assignment $\ul{L}_{AB} \to \Phi(\ul{L}_{AB})$ extends to a functor 
\[
\Phi: \Don^\#(M_A, M_B) \to \Fun(\Don^\#(M_A), \Don^\#(M_B)),
\]
where $\Fun(\Don^\#(M_A), \Don^\#(M_B))$ is the category of functors from $\Don^\#(M_A)$ to $\Don^\#(M_B)$.  The eventual goal is to extend both these functors to the chain level; that is, with the corresponding enlarged Fukaya categories. 

The enlarged Fukaya category $\Fuk^\#(M_A, M_B)$ is a generalization of the Fukaya category.  Its objects are admissible generalized Lagrangian correspondences from $M_A$ to $M_B$, and morphisms between objects are elements of Floer chain groups.  Higher compositions 
\[
\mu^n: CF(\ul{L}_{n-1}, \ul{L}_n)\otimes \ldots \otimes CF(\ul{L}_0, \ul{L}_1) \longrightarrow CF(\ul{L}_0, \ul{L}_n)
\]
are defined by counting isolated  pseudoholomorphic generalized $(n+1)$-gons.  The quilted domains behind the $\mu^n$ are obtained from the $(n+1)$-pointed disks of the usual Fukaya category by ``doubling then attaching strips".   Given an $(n+1)$-pointed disk, we first take an identical copy of it, then attach quilted strips of a fixed width joining corresponding boundary components.  The number of components in the quilted strips depends on the number of intermediate manifolds in the generalized Lagrangian correspondences being used (Figure \ref{quilted_gon}).  

\begin{figure}
\center{\includegraphics[height=1in]{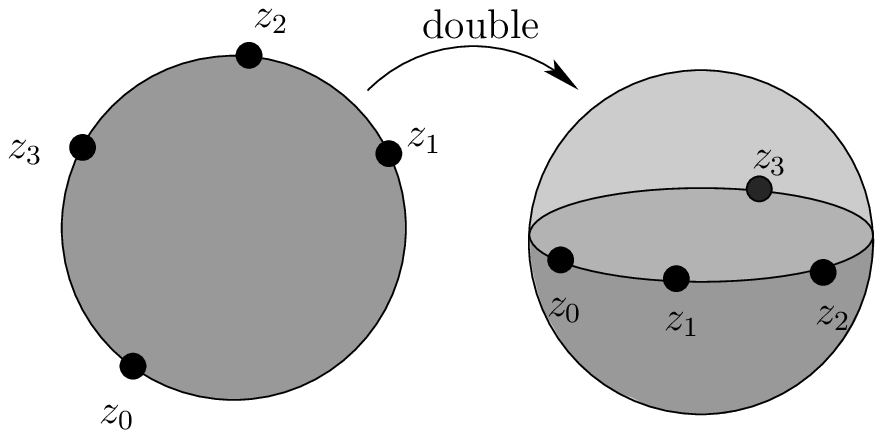} \includegraphics[height=1.5in]{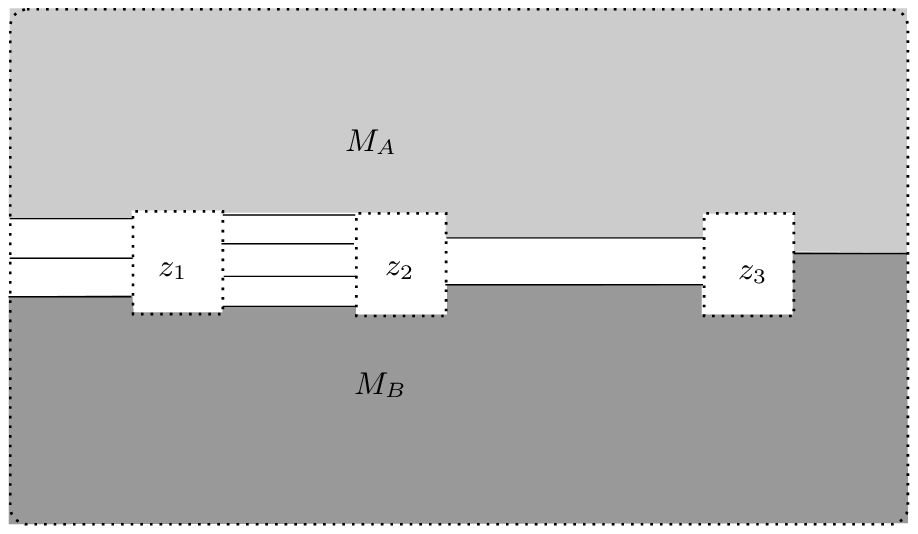}}
\caption{Schematic representation of a generalized $3+1$-gon.  The case drawn is for $\ul{L}_0$ having 3 arrows (i.e., 2 intermediate symplectic manifolds in the sequence), $\ul{L}_1$ 4 arrows, $\ul{L}_2$ two arrows and $\ul{L}_3$ one arrow.  The big (outer) dotted square corresponds to the marked point $z_0$ at $\infty$.}\label{quilted_gon}
\end{figure}

For a fixed collection of generalized Lagrangian correspondences, the complex structures on the components of the quilted $(n+1)$-gon depend only on the complex structure on the $(n+1)$-pointed disk.  Thus the quilted domains behind the maps $\mu^n$ are still parametrized by the moduli space $\M_{0,n}^\R$.  The gluing result of this paper (Theorem \ref{main_theorem}) applies with only cosmetic changes to isolated regular pseudoholomorphic quilted polygons, which is an analytical ingredient behind proving that the compositions $\mu^n$ in the enlarged Fukaya category satisfy the \ainfty associativity relations,  In particular, $\Fuk^\#(M_A, M_B)$ has the structure of a non-unital \ainfty category.  In the special case that $M_A = \{ pt\} $, we write $\Fuk^\#(M_B):= \Fuk^\#(pt,M_B)$.  


Given a Lagrangian correspondence $L_{AB}$ from $M_A$ to $M_B$, there is an associated non-unital \ainfty functor $\Phi(L_{AB}): \Fuk^\#(M_A) \to \Fuk^\#(M_B)$.  We give a very rough sketch of the construction here, referring to \cite{ainfty_functors} for details. 

Recall the rather combinatorial definition of a non-unital \ainfty functor \cite[Chapter I, (1b)]{seidel-book}:

\begin{definition*} A {\em non-unital $A_\infty$ functor} between non-unital \ainfty categories $\A$ and $\B$ consists of a map $\Phi: Ob\ \A \to Ob\ \B$ on objects, together with a sequence of maps $\Phi^d, d \geq 1$,
\[
\Phi^d: \Hom_\A(X_0, X_1) \otimes \ldots \otimes \Hom_\A(X_{d-1},X_d) \lra \Hom_\B(\Phi(X_0), \Phi(X_d))[1-d],  
\]
where the $\Phi^d$ fit together with the $\mu_\A^j, \mu_\B^k$ to satisfy the {\em $A_\infty$ functor relations}:
\begin{eqnarray}\label{ainfty_functor_relations}
\sum\limits_{i,j} (-1)^* \Phi^e(a_d, \ldots, a_{i+j+1}, \mu_\A^j(a_{j+i}, \ldots, a_{i+1}), a_i, \ldots, a_1) = \hspace{1in} \\
\sum\limits_{r, i_1+\ldots i_r=d} \mu_\B^r(\Phi^{i_r}(a_d, \ldots, a_{d-i_r}), \ldots, \Phi^{i_1}(a_{i_1}, \ldots, a_1)) \nonumber, 
\end{eqnarray}
where $* = |a_1| + \ldots + |a_i| - i$.
\end{definition*}

Here, $\A = \Fuk^\#(M_A)$ and $\B = \Fuk^\#(M_B)$.  Define $\Phi(L_{AB})$ as follows. On objects, 
\bea
\Phi(L_{AB}): Ob \ \Fuk^\#(M_A)  \longrightarrow  Ob \ \Fuk^\#(M_B), \ \ \ \ul{L}   \mapsto   \ul{L}\# L_{AB},
\eea
i.e., concatenate $L_{AB}$ to each generalized Lagrangian submanifold of $M_A$.  Now suppose that $d \geq 1$, and $\ul{L}_0, \ul{L}_1, \ldots, \ul{L}_d$ are generalized Lagrangian submanifolds of $M_A$.  Abbreviate $\ul{L}_{0,AB}:= \ul{L}_0\# L_{AB}$, likewise $\ul{L}_{d,AB}$.  Define the multilinear maps
\bea
\Phi(L_{AB})^d : CF(\ul{L}_{d-1}, \ul{L}_d)\otimes \ldots \otimes CF(\ul{L}_0, \ul{L}_d) & \longrightarrow &  CF(\ul{L}_0\# L_{AB}, \ul{L}_d \# L_{AB})\\
(\ul{p}_d,\ldots,\ul{p}_1) & \mapsto & \sum\limits_{\ul{q} \in \I(\ul{L}_{0,AB},\ul{L}_{d,AB})} \ N_{\ul{q}} \langle \ul{q}\rangle,
\eea
where the coefficient $N_{\ul{q}}$ is a signed count of a 0-dimensional moduli space of {\em pseudoholomorphic quilted disks}, which are pairs $(r,\ul{u})$ where $r \in \RR^{d,0}$ parametrizes a quilted domain $\sS_{r}$, and $\ul{u}$ is a pseudoholomorphic quilt map from $\sS_{r}$, with a seam condition given by $L_{AB}$ (Figure \ref{quilt_disk_mod}).

\begin{figure}[ht]
\center{\includegraphics[height=1.2in]{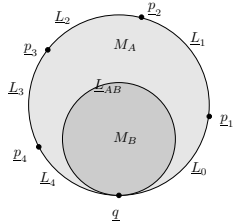}}
\caption{A quilted disk with boundary conditions.}
\label{quilt_disk_mod}
\end{figure}

The codimension one boundary strata of the multiplihedra come in two distinct types. The first type of facet is a pointed quilted disk joined to a pointed disk; call this a {\em facet of Type 1}.  The second type is a pointed disk with several pointed quilted disks; call this a {\em facet of Type 2}.  These facets are closely related to the \ainfty functor relations, which have four distinct kinds of terms, corresponding in this construction to: a) broken tuples whose domains correspond to facets of Type 1, b) broken tuples whose domains correspond to facets of Type 2, and c) either incoming or outgoing broken Floer trajectories (Figure \ref{multi_facets} ).  \begin{figure}[h]
\center{\includegraphics[height=1in]{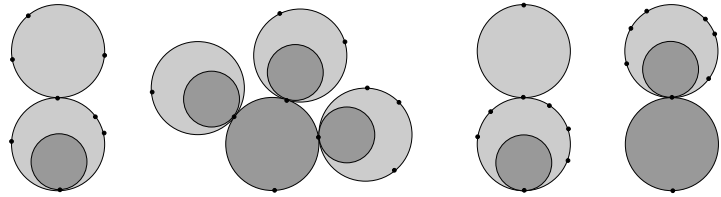}}
\caption{Four cases of gluing, from left to right: facet of Type 1,  facet of Type 2, incoming Floer trajectory, outgoing Floer trajectory.}\label{multi_facets}
\end{figure}

The purpose of this paper is to prove the following statement. 

\begin{theorem1}\label{main_theorem}
Let $(M_A, \omega_A)$ and $(M_B, \omega_B)$ be compact symplectic manifolds, and $L_{AB}$ a Lagrangian correspondence from $M_A$ to $M_B$. Suppose that $\ul{L}_0, \ldots, \ul{L}_d$ are generalized Lagrangian submanifolds of $M_A$, all of whose intermediate symplectic manifolds are  compact.  Write $\ul{L}_{0,AB}:= (\ul{L}_0, L_{AB})$, $\ul{L}_{d,AB}:= (\ul{L}_d, L_{AB})$ for the generalized Lagrangian submanifolds of $M_B$ obtained by concatenating $L_{AB}$ to $\ul{L}_0$ and $\ul{L}_d$.  Let $\ul{x}_0 \in \I(\ul{L}_{0,AB}, \ul{L}_{d,AB})$ and for $i = 1, \ldots, d$ let $\ul{x}_i \in \I(\ul{L}_{i-1},\ul{L}_i)$.  Given either:
\begin{enumerate} 
\item a regular pair  
\bea
(r_1, \ulu_1) & \in & \M_{d-e+1,0}(\ul{x}_0, \ul{x_1}\ldots, \ul{x}_{i-1}, \ul{y}, \ul{x}_{i + e+1}, \ldots, \ul{x}_d)^0\\
(r_2, \ulu_2) & \in & \M_{e}(\ul{y}, \ul{x}_i, \ul{x}_{i+1},\ldots, \ul{x}_{i+e})^0
\eea
where $2 \leq e \leq d$, $1 \leq i \leq d-e$, and $\ul{y} \in \I(\ul{L}_{i-1}, \ul{L}_{i+e})$;
\item or a regular $(k+1)$-tuple 
\bea
(r_0, \ulu_0) & \in & \M_{k}( \ul{x_0}, \ul{y}_1, \ldots, \ul{y}_k)^0\\
(r_1, \ulu_1)  & \in & \M_{d_1, 0}(\ul{y}_1, \ul{x}_1, \ldots, \ul{x}_{d_1})^0\\
(r_2, \ulu_2) & \in & \M_{d_2, 0}(\ul{y}_2, \ul{x}_{d_1 + 1}, \ldots, \ul{x}_{d_1 + d_2})^0\\
\ldots & & \\ 
(r_k, \ulu_k) & \in & \M_{d_k, 0}(\ul{y}_k, \ul{x}_{d_1 + \ldots + d_{(k-1)} +1}, \ldots, \ul{x}_{d_1+ \ldots + d_{k-1} +d_k})^0
\eea
where $d_1 + \ldots + d_k = d$, $d_i \geq 1$ for each $i$, and $\ul{y}_i \in \I(\ul{L}_ {d_1 + \ldots + d_{(i-1)}}, \ul{L}_{d_1 + \ldots + d_{i}})$ (interpreting $d_0$ as 0); 
\item or a regular pair 
\bea
(r,\ulu) & \in & \M_{d,0}(\ul{x}_0, \ldots, \ul{x}_{i-1}, \ul{y}, \ul{x}_{i+1}, \ldots, \ul{x}_d)^0\\
\ul{v} & \in & \widetilde{\M}_{1}(\ul{y}, \ul{x}_i)^0
\eea
 where $ 1 \leq i \leq d$, and $\ul{y} \in \I(\ul{L}_{i-1}, \ul{L}_i)$;
 \item or a regular pair 
 \bea
 \ul{v} & \in & \widetilde{\M}_{1}(\ul{x}_0, \ul{y})^0\\
(r, \ulu) & \in & \M_{d,0}(\ul{y}, \ul{x}_1, \ldots, \ul{x}_d)^0
\eea
where $\ul{y} \in \I(\ul{L}_{0, AB}, \ul{L}_{d, AB})$,
 \end{enumerate}
there is an associated continuous gluing map 
\[
g: (R_0, \infty) \to \M_{d,0}(\ul{x}_0, \ldots, \ul{x}_d)^1
\] 
defined for some $R_0 >>0$, such that $g(R)$ Gromov converges to the given pair/tuple as $R \to \infty$.  Moreover, the gluing map surjects onto sufficiently small Gromov neighborhoods of the given broken pairs/tuples. 
\end{theorem1}

\subsection{Organization}  Section \ref{quilt_disks} covers preliminaries on quilts, holomorphic quilt maps, and the multiplihedra.  Section \ref{families} constructs the families of quilted surfaces that we will use.  For analytic simplicity we require all our quilted surfaces to have striplike ends, and we emphasise that:
\begin{center}
\begin{tabular}{| c |}
\hline
\\
In this paper, all strips are the standard complex strip, \\
$Z \cong \{z \in \C | 0 \leq \Im z \leq 1\}$. \\
\\
\hline
\end{tabular}
\end{center}
Thus the analysis along the striplike ends reduces to the usual theory of pseudoholomorphic strips, since a quilted strip consisting of strips of equal width can be ``folded'' into an ordinary strip with boundary conditions in a big product manifold (Figure \ref{folding_strip}). 
\begin{figure}[h]
\includegraphics[height=4.3in]{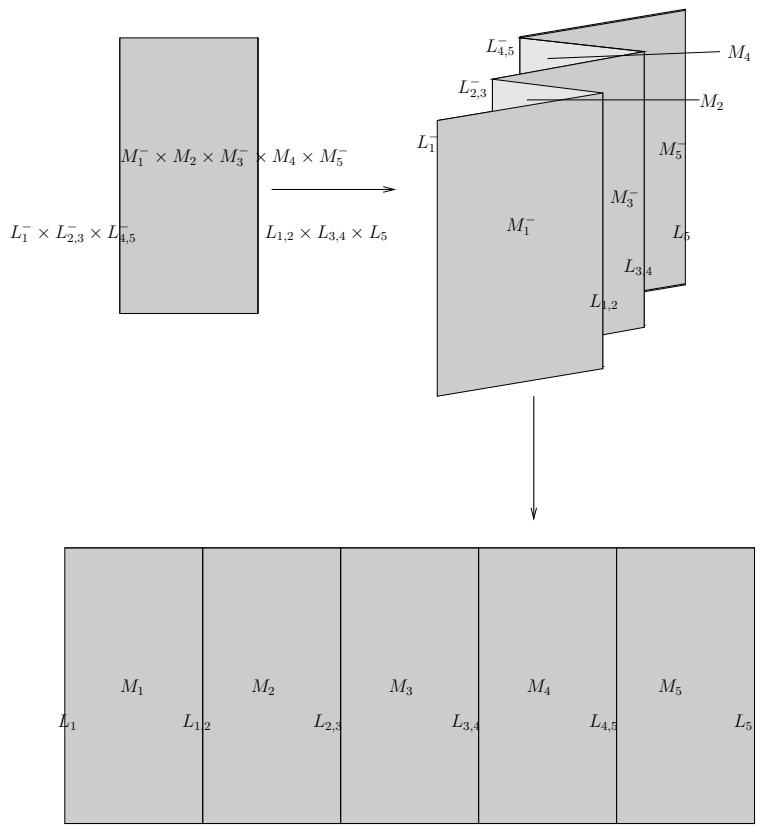}
\caption{Unfolding an ordinary strip into a quilted strip.}\label{folding_strip}
\end{figure}

A somewhat {\em ad hoc} aspect of the construction comes from the fact that if one simply takes the canonical bundle of quilted disks over the multiplihedron, there is no natural quilted striplike end over the marked point $z_0$, so we will manually construct the bundles $\sS^{d,0} \lra \RR^{d,0}$ of quilted surfaces parametrized by the multiplihedra, with quilted striplike ends built in.  

Section \ref{section_j-holo_quilts} defines the moduli spaces of pseudoholomorphic pointed quilted disks.  The moduli spaces incorporate parameters coming from the finite dimensional spaces $\RR^{d,0}$ parametrizing the domains.  The framework is based on \cite{seidel-book} Part II, sections 7 through 9, which defines pseudoholomorphic polygons for the Fukaya category.  We fix a universal choice of perturbation data for all compactified parameter spaces $\ol{\RR}^{d,0}$ at once, with recursive compatibility properties with respect to boundary strata, and study solutions of the inhomogeneous pseudoholomorphic map equations in an appropriate Banach bundle setting.  We define local trivializations in which the moduli space of holomorphic quilted disks is modeled on zeros of a Fredholm map $\F_{\sS, r, \ul{u}}$, and write out the linearized operator $D_{\sS, r, \ul{u}}$ explicitly.   

Section \ref{gluing_section} covers the estimates needed to prove Theorem \ref{main_theorem}.  The gluing map is an application of the Implicit Function Theorem, and the strategy of proof is standard, though the details are lengthy.  

Appendices A and B are included for completeness.   Appendix A establishes Sobolev embedding estimates that are relevant to the quadratic estimate (Section 5.5).  In order to obtain estimates that are independent of the gluing parameter, we need to establish Sobolev embedding constants that are uniform for all domains in the parametrized family;  that is, constants that do not vary with the parameters coming from the multiplihedra or associahedra.  The main  ingredient behind the uniform embedding constants is the fact that the non-compact striplike ends of all domains in the families satisfy the same ``cone condition".  Appendix B contains explicit convexity estimates for the solutions of Floer's inhomogeneous equation near non-degenerate intersections.   

\subsection*{Acknowledgements}  Most of the work in this paper was part of my Ph.D. thesis, carried out at Rutgers University under the supervision of Chris Woodward.  This paper owes a great deal to his ideas and advice.  I also thank Katrin Wehrheim for enormously helpful discussions on analysis and quilts.       

\section{Preliminaries}\label{quilt_disks}

\subsection{Quilts} We review the basic set-up of quilts, following \cite{ww_1a}.
\begin{definition}
A {\em quilted surface} $\underline{S}$ with strip-like ends  consists of the following data:
\begin{enumerate}
\item A collection $\underline{S} = (S_\alpha)_{\alpha \in \A}$, where each $S_\alpha$ is a pointed Riemann surface with markings $\zeta^\alpha_1, \ldots, \zeta^\alpha_{n_\alpha} \in \partial S_\alpha$, and for each marking $\zeta$ a {\em striplike end}, which is a holomorphic embedding $\epsilon_\zeta: \R_{\geq 0} \times [0,1] \to S$ such that $ \epsilon_\zeta(s,0), \epsilon_{\zeta}(s,1) \in \partial S$ for all $s\in \R_{\geq 0}$, and $\lim_{s\to \infty}\epsilon(s,t) \to \zeta$ uniformly in $t$.  Write $\mathcal{E}(S_\alpha)$ for the set of boundary components of $S_\alpha\setminus \{\zeta^\alpha_1, \ldots, \zeta^\alpha_{n_\alpha}\}$. In general the boundary components in $ \mathcal{E}(S_\alpha)$ are diffeomorphic to $\R$ or $S^1$, depending on whether they are incident to any marked points or not.
\item A collection $\mC$ of {\em seams}, whose elements are tuples 
\[
\mC \ni \sigma = \{(\alpha, I), (\alpha^\prime, I^\prime), \varphi_{(I, I^\prime)}\}
\]
where $\alpha, \alpha^\prime \in \A$, and $I \in \cE(S_\alpha), I^\prime \in \cE(S_{\alpha^\prime})$, $I \neq I^\prime$, and $\varphi_{(I,I^\prime)}: I \overset{\sim}{\longrightarrow} I^\prime$ is a real-analytic identification of the boundary components. If $I$ and $I^\prime$ are noncompact, each of them intersects a striplike end at both ends -- in this case we require that $\phi_{(I,I^\prime)}$ be {\em compatible} with the strip-like ends: that is, if $\epsilon_\zeta(s,0) \subset I$ then there is a marked point $\zeta^\prime$ with $\epsilon_{\zeta^\prime}(s,1) \subset I^\prime$ and $\phi_{(I,I^\prime)}(\epsilon_\zeta(s,0)) = \epsilon_{\zeta^\prime}(s,1)$. Similarly if $\zeta$ is a marked point on $\partial S_\alpha$ with $\epsilon_\zeta(s,1) \subset I$, then there is a marked point $\zeta^\prime \in \partial S_{\alpha^\prime}$ such that $\epsilon_{\zeta^\prime}(s,0)\subset I^\prime$ and $\phi_{(I,I^\prime)}(\epsilon_{\zeta}(s,1)) = \epsilon_{\zeta^\prime}(s,0)$. 

\item Boundary components that do not appear in any seam 
are called {\em true boundary components} of $\ul{S}$.
\end{enumerate}

\end{definition}


\subsection{Quilt maps} Label each surface $S_\alpha \in \ul{S}$ with a symplectic manifold $(M_\alpha, \omega_\alpha)$, each seam $\sigma = \{(\alpha, I), (\beta, I^\prime)\}$ with a Lagrangian correspondence $L_\sigma \subset M_\alpha^-\times M_\beta$, and each true boundary component $(\alpha, I)$ with a Lagrangian submanifold $L_{(\alpha, I)} \subset M_\alpha$.  Denote the labels collectively by $\ul{M} = (M_\alpha)_{\alpha \in \A}, \ul{L} = (L_\sigma)_{\sigma \in \mC}\cup (L_{(\alpha, I)})_{\alpha \in \A, I \in \cE(S_\alpha)}$.

\begin{definition}\label{def_quiltmap}
A {\em (smooth) quilt map $\ul{u} : \ul{S} \to \ul{M}$ with boundary conditions in $\ul{L}$} is a tuple 
$\{u_\alpha\}_{\alpha\in \A}$ of smooth maps $u_\alpha: S_\alpha \to M_\alpha$ satisfying:
\begin{enumerate}

\item for each true boundary component $(\alpha, I)$, $u_\alpha\big\lvert_{I} \subset L_{(\alpha, I)}$;
\item for each seam $\sigma = \{(\alpha, I), (\beta, I^\prime)\}$, 
$u_\alpha(\phi_I^{-1}(s,0)) \times u_\beta(\phi_{I^\prime}^{-1}(s,0)) \subset L_{\sigma}.$

\end{enumerate}
Let $C^\infty(\ul{S}, \ul{M};\ul{L})$ be the set of smooth quilt maps from $\ul{S}$ into target manifolds $\ul{M}$ with boundary conditions in $\ul{L}$.
\end{definition}

For each $\alpha \in \A$, let $J_\alpha: S_\alpha \to \mathcal{J}(M_\alpha, \omega_\alpha)$ be a choice of $\omega_\alpha$-compatible almost complex structure $J_\alpha(z)$ varying smoothly with $z \in S_\alpha$.  Along a striplike end $\epsilon_\zeta(s,t)$, $(s,t) \in [0,\infty) \times [0,1]$, we take $J_\alpha(\epsilon_\zeta(s,t)) =: J_\alpha(t)$, i.e., to be independent of $s$.  We write $j_\alpha$ for the complex structure on $S_\alpha$. 
 
 \begin{definition}
 A {\em pseudoholomorphic quilt map} $\ul{u}: \ul{S} \to \ul{M}$ with boundary in $\ul{L}$ is a smooth quilt map with boundary in $\ul{L}$ such that for each $\alpha$, $J_\alpha(z, u_\alpha(z))\circ du_\alpha(z) = du_\alpha \circ j_\alpha(z)$. 
 \end{definition}

\subsection{Multiplihedra} We describe two realizations of the multiplihedra following \cite{multiplihedra}.
\subsubsection*{Moduli of quilted disks} A pointed quilted disk is a tuple $(D, C, z_0, \ldots, z_d)$, where $D$ is the closed unit disk in $\C$, $(z_0, \ldots, z_d)$ is an ordered configuration of points in $\partial D$ (compatible with the orientation of $\partial D$), and $C \subset D$ is a circle that is tangent to $\partial D$ at the distinguished point $z_0$.  The group $PSL(\R)$ of complex automorphisms of $D$ acts freely and properly on the space of pointed quilted disks, and the moduli space of pointed quilted disks is the quotient by this action, which we denote by $\RR^{d,0}$. The space $\RR^{d,0}$ is compactified by {\em stable nodal quilted disks}.  The combinatorial type of a stable nodal quilted disk is a stable colored, rooted ribbon tree:
\begin{definition}
A {\em colored, rooted ribbon tree} $T=(E_\infty(T), E(T),
V(T), V_{col}(T))$ is a tree with vertices $V(T)$, a collection of semi-infinite edges $E_\infty(T) = \{e_0, e_1, \ldots, e_n\}$, and a collection of finite edges $E(T)$, together with a distinguished subset $V_{col}(T) \subset V(T)$
of {\em colored vertices}, such that
\begin{enumerate}
\item Each semi-infinite edge in $E_\infty(T)$ is incident to a single vertex in $V(T)$,
\item Each finite edge in $E(T)$ is incident to exactly two vertices in $V(T)$,
\item the ribbon structure at every vertex ensures that the planar structure of $T$ has  
the semi-infinite edges $e_0, \ldots, e_n$ arranged in counter-clockwise order; the edge $e_0$ is called the {\em root}, and $e_1, \ldots, e_n$ are called the {\em leaves};
\item each non-self-crossing path from a leaf $e_i$ to the root $e_0$ contains exactly one colored vertex.
\end{enumerate}
The colored ribbon tree is {\em stable} if every uncolored vertex has valency $\geq 3$, and each colored vertex has valency $\geq 2$.
\end{definition}

\begin{definition} 
A {\em nodal $(d+1)$-quilted disk} $S$ is a collection of quilted and
unquilted marked disks, identified at pairs of points on the boundary.
The combinatorial type of $S$ is a colored rooted ribbon tree ${T}$, 
where the colored vertices represent quilted disks, and the remaining
vertices represent unquilted disks.
A nodal quilted disk is {\em stable} if and only if
\begin{enumerate}
\item Each quilted disk component contains at least $2$ singular or
marked points;
\item Each unquilted disk component contains at least
$3$ singular or marked points.
\end{enumerate}
\end{definition} 

\begin{figure}[ht]
\includegraphics[height=2in]{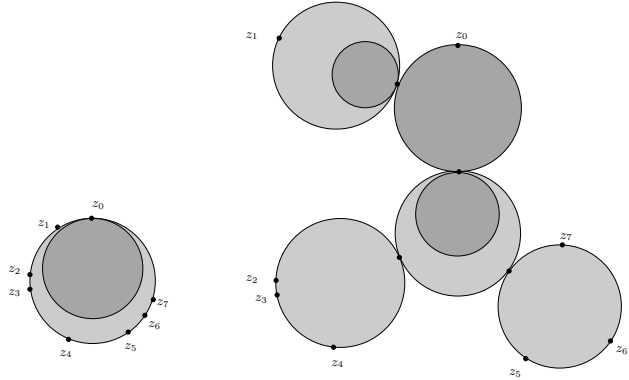}
\caption{A pointed quilted disk in $\RR^{7,0}$, and a nodal pointed quilted disk in the compactification $\ol{\RR}^{7,0}$.}\label{example}
\end{figure}
Let $\RR_T$ denote the stable, nodal quilted disks with combinatorial type $T$.  As a set, 
\[
\ol{\RR}^{d,0} := \bigcup\limits_{T}\RR_T,
\]
where $T$ ranges over all stable, $d$-leafed colored rooted ribbon trees.
\subsubsection*{Metric colored ribbon trees}
 There is a dual realization of the multiplihedra as metric trees, which is analogous to the realization of the associahedra as a space of metric ribbon trees.  A metric colored ribbon tree is a colored ribbon tree and a map $\lambda: E(T) \to [0,\infty]$ of {\em edge lengths} which satisfy the condition that the colored vertices are all the same distance from the root.  If $k$ is the number of colored vertices of $T$, this condition imposes $k-1$ relations on the edge lengths, and defines a cone $\cG_T \subset [0,\infty)^{|E(T)|}$ of dimension
$|E(T)|-k+1$. 

\begin{definition} A tuple of lengths is {\em admissible} if it is in the cone $\cG$.  
\end{definition}

\begin{example}
For the tree in Figure \ref{bicolor}, the edge lengths satisfy the relations
$$ \lambda_1 + \lambda_2 + \lambda_3 = \lambda_1 + \lambda_2 +
  \lambda_4 = \lambda_1 + \lambda_2 + \lambda_5 = \lambda_1 +
  \lambda_6 = \lambda_7.  $$
These are equivalent to the relations $\lambda_3 = \lambda_4 = \lambda_5, \lambda_3 + \lambda_2 = \lambda_6,$ and $\lambda_6 + \lambda_1 = \lambda_7$.
\begin{figure}[h]
\includegraphics[height=1.8in]{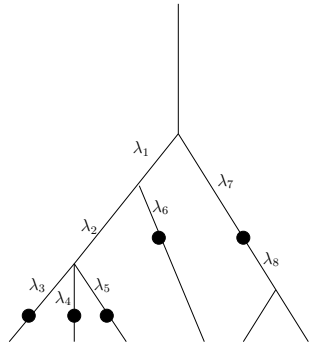}
\caption{A colored ribbon tree with interior edges labeled by gluing lengths $(\lambda_1, \ldots, \lambda_8)$.}
\label{bicolor}
\end{figure}
\end{example}
For each stable colored, rooted ribbon tree $T$, set
\bea
C_T = \{\lambda: E(T) \to [0,\infty]\big\lvert \lambda \in \cG_T\},\ \ \ \ 
\ol{C}^{d,0}  =  \left(\bigcup\limits_T C_T\right)/\sim
\eea
where the equivalence relation identifies an edge of length $0$ in $T$ with the tree obtained from $T$ by contracting that edge.  In \cite{multiplihedra} it is shown that $ \ol{C}^{d,0} \cong \ol{\RR}^{d,0} $.

\section{Families of quilts}\label{families}

\subsection{Pointed disks and associahedra}\label{quilts_assoc}
The construction of higher compositions in the Fukaya category is based on families of pointed disks.   We recall an explicit map in \cite{zero-loop} identifying the moduli space of metric ribbon trees with the moduli space of pointed disks. 

Let $d\geq 2$ and let $(T,\lambda) = (V(T), e_0, e_1,\ldots,e_d, E(T), \lambda: E(T) \to \R_+)$ be a metric ribbon tree: $V(T)$ are vertices of a tree $T$, $e_0, e_1,\ldots,e_d$ are semi-infinite (exterior) edges, and $E(T)$ are finite (interior) edges, and for each $e \in E(T)$, $\lambda(e) \geq 0$ is the length of $e$. Thicken each semi-infinite edge $e_i$ to a semi-infinite strip, $Z_i = [0,\infty) \times [0,1]$, and thicken each edge $e\in E(T)$ of length $\lambda(e)$ to a finite strip $Z_e = [0,\lambda(e)]\times [0,1]$. 
Consider the punctured strips $Z_i\setminus (0,\frac{1}{2}), i=0,\ldots,d$ and $Z_e \setminus \{  (0,\frac{1}{2}), (\lambda(e),\frac{1}{2})\}, e \in E(T)$.  Use the cyclic ribbon structure of incident edges at each vertex $v$ to define identifications of adjacent half-intervals at the ends of adjacent strips (see Figure \ref{glue}). The resulting Riemann surface with boundary has a hole for each vertex in the tree, but each hole can be conformally filled to produce a Riemann surface that is biholomorphic to a $d+1$-pointed disk, and the semi-infinite strips $Z_0, \ldots, Z_d$ determine holomorphic embeddings of striplike ends at the points.   

\begin{figure}[h]
\center{
\includegraphics[height=1in]{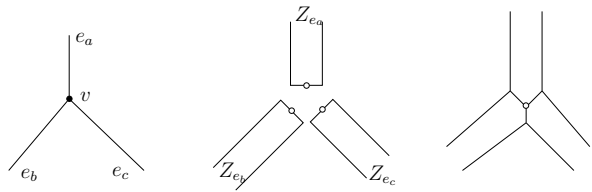}
}
\caption{Identifying strips at a trivalent vertex.}
\label{glue}
\end{figure}

\subsubsection{Attaching strips}\label{attaching_strips} 


A {\em generalized} $d+1$-pointed disk (Figure \ref{gen_pd}) is a quilted surface obtained by attaching strips to boundary components of a $d+1$-pointed disk. In the above construction, the boundary components of the Riemann surfaces are naturally identified with the boundary of an infinite strip $\R \times [0,1]$.  The identification is unique up to translation, and can be made translationally invariant for all surfaces in the family by requiring that a common boundary point (for instance, a fixed point on one of the striplike ends) always be identified with $0$. Thus one can fix a unique attaching map of a standard strip to that boundary component.  Proceeding inductively one can attach further strips to the boundary of an attached strip, without introducing any additional degrees of freedom.       
\begin{figure}[h]
\center{\includegraphics[height=2in]{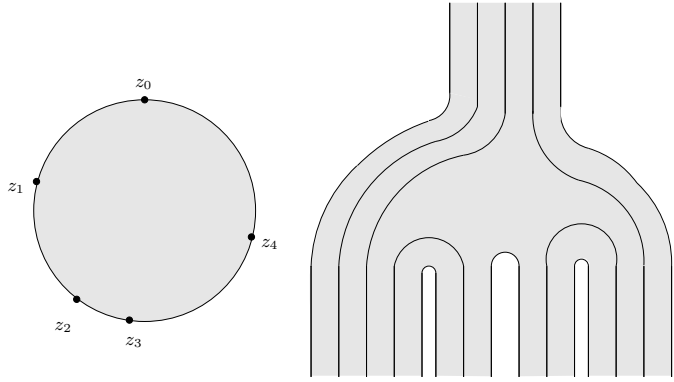}}
\caption{A quilted surface parametrized by a marked disk in $\RR^4$.}
\label{gen_pd}
\end{figure}

\subsection{Quilts parametrized by multiplihedra}\label{quilts_multi}
Let $d \geq 1$. We will construct a fiber bundle $\sS^{d, 0} \longrightarrow \RR^{d, 0}$ of quilted surfaces that are diffeomorphic to pointed quilted disks, but are equipped with strip-like ends consisting of quilted strips, such that all strips have width 1.  We will do it in two steps: first, constructing the quilted surfaces, and second, fixing the complex structures.

\subsubsection{Constructing the quilted surfaces}
 The construction is based on the fact that the multiplihedron is parametrized by metric trees with some additional metric data.  The quilted surfaces that we define will be pointed Riemann surfaces, as constructed in Section \ref{quilts_assoc}, which are parametrized by the underlying metric tree.  The extra metric data will be used to fix an embedded 1-submanifold.

For $d \geq 1$, we will define the bundle over the boundary of the multiplihedron $J_d$ in terms of lower strata, extend the bundle over a neighborhood of the boundary by gluing striplike ends, and finally fix a smooth interpolation of the bundle over the interior of the $d$-th multiplihedron.  

\noindent{\it Base step:} For $d=1$, let $Z = \R \times [0,1]$ be the standard strip of unit width.  Fix a smooth, connected embedded 1-manifold $L \subset Z$ such that 
\begin{itemize}
\item $L\cap Z_{\leq -1} = \R_{\leq 0} \times \{\frac{1}{3}, \frac{2}{3}\}$,

\item $L \cap Z_{\geq 0} = \varnothing$,

\item $L \cap \partial Z = \varnothing$.

\end{itemize}
Denote this quilted surface by $\sS^{1,0}$. As shorthand, for $x,y\in \R$ we let $\sS^{1,0}_{\leq x}, \sS^{1,0}_{\geq y}$ denote the pieces of $\sS^{1,0}$ corresponding to $Z_{\leq x}, Z_{\geq y}$ respectively. 

\begin{figure}[ht]
\center{\includegraphics[height=.7in]{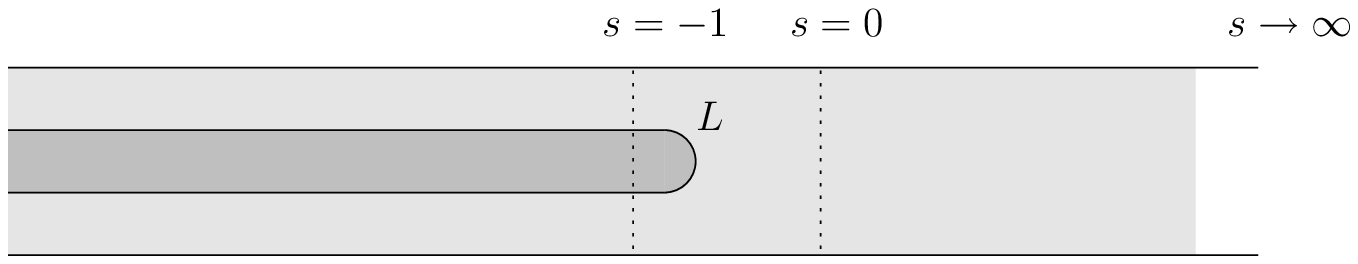}}
\caption{The basic piece $\sS^{1,0}$.}
\label{step1}
\end{figure}

\noindent{\it Inductive hypothesis:} Assume that all bundles $\sS^{e,0} \to \ol{\RR}^{e,0}$ have been constructed for $1 \leq e < d$. Assume furthermore that the bundles $\sS^{e,0} \to \RR^{e,0}$ are compatible with the bundles $\sS^e \to \RR^{e}$ constructed in the previous section, i.e., there is a commuting diagram \\
\begin{equation}\label{commute}
\begin{CD}
\sS^{e,0} @>>> \RR^{e,0}\\
         @VVV    @VVV\\
\sS^e @>>> \RR^e         
\end{CD}
\end{equation}
where the vertical maps are forgetful maps.  That is, $\RR^{e,0}\to \RR^e$ is the map from stable metric colored trees to stable metric trees that forgets the additional colored structure, and $\sS^{e,0} \to \sS^e$ is the map from a quilted pointed Riemann surface to an unquilted pointed Riemann surface, that forgets the embedded submanifold $L$.  

\noindent{\it Inductive step:} We extend the bundle $\ol{\sS}^d \lra \ol{\cR}^d$  over an open neighborhood of $\partial\ol{\RR}^d$, as follows.  The codimension one facets of $\RR^d$ are indexed by two types of colored tree (Figure \ref{facets}). Facets of Type 1 are denoted by $F_{(e,i)}$ where the label $(e,i)$ is such that $0\leq i \leq d$, $2\leq e \leq d-i$.  We denote the indexing tree by $T_{(e,i)}$.  Facets of Type 2 are denoted by $F_{s_1+s_2+\ldots+s_r}$, where the labels $s_1+ \ldots+ s_r$ satisfy $r\geq 2$, $1 \leq s_i \leq d-1$, $s_1 + s_2 + \ldots s_r =d$. We denote the indexing tree by $T_{s_1+\ldots+s_r}$. 
\begin{figure}[h]
\center{\includegraphics[height=1in]{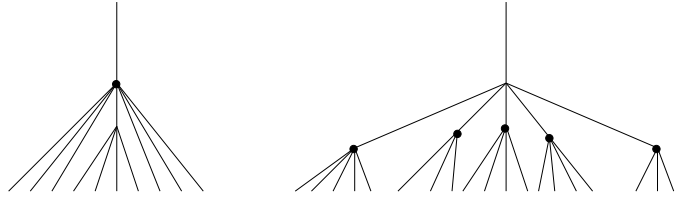}}
\caption{Combinatorial types of facets of the multiplihedron: Type 1 (left) and Type 2 (right).}
\label{facets}
\end{figure}

\noindent{\it Type 1:} Consider a boundary facet $F_{(e,i)} \subset \partial \ol{\RR}^{d,0}$; denote the single edge of $T_{(e,i)}$ by $a$.   Every point in $\ol{\RR}^{d,0}$ is represented by a metric colored  tree, $(T,\lambda)$, and such a metric colored tree is in $F_{(e,i)}$ if and only if 
\begin{enumerate}
\item there exists a morphism $p: T\to T_{(e,i)}$ of trees, 
\item  for the unique edge $p^{-1}(a) \in T$, $\lambda(p^{-1}(a)) = \infty$. 
\end{enumerate}

Cutting the edge $p^{-1}(a)$ produces a pair 
$(T_1,\lambda_{1}) \in \RR^{d-e+1,0}$, and $(T_2, \lambda_{2})\in \RR^{e}$, where $\lambda_i := \lambda\lvert_{T_i}$.  Note that the edge $a$ of $T$ has become the leaf $e_{i+1}$ of $T_1$ and the root $e_0$ of $T_2$. Let $\sS_{(T_1,\lambda_1)}$ and $\sS_{(T_2,\lambda_2)}$ be the corresponding surfaces, constructed by inductive hypothesis.  

Let $\nu := \lambda(p^{-1}(a)) \in [0,\infty]$ be a normal coordinate to the facet $F_{(e,i)}$, i.e., the length of the edge $p^{-1}(a)$. For $\nu \geq 1$, truncate the strips $Z_{e_i}$ and $Z_{e_0}$ at length $\nu/2$ and identify the two strips along the truncations; this replaces the semi-infinite strips $Z_{e_i}$ and $Z_{e_0}$ with a finite strip of length $\nu$. 

\noindent{\it Type 2:} Now consider a boundary facet $F_{s_1+\ldots+s_r}\subset \partial \ol{\RR}^{d,0}$; denote the $r$ edges of $T_{s_1+\ldots+s_r}$ by $a_1, \ldots, a_r$.  A metric colored tree $(T,\lambda)$ is in $F_{s_1+\ldots s_r}$ if and only if 
\begin{enumerate}
\item there exists a morphism $p: T\to T_{s_1+\ldots s_r}$ of trees, 
\item  for $i=1,\ldots,r$, $\lambda(p^{-1}(a_i)) = \infty$. 
\end{enumerate}
Cutting the edges $p^{-1}(a_i)$ in $T$ produces a tuple $(T_0,\lambda_0) \in \RR^{r}, (T_1,\lambda_1) \in \RR^{s_1,0}, \ldots, (T_r,\lambda_r)\in \RR^{s_r,0}$.  Note that each edge $a_i$ has become the leaf $e_i$ of $T_0$ and the root $e_0$ of $T_i$.  By inductive hypothesis the surfaces $\sS_{(T_0,\lambda_0)}, \sS_{(T_1,\lambda_1)},\ldots, \sS_{(T_r,\lambda_r)}$ have already been constructed. 
We fix a normal coordinate $\nu$ to the facet $F_{s_1+\ldots+s_r}$, taking $\nu:=\lambda(p^{-1}(a_1))$; the relations on the colored tree $T$ then determine the values of admissible lengths $\lambda(p^{-1}(a_2)), \ldots, \lambda(p^{-1}(a_r))$.  Truncate both the semi-infinite strip $Z_{e_i}$ of $T_0$ and the semi-infinite strip $Z_{e_0}$ of $T_i$ at length $\lambda(p^{-1}(a_i))/2$, and identify them along the truncations.

These constructions are only well-defined for large $\nu$, defining the bundle $\sS^{d,0} \to \RR^{d,0}$ over  an open neighborhood of $\partial \ol{\RR}^{d,0}$.  The bundle is compatible by construction with the forgetful maps, i.e., the diagram (\ref{commute}) commutes.   We extend the bundle over the remainder of the interior of $\RR^{d,0}$ by choosing a family of smooth isotopies.  All of the surfaces in the family have the same striplike ends, so we may assume that the isotopies are compactly supported.  Furthermore, we may also choose the isotopies to be compatible with the forgetful maps, that is, we may choose them such that (\ref{commute}) commutes.   This completes the inductive step. 

\subsubsection{Fixing the complex structures}  The final step is to fix an area form on each quilted surface in the bundle $\sS^{d,0}\to \RR^{d,0}$, and a compatible complex structure, using the fact that the space of area forms is convex, and the space of compatible complex structures contractible.  Given coordinates $(s,t)$ we say that the area form, and complex structure are {\em standard} if they are respectively $ds\wedge dt$, and $j(\partial_s)= \partial_t, j(\partial_t) = -\partial_s$.  We fix the area form and complex structure as follows.  First, on each striplike end, the area form and complex structure are chosen to be the standard ones on each strip.  In particular, on the quilted striplike end, the $t$ coordinate corresponds to a rescaling of  the corresponding coordinate of the underlying strip by a factor of 3.  Second, choose a tubular neighborhood for each boundary component, and a tubular neighborhood of the seam. For the tubular neighborhood of the seam, fix local coordinates $(s,t)$ so that along the striplike end they coincide with the local coordinates fixed there, and then take the standard area form and complex structure for these coordinates.  For each boundary component, one takes a tubular neighborhood of that component in the coordinates of the underlying surface, rescaling if necessary in the $t$ direction in order to get local coordinates which match the standard coordinates on the striplike ends, and then take the standard area form and complex structure with respect to those coordinates. Lastly, using convexity of area forms one can smootly extend the area form over the remainder of the quilt, and then smoothly extend the compatible complex structure over the remainder of the quilt.  To achieve consistency one builds up choices of area form and complex structure over all bundles $\sS^{d,0}\lra \RR^{d,0}$ inductively, using the gluing procedure to determine area form and complex structure for all quilts in the bundle over a neighborhood of $\partial \RR^{d,0}$, and then smoothly interpolating both over the interior of $\RR^{d,0}$.  

\subsubsection{Attaching strips.}
Each boundary component is by construction equipped with a tubular neighborhood on which the complex structure is the standard complex structure for a strip.  Strips of unit width can be real-analytically attached to the boundary components of the quilted surfaces in the families $\sS^{d,0}$; we call the generalized quilted surfaces obtained in this manner {\em generalized pointed quilted disks} (Figure \ref{gen_disk}).
\begin{figure}[h]
\center{\includegraphics[height=1in]{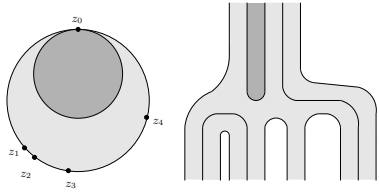}}
\caption{A generalized quilted disk, parametrized by a quilted marked disk.}\label{gen_disk}
\end{figure}

\begin{definition}
Borrowing more terminology from \cite{seidel-book}, we say that the {\em thin} part of a quilt $\sS_r$ consists of the striplike ends, together with the images of truncated striplike ends that were glued to form $\sS_r$.  The complement of the thin part will be called the {\em thick} part.  We call this the {\em thick/thin decomposition} of $\sS_r$.   
\end{definition}

\section{Pseudoholomorphic quilted disks}\label{section_j-holo_quilts}

The framework that we use is based on that of \cite{seidel-book}, Part II, sections 7 through 9.

\subsection{Floer and perturbation data} 

To each striplike end of $\ul{S}$, with Lagrangian boundary conditions given by $\ul{L}$, we assign a {\em Floer datum}, which is a regular pair 
\[
(\ul{H}, \ul{J}) = ( (H_k)_{k=1,\ldots,m}, (J_k)_{k=1,\ldots,m})
\] 
of a Hamiltonian perturbation and an almost complex structure of split type.   

\begin{definition}
Let $\ul{L}, \ul{L}^\prime$ be generalized Lagrangian submanifolds of $M$, and let $(\ul{H}, \ul{J})$ be a Floer datum for the pair.  A {\em generalized intersection} of a pair of generalized Lagrangians 
\bea
\ul{L} &=& \{pt\} \overset{L^{-n}}{\lra} M_{-n} \overset{L^{-n,-n+1}}{\lra} \ldots M_{-1}\overset{L^{-1,0}}{\lra} M_0=M,\\
\ul{L}^\prime &=& \{pt\} \overset{L^{m}}{\lra} M_{m}\overset{L^{m,m-1}}{\lra} \ldots M_1\overset{L^{1,0}}{\lra} M_0=M
\eea
given a Hamiltonian perturbation $\ul{H}$ is a tuple of paths
\[
\ul{x} = \{ (x_i)_{i=-n}^m \lvert \ x_i:[0,1]\to M_i, x_i^\prime(t) = X_{H^i_t}(x_i(t)), (x_{i-1}(1), x_{i}(0))\in L_{i-1,i} \}.
\]
We denote the set of generalized intersections by $\mathcal{I}(\ul{L}, \ul{L}^\prime)$. 
\end{definition}

Fix a collection of generalized Lagrangian submanifolds $\ul{L}_0, \ldots, \ul{L}_d$ of a symplectic manifold $M_A$.   Given a Lagrangian correspondence $L_{AB}$ between $M_A$ and $M_B$, write 
\bea
\ul{L}_{0, AB}:=\ul{L}_0 \# L_{AB} , \ \ \ul{L}_{d, AB}:= \ul{L}_d\# L_{AB}
 \eea 
for the concatenations, which are generalized Lagrangian submanifolds of $M_B$.  

Assign each pair $(\ul{L}_i, \ul{L}_{i+1})$ a Floer datum $(\ul{H}_i, \ul{J}_i)$, and assign  $(\ul{L}_{0,AB}, \ul{L}_{d,AB})$ a Floer datum  $(\ul{H}_{AB}, \ul{J}_{AB})$.  A {\em perturbation datum} is a pair $(\ul{K}, \ul{J})$, $\ul{K}$ in the perturbation datum is a smooth family of $1$-forms on the fibers $\ul{S}_r$ which take values in the space of Hamiltonian functions on $\ul{M}$.  Thus, a choice of $\ul{K}$ determines  a 1-form $\ul{Y}$ on each fiber $\ul{S}_r$, taking values in the space of Hamiltonian vector fields on $\ul{M}$.  A perturbation datum is {\em compatible} with the striplike ends and Floer data if it is equal on each striplike end to the Floer datum for the pair of Lagrangians labeling the end.   We fix a {\em universal choice of perturbation data} (i.e., for all bundles $\sS^{d,0}$ and $\sS^d$), which is {\em consistent} with the recursive constructions of the bundles, i.e., the following conditions hold:
 first, for each $d\geq 1$ (resp. $d\geq 2$) there is a subset $U \subset \RR^{0,d}$ (resp. $\RR^d$), where gluing parameters are sufficiently small, such that the perturbation data is given by the Floer data on the thin parts of the surfaces $\sS_r, r\in U$.  Second, if $(\tilde{K}, \tilde{J})$ is the perturbation datum on $U\subset \sS^{d,0}$ (resp. $\sS^d$) obtained by gluing perturbation data of the lower strata, then the perturbation datum in the universal family for $\sS^{d, 0}$ (resp. $\sS^{d}$) extends smoothly over $\partial\RR^{d,0}$ (resp. $\partial\RR^d$) and agrees with $(\tilde{K}, \tilde{J})$ over the boundary $\partial\RR^{d,0}$ (resp. $\partial\RR^d$). 

Fix a quilted surface $\ul{S}_{r_0}, r_0 \in \RR^{d,0}$, which is labeled by the Lagrangians $\ul{L}_0, \ldots, \ul{L}_d$ and the Lagrangian correspondence $L_{AB}$. In a neighborhood $U$ of $r_0$ all of the quilted surfaces are diffeormorphic, giving rise to a family of diffeomorphisms parametrized by points in $U$,
\bea
\Psi: U \times \sS_{r_0} \mapsto \sS\big\lvert_U, \ \ \ \ 
\Psi(r, \cdot) : \sS_{r_0} \overset{\cong}{\lra} \sS_r
\eea
which identify strip-like ends, i.e., $\Psi(r, \ul{\eps}_\zeta(r_0, s,t)) = \ul{\eps}_\zeta(r, s, t)$. 
Each quilted surface $\ul{S}_r$ is equipped with complex structures $\ul{j}_{\sS_r}$ which can be pulled back by $\Psi_r$ to give a family of complex structures on $\ul{S}_{r_0}$ parametrized by $r$, $
\ul{j}(r):= \Psi_r^*(\ul{j}_{\sS_r})$.

The inhomogeneous pseudo-holomorphic map equation for $(r,\ul{u}) \in \B$ is 
\begin{equation}\label{pseudo-holo}
\left\{\begin{array}{ll}d\ul{u}(z) + \ul{J}(r,\ul{u},z)\circ d\ul{u}(z) \circ \ul{j}(r) = \ul{Y}(r,\ul{u},z) + \ul{J}(r,\ul{u},z)\circ \ul{Y}(r,\ul{u},z) \circ \ul{j}(r)\\
\ul{u}(C) \subset L_C \ \mbox{for all seams/boundary components}\ C \ \mbox{with label}\ L_C.
\end{array}\right.
\end{equation}
The compatibility of the perturbation data with the Floer data along the striplike ends means that the above equation reduces to Floer's equation along the striplike ends.  In particular, solutions with finite energy converge exponentially along each striplike end to a generalized intersection of the pair of  generalized Lagrangians labeling that striplike end.   

A tuple
$(\underline{y}_0, \ul{x}_1, \ldots, \ul{x}_d) \in \cI(\ul{L}_{0,AB}, \ul{L}_{d,AB})\times \cI(\ul{L}_0, \ul{L}_1) \times \ldots \times \cI(\ul{L}_{d-1}, \ul{L}_d)$
of generalized intersections determines a fiber bundle $\B_{\sS\big\lvert_U} \lra U \subset \RR^{d,0}$, 
whose fiber over $r$ is the space of smooth quilt maps from $\sS_r$ to $\ul{M}$ with boundary in $\ul{L}$ and which converge along striplike ends to the given tuple of intersection points. There is another fiber bundle 
\[
\E_{\sS\big\lvert_U} \lra \B_{\sS\big\lvert_U},
\] 
whose fiber $\E_{(r, \ul{u})} = \Omega^{0, 1,r}(\ul{\sS}_r, \ul{u}^*T\ul{M})$ of $(0,1)$ forms on $\sS_r$ taking values in the pullback bundle $\ul{u}^*TM$, and the $(0,1)$ part is with respect to $\ul{J}(r,\ul{u})$ and $\ul{j}(r)$. 

Since equation (\ref{pseudo-holo}) can be written as $(d\ul{u} - \ul{Y})^{0,1} = 0$, we will also abbreviate it as $(\ol{\partial} - \ul{\nu})(r, \ul{u})  = 0$, where $\ol{\partial}(r,\ul{u}):= (d\ul{u})^{0,1}$ and $\ul{\nu}(r,\ul{u}):= (\ul{Y}(r, \ul{u}))^{0,1}$, with the $0,1$ taken with respect to $\ul{J}(r,\ul{u},z)$ and $\ul{j}(r)$.

Thus $\ol{\partial} - \ul{\nu}: \B \to \E$ defines a section whose intersection with the zero-section $\B_0$ is the set of solutions of (\ref{pseudo-holo}), whose striplike ends converge to the prescribed intersections.  The {\em energy} of such a solution $(r,\ulu)$ is the quantity 
\bea
E(r,\ulu) & = &\half \int_{\sS_r} |d\ulu - \ulY|^2 \ \dvol_{\sS_r} = \half\sum\limits_{\alpha\in \mathcal{A}} \int_{\sS_{r, \alpha}} |du_\alpha - Y_\alpha|^2 \dvol_{\sS_{r, \alpha}}
\eea
where the norms $|\cdot|$ on $TM_\alpha$ come from the induced metric $g_{J_\alpha(r)}$, and the area form $\dvol_{\sS_{r, \alpha}}$ comes from the construction of the bundles in Section \ref{families}.

\begin{definition} The {\em moduli space of holomorphic quilted disks} $\M_{d,1}(\underline{y}_0, \ul{x}_1, \ldots, \ul{x}_d)$ consists of all (finite energy) solutions $(r,\ul{u})$ to (\ref{pseudo-holo}) which converge along the strip-like ends labeled $\zeta_0, \ldots, \zeta_d$ to the respective intersections
$(\underline{y}_0, \ul{x}_1, \ldots, \ul{x}_d)$.  
For a collection $(\ul{x}_0, \ul{x}_1, \ldots, \ul{x}_d) \in \cI(\ul{L}_{0}, \ul{L}_{d})\times \cI(\ul{L}_0, \ul{L}_1) \times \ldots \times \cI(\ul{L}_{d-1}, \ul{L}_d)$, the {\em moduli space of holomorphic disks} $\M_{d}(\ul{x}_0, \ul{x}_1, \ldots, \ul{x}_d)$ consists of all (finite energy) solutions of the analogous equations for the surfaces $\sS^d \to \RR^d$, with generalized Lagrangian boundary conditions $\ul{L}_0, \ldots, \ul{L}_d$.  Similarly, given $\ul{x}, \ul{y} \in \I(\ul{L}, \ul{L}^\prime)$, the {\em moduli space of Floer trajectories} is $\widetilde{\M}(\ul{x},\ul{y}) := \M(\ul{x},\ul{y})/\R$, where $\M(\ul{x}, \ul{y})$ consists of all (finite energy) solutions to Floer's inhomogeneous equation converging to $\ul{x}$ and $\ul{y}$.  
\end{definition}

\subsection{Sobolev spaces}\label{triv}

Consider a quilt consisting of surfaces $\ul{S} = (S_\alpha)_{\alpha \in \mathcal{A}}$ and set $\mathcal{C}$ of seams and boundary components, equipped with labelings of the surfaces by target symplectic manifolds $\ul{M} = (M_\alpha)_{\alpha\in \mathcal{A}}$ and labelings of the seams and boundary components by Lagrangian boundary conditions $\ul{L} = (L_C)_{C\in \mathcal{C}}$.   Let $\ul{u}: \ul{S} \to \ul{M}$ be a smooth quilt map with boundary in $\ul{L}$, and let $\ul{J} = (J_\alpha(z), z \in S_\alpha)_{\alpha \in \mathcal{A}}$ be a smooth family of compatible almost complex structures varying over the surfaces.  Given an area form on each patch of $\ul{S}$, define Sobolev norms on sections of the pull-backs of the tangent bundles by
\bea
W^{1,p}(\ul{S}, \ul{u}^* T\ul{M}) & := &\bigoplus\limits_{\alpha \in \mathcal{A}} W^{1,p}(S_\alpha, u_\alpha^*TM_\alpha),\\
L^p(\ul{S}, \Lambda^{0,1}\otimes_{\ul{J}}\ul{u}^*T\ul{M}) & := & \bigoplus\limits_{\alpha \in \mathcal{A}} L^p(S_\alpha, \Lambda^{0,1}\otimes_{J_\alpha} u_\alpha^* TM_\alpha)
\end{eqnarray*}
where the norms on $TM_\alpha$ come from the induced metrics $g_{J_\alpha(z)}(\cdot, \cdot) := \omega(\cdot, J_\alpha(z) \cdot)$.
\subsection{Metric connections and local trivializations}  
We recall the following lemma of Frauenfelder on the existence of metrics with useful properties with respect to a given almost-complex structure and a given Lagrangian submanifold.  
\begin{lemma}[Lemma 4.3.3 in \cite{mcd-sal}]\label{frauen}
Let $(M,\omega)$  be a symplectic manifold equipped with an almost complex structure $J$, and let $L\subset M$ be a Lagrangian submanifold.  There exists a Riemannian metric $g = \langle \cdot, \cdot \rangle$ on $M$ such that 
\begin{enumerate}
\item[(i)] $\langle J(p)v, J(p)w\rangle = \langle v, w\rangle$ for $p\in M$ and $v, w \in T_pM$,

\item[(ii)] $J(p)T_pL$ is the orthogonal complement of $T_pL$ for every $p\in L$,

\item[(iii)] $L$ is totally geodesic with respect to $g$.

\end{enumerate}
\end{lemma}

Let $\nabla$ be the Levi-Civita connection of $g$.  The Hermitian connection $\widetilde{\nabla}$ associated to an almost complex structure $J$ is 
\begin{equation}\label{complex_connection}
\widetilde{\nabla}_v X := \nabla_v X - \frac{1}{2} J( \nabla_v J) X.
\end{equation}

\begin{lemma}\label{tot_geod} Under the assumptions of Lemma \ref{frauen}, $L$ is totally geodesic with respect to $\widetilde{\nabla}$, the Hermitian connection (\ref{complex_connection}).  
\end{lemma} 
\begin{proof}
Write $P: TM\big\lvert_L \to TL$ for orthogonal projection of $TM$ onto $TL$ defined pointwise on $L$.   
It suffices to check that for every $p \in L$, and every $X, Y \in T_p L$, $\widetilde{\nabla}_X Y \in T_p L$.  
By definition,
\bea
\widetilde{\nabla}_X Y & = & \nabla_X Y - \frac{1}{2} J (\nabla_X J) Y. 
\eea
where $\nabla$ is the Levi-Civita connection of $g$.  By assumption, $L$ is totally geodesic with respect to $\nabla$, therefore $\nabla_X Y \in T_p L$ for all $p \in L$ and all $X, Y \in T_p L$.   So it is enough to show that $J(\nabla_X J)Y \in T_p L$.  Since the orthogonal complement of $T_p L$ is $JT_pL$, it reduces to showing that for all $X, Y, Z \in T_pL$,  $J(\nabla_X J)Y \perp JZ$.  
\bea
g(J(\nabla_X J)Y, JZ) & = & g((\nabla_X J)Y, Z)
 =  d (g(JY,Z))(X) - g(J \nabla_XY, Z) - g(JY, \nabla_XZ)
 =  0.
\eea
\end{proof}

Now let $\sS \lra \RR$ be a family  of quilted surfaces of type either $\sS^d \lra \RR^d$ or $\sS^{d,0} \lra \RR^{d,0}$, as constructed in Section \ref{families}.   We fix a basepoint $r_0 \in \RR$, and a neighborhood $\mathcal{U} \subset \RR$ of $r_0$.  Write $\ul{S}:=\sS_{r_0}$ for the quilted surface over $r_0$.  All the quilted surfaces over $U$ are diffeomorphic, and moreover the diffeomorphisms can be chosen to be the identity on the striplike ends, so the bundle $\sS\big\lvert_{\mathcal{U}} \to U$ can be viewed as a single quilted surface $\ul{S}$ with a family of complex structures parametrized by $r \in U$, such that the complex structures $\ul{j}(r)$ are all standard on the striplike ends.    
 
 Each patch $S_\alpha$ of the quilt $\ul{S}$ is labeled by a target symplectic manifold $(M_\alpha, \omega_\alpha)$ and an $\omega_\alpha$-compatible almost complex structure $J_\alpha = J_\alpha(r_0)$, which determine an induced metric $g_{J_\alpha(r_0)}$.  We collectively write $\ul{g}$ for these metrics, and write $\ul{\exp}$ for the exponential maps of their associated Hermitian connections (\ref{complex_connection}).  Note that once a base point $r_0$ is fixed, this fixes a metric on each $M_\alpha$ for each $z \in S_\alpha$.

The Lagrangian boundary conditions are not necessarily totally geodesic with respect to the exponential maps in $\ul{\exp}$.  We therefore introduce a collection of auxiliary metrics $\ul{g}^Q$, parametrized by the quilt $\ul{S}$, to satisfy the prescribed boundary conditions at seams and boundary components (c.f. \cite[Remark 2.2]{ww_1a}).  
 
 For each boundary component or seam $C$ labeled by a Lagrangian $L_C \subset M_\alpha^-\times M_\beta$,  fix a metric $g_C$ on $M_\alpha^-\times M_\beta$ that satisfies properties (i), (ii) and (iii) of Lemma \ref{frauen} with respect to the almost complex structure $(-J_\alpha)\times J_\beta$ and the Lagrangian $L_C$.  A $\delta$-tubular neighborhood of $C$ in the quilt is a copy of $\R \times [-\delta,\delta]$, where the side where $t <0$ is labeled by $M_\alpha$, the side where $t >0$ is labeled by $M_\beta$, and the seam $t=0$ is labeled by $L_C$.   We fold this to $\R \times [0,\delta]$ labeled by the product manifold $M_\alpha^-\times M_\beta$ and the Lagrangian boundary condition $L_C$.  Over the interval $[0,\delta]$ fix a smooth interpolation (of metrics on $M_\alpha^-\times M_\beta$) between the metric $g_C$ and the induced metric $g_{J_\alpha}\times g_{J_\beta}$.   Note that both $g_C$ and $g_\alpha\times g_\beta$ are $(-J_\alpha)\times J_\beta$-invariant, so we can choose the interpolating metrics to be $(-J_\alpha)\times J_\beta$-invariant too.  On the complements of the tubular neighborhoods, we take $\ul{g}^Q$ to be the induced metric $\ul{g}$.
 
 We write $\ul{g}^Q$ for the resulting collection of metrics parametrized by the quilt $\ul{S}$.  We emphasise that on each tubular neighborhood of a {\em seam}, we view $\ul{g}$ as a metric on the associated {\em product} manifold, parametrized by the {\em folded} tubular neighborhood.  Let $\ul{\exp}^Q$ denote the collection of exponential maps of the associated Hermitian connections (\ref{complex_connection}).  
We will use $\ul{\exp}^Q$ to define {\em quadratic corrections} to $\ul{\exp}$.  
  
Define a function $Q: \ul{S}\times \Omega^0(\ul{S}, \ul{u}^*T\ul{M}) \to \Omega^0(\ul{S}, \ul{u}^*T\ul{M})$ as follows.  Let $\ul{\xi}  \in \Omega^0(\ul{S}, \ul{u}^*T\ul{M})$ be a section of the pull-back bundle over the quilt $\ul{S}$.  In a tubular neighborhood of a true boundary component of the patch $S_\alpha$, with local coordinates $z = s + it \in \R\times i[0,\delta)$, define $Q((z,\ul{\xi})(z)$ by the condition that $\exp_{u_\alpha(z)}(\xi(z) + Q(z,\ul{\xi})(z)) = \exp^Q_{u_\alpha(z)}(\xi(z))$.  In a tubular neighborhood of a seam between patches $S_l$ and $S_r$, suppose we have local coordinates $z=s+it \in \R \times i[-\delta, \delta]$, where $S_l$ corresponds to $t <0$ and $S_r$ corresponds to $t>0$.  For $z=s+it \in S_l$ we define $Q(z,\ul{\xi})(z)$ and $Q(\ol{z}, \ul{\xi})(\ol{z})$ (where $\ol{z} = s-it \in S_r$) by the condition that
\[
(\exp_{u_l(z)}(\xi_l(z) +Q(z,\ul{\xi})(z)),\exp_{u_r(\ol{z})}(\xi_r(\ol{z}) + Q(\ol{z}, \ul{\xi})(\ol{z})) = \exp^Q_{(u_l(z),u_r(\ol{z}))}(\xi_l(z),\xi_r(\ol{z}))
\]
 as points on the product manifold $M_l^-\times M_r$.   For a point $z \in \ul{S}$ in the complement of the tubular neighborhoods, $Q(z, \ul{\xi}) = 0$ for all $\ul{\xi} \in \Omega^0(\ul{S}, \ulu^*T\ul{M})$.

\begin{lemma}
$Q(z,\ul{\xi})$ defined above has the following properties.
\begin{enumerate}
\item For all $z \in \ul{S}$, $Q(z, 0) = 0$.

\item For all $z \in \ul{S}$, and $\ul{\xi} \in \Omega^0(\ul{S},\ulu^*T\ulM)$,  
$\partial_\lambda Q(z, \lambda \ul{\xi})\lvert_{\lambda=0} = dQ(z,0)(0,\ul{\xi})  = 0$.

\end{enumerate}
\end{lemma}
\begin{proof}
(a) follows from $\ulu = \ul{\exp}_{\ulu}^Q(0) = \ul{\exp}_{\ulu}(Q(z,0))$.  (b) follows from differentiating $\ul{\exp}_{\ulu}^Q(\lambda \ul{\xi}) = \ul{\exp}_{\ulu}(\lambda \ul{\xi} + Q(z, \lambda \ul{\xi}))$ with respect to $\lambda$ at ${\lambda=0}$, which gives $\ul{\xi} = \ul{\xi} + \partial_\lambda Q(z, \lambda \ul{\xi})\lvert_{\lambda=0}$.
\end{proof}
The following lemma will be relevant for explicit computations of linearized operators.
\begin{lemma}
Let $(M, \omega)$ be a symplectic manifold, $p\in M$ and $\xi \in T_pM$. Fix a connection $\nabla$ on $TM$, with exponential map $\exp$.  Suppose that $Q: T_pM \to T_pM$ satisfies $Q(0) = 0$ and $DQ(0)=0$. For $\lambda \in [0,\delta)$, and $t \in [0,1]$, let
\bea
\gamma(\lambda, t) & = & \exp_p(\lambda t \xi + Q(\lambda t \xi))\\
\theta(\lambda, t) & = & \exp_p(t (\lambda \xi + Q(\lambda \xi))).
\eea
(Note that these paths coincide at $t=0,1$.) Let $\Phi_p(\lambda \xi)$ denote parallel transport along the curve $t\mapsto \gamma(\lambda,t)$, and let $\Psi(\lambda \xi)$ denote parallel transport along the curve $t\mapsto \theta(\lambda, t)$.  Then, for every $\eta \in T_p M$, we have that
\begin{equation}
\frac{d}{d\lambda}\Big\lvert_{\lambda=0} \Phi_p(\lambda \xi)^{-1} \Psi_p(\lambda \xi) \eta = 0.
\end{equation}
Thus, in particular, for any vector field $f$ along the curve $\lambda \mapsto \exp(\lambda\xi + Q(\lambda \xi))$,   
\begin{equation}
\frac{d}{d\lambda}\Big\lvert_{\lambda = 0} \Phi_p(\lambda\xi)^{-1} f = \nabla_\lambda f \Big \lvert_{\lambda = 0}.
\end{equation}
\end{lemma}
\begin{proof}
Let $p_\lambda = \gamma(\lambda,1) = \theta(\lambda,1)$ be the common end point of the two paths for fixed $\lambda$.  It suffices to show that the difference
\[
\Psi_p(\lambda \xi)\eta - \Phi_p(\lambda \xi) \eta \in T_{p_\lambda}M
\]  
is only quadratic in $\lambda$.  For sufficiently small $\lambda$ both curves must be contained in a local coordinate chart and we can calculate the difference directly.  Let $x_1, \ldots, x_n$ be local coordinates for a neighborhood of $p$ in $M$.  Write $\eta_\lambda(\gamma(\lambda, t))$ for the parallel transports of $\eta \in T$ along $\gamma(\lambda, t)$, and write $\widetilde{\eta}_\lambda(\theta(\lambda,t))$ for the parallel transports of $\eta$ along $\theta(\lambda,t)$.  In terms of the local coordinate components, for $j = 1, \ldots, n$ we have from the parallel transport equations and the fundamental theorem of calculus that
\bea
\eta_\lambda^j(p_\lambda) & = & \eta^j_\lambda(p) - \int\limits_0^1 \Gamma_{ih}^j(\gamma_\lambda(t)) \eta^i_\lambda(\gamma_\lambda(t))\frac{d\gamma^h_\lambda}{dt} \ dt\\
\widetilde{\eta}_\lambda^j(p_\lambda) & = & \widetilde{\eta}^j_\lambda(p) - \int\limits_0^1 \Gamma_{ih}^j(\theta_\lambda(t))\widetilde{\eta}^i_\lambda(\gamma_\lambda(t))\frac{d\theta^h_\lambda}{dt} \ dt
\eea 
where $\Gamma_{ih}^j$ are the Christoffel symbols for $\nabla$ in these coordinates.  Since $\eta_\lambda(p) = \eta = \widetilde{\eta}_\lambda(p)$, we have that in each coordinate, $j = 1, \ldots, n$, 
\[
\eta_\lambda^j(p_\lambda) - \widetilde{\eta}_\lambda^j(p_\lambda) = -\int\limits_0^1  (\Gamma_{ih}^j(\gamma_\lambda(t)) \eta^i_\lambda(\gamma_\lambda(t))\frac{d\gamma^h_\lambda}{dt} - \Gamma_{ih}^j(\theta_\lambda(t))\widetilde{\eta}^i_\lambda(\gamma_\lambda(t))\frac{d\theta^h_\lambda}{dt}) \ dt.
\]
We now compute the derivative of this quantity with respect to $\lambda$ at $\lambda =0$. 
\bea
\frac{d}{d\lambda} \Big\lvert_{\lambda=0}(\eta_\lambda^j(p_\lambda) - \widetilde{\eta}_\lambda^j(p_\lambda)) & = &  -\frac{d}{d\lambda} \Big\lvert_{\lambda=0} \int\limits_0^1  (\Gamma_{ih}^j(\gamma_\lambda(t)) \eta^i_\lambda(\gamma_\lambda(t))\frac{d\gamma^h_\lambda}{dt} - \Gamma_{ih}^j(\theta_\lambda(t))\widetilde{\eta}^i_\lambda(\gamma_\lambda(t))\frac{d\theta^h_\lambda}{dt} \ dt\\
& = & -\int\limits_0^1 \frac{\partial}{\partial\lambda}\left( \Gamma_{ih}^j(\gamma_\lambda(t)) \eta^i_\lambda(\gamma_\lambda(t))\frac{d\gamma^h_\lambda}{dt} - \Gamma_{ih}^j(\theta_\lambda(t))\widetilde{\eta}^i_\lambda(\gamma_\lambda(t))\frac{d\theta^h_\lambda}{dt}\right) \Big\lvert_{\lambda = 0} \ dt.
\eea
We now show that the integrand is zero.  When $\lambda = 0$, $\gamma(0, t) = \theta(0,t) = p$ and therefore $\frac{d\theta^h_\lambda}{dt}\lvert_{\lambda = 0 } = \frac{d\gamma^h_\lambda}{dt}\lvert_{\lambda=0} = 0$.  Hence, the integrand above reduces to 
\bea
\Gamma^j_{ih}(p)\eta^i \frac{\partial}{\partial\lambda} \left(\frac{d\gamma^h_\lambda}{dt} -  \frac{d\theta^h_\lambda}{dt}   \right)\Big \lvert_{\lambda = 0}
\eea
and now it is enough to show that $\frac{\partial}{\partial\lambda} \left(\frac{d\gamma^h_\lambda}{dt} -  \frac{d\theta^h_\lambda}{dt}   \right)\Big \lvert_{\lambda = 0} = 0$.  Using the equality of mixed partials for the smooth function $\gamma^h(\lambda, t) - \theta^h(\lambda, t)$ we can write
\bea
\frac{\partial}{\partial\lambda} \left(\frac{\partial}{\partial t} (\gamma^h(\lambda, t) -  \theta^h(\lambda, t))\right) \lvert_{\lambda = 0}   & = &  \frac{\partial}{\partial t} \left(\frac{\partial}{\partial \lambda} (\gamma^h(\lambda, t) - \theta^h(\lambda, t))\lvert_{\lambda = 0}\right). 
\eea
Now using the definitions of $\gamma(\lambda, t)$ and $\theta(\lambda, t)$ we see that $\partial_\lambda \gamma(\lambda, t) \lvert_{\lambda = 0} =   t\xi + DQ(0)(t\xi) = t\xi$, while $\partial_\lambda \theta(\lambda, t) = t\xi + tDQ(0)(\xi) = t\xi$, so the quantities above must be zero.

\end{proof}

\subsection{Local trivializations}
We now define local trivializations of the bundles $\B\lvert_U \to U$ and $\cE \to \B\lvert_U$.  A small neighborhood of $(r_0, \ul{u}_0)$ in $\B\big\lvert_\mathcal{U}$ consists of pairs $(r, \ulu)$ where $r\in U$ parametrizes the complex structure $j(r)$ on $\ul{S}$, and $\ulu: \ul{S} \to \ulM$ is a map satisfying the prescribed Lagrangian boundary conditions and limits along the striplike ends.  For $\ulu$ sufficiently close to $\ulu_0$, there is a unique $\ulxi \in \Omega^0(\ul{S}, \ulu_0^*T\ul{M})$ such that $\exp_{\ulu_0}^Q(\ulxi) = \ulu$.  Write $\Phi_{\ulu}^Q(\ul{\xi}): T_{\ulu} \ulM \to T_{\exp^Q_{\ulu}(\ul{\xi})}\ulM$ for parallel transport along curves $\ul{\exp}^Q_{\ulu}(\tau\ul{\xi} )=\ul{\exp}(\tau\ul{\xi}+ Q(\tau\ul{\xi}))$, $\tau \in [0,1]$, with respect to the Hermitian connection $\widetilde{\nabla}$.  We emphasize that these curves are {\em not} geodesics for the connection $\widetilde{\nabla}$, since they are defined with the corrected exponential map $\ul{\exp}^Q$ and not with $\exp$.  

Define a projection map 
\bea
[ \ ]^{0,1}_{(r)}  =   \Pi_{r}: \Omega^1(\ul{S}, \ulu^*T\ulM) & \to & \Omega^{0,1}_{r}(\ulu{S}, \ulu^*T\ulM), \ \ \ 
\psi  \mapsto  \frac{1}{2}(\psi + J(r)\circ \psi \circ j(r))
\eea
i.e., projection onto the $(0,1)$ part with respect to $J(r)$ and $j(r)$.  Note that if $\psi \in \Omega^{0,1}_r(\ul{S}, \ulu^*T\ulM)$, then $\Pi_r(\psi) =\psi$.  Also, if $r^\prime$ is close to $r$, then $\Pi_r$ determines an isomorphism
\[
\Pi_r : \Omega^{0,1}_{r^\prime}(\ul{S}, \ulu^*T\ulM) \to \Omega^{0,1}_{r}(\ul{S}, \ulu^*T\ulM).
\]
The fibers of $\cE$ over a small neighborhood of $(r_0, \ulu_0)$ can be identified using the isomorphisms  
\bea
\Phi_{\sS, r_0,\ul{u}_0}(\rho, \ul{\xi})^{-1} :  \Omega^{0,1}_{\exp_{r_0}(\rho)}(\ul{S}, \ul{\exp}^Q_{\ul{u}_0}(\ul{\xi})^*T\ul{M})  \to  \Omega^{0,1}_{r_0}(\ul{S}, \ul{u}_0^*T\ul{M}) \\
\beta  \mapsto  \Phi^Q_{\ul{u}_0}(\ul{\xi})^{-1} \frac{1}{2}\left( \beta + J(r_0,u,z) \circ  \beta \circ j(r_0) \right).
\eea
 Fix a metric on the compact finite dimensional space $\ol{\RR}$ once and for all, and an  exponential map $\exp_{r_0} : T_{r_0} U \to U$.  Use the exponential map on $\RR$ and the maps $\ul{\exp}^Q$ on $\ul{M}$ to identify $\B\lvert_{U} \cong T_{r_0} \RR \times \Omega^0(\ul{S}, \ul{u}_0^*T\ulM)$.   

The vertical part of the section $\ol{\partial} - \ul{\nu}: \B \to \cE$ at $(r_0, \ulu_0)$ is given by the non-linear map
\bea
\F_{\sS, r_0,\ulu_0}: T_{r_0} \RR \times \Omega^0(\ul{S}, \ul{u}_0^*T\ulM) & \to & \Omega^{0,1}_{r_0}(\ul{S}, \ul{u}_0^*T\ulM)\\
(\rho,\ul{\xi}) & \mapsto & \Phi_{\sS,r_0,\ul{u}_0}(\rho, \ul{\xi})^{-1}(\ol{\partial} - \ul{\nu})(\exp_{r_0}\rho, \ul{\exp}^Q_{\ul{u}_0}\ul{\xi}),
\eea
which extends to a non-linear function between Banach completions
\[
\F_{\sS, r_0, \ulu_0} : T_{r_0} \RR \times W^{1,p}(\ul{S}, \ul{u}_0^*T\ulM) \to L^p(\ul{S}, \Lambda^{0,1}\otimes_{J(r), j(r)} \ulu_0^*T\ul{M}).
\]

\subsection{The linearized operator}\label{lin_op}

We give explicit computations of the linearized operators.  For convenience we introduce the notation 
\bea
\widetilde{\F}_{\sS, r,\ul{u}}{(\rho,\ul{\xi})}& := & \Phi_{\sS,r,\ul{u}}(\rho, \ul{\xi})^{-1}\ol{\partial}(\exp_r\rho, \exp_{\ul{u}}\ul{\xi})\\
{\cP}_{\sS, r,\ul{u}}{(\rho,\ul{\xi})}& := & \Phi_{\sS,r,\ul{u}}(\rho, \ul{\xi})^{-1}\ul{\nu}(\exp_r\rho, \exp_{\ul{u}}\ul{\xi}),
\eea
($\cP$ for ``perturbation term"). The corresponding linearized operators are given by 
\[
  \underset{=:D_{\sS,r,\ul{u}}}{\underbrace{d\F_{\sS, r,u}(0,0)}} = \underset{=:\widetilde{D}_{\sS,r,\ul{u}}}{\underbrace{d\widetilde{\F}_{\sS, r,\ulu}(0,0)}} - \underset{=:P_{\sS,r,\ul{u}}}{\underbrace{d\cP_{\sS, r,\ulu}(0,0)}}. 
  \]
Abbreviating $r_\lambda := \exp_{r}(\lambda\rho), \ulu_{\lambda}:= \ul{\exp}^Q_{\ulu}(\lambda \ul{\xi})$, we can write
  \bea
  \widetilde{D}_{\sS,r,\ul{u}}(\rho,\ul{\xi}) & = 
 & \frac{d}{d\lambda}\Big\lvert_{\lambda=0} \Phi^Q_{\ulu}(\lambda\ul{\xi})^{-1}\frac{1}{2}\left( \ol{\partial}(r_\lambda, \ulu_\lambda) + J(r, \ulu_\lambda) \circ \delbar(r_\lambda, \ulu_\lambda) \circ j(r)\right)\\
  & = & \frac{1}{2} \left(\frac{d}{d\lambda}\Big\lvert_{\lambda=0}\Phi^Q_{\ulu}(\lambda\ul{\xi})^{-1}\ol{\partial}(r_\lambda, \ulu_\lambda) + J(r, \ulu) \circ \frac{d}{d\lambda}\Big\lvert_{\lambda=0}\Phi^Q_{\ulu}(\lambda\ul{\xi})^{-1}\ol{\partial}(r_\lambda, \ulu_\lambda)\circ j(r)\right)\\
  & = & \left[ \frac{d}{d\lambda}\Big\lvert_{\lambda=0}\Phi^Q_{\ulu}(\lambda\ul{\xi})^{-1}\ol{\partial}(r_\lambda, \ulu_\lambda)  \right]^{0,1}
  \eea
  where $\left[ \ \right]^{0,1}$ denotes projection onto the $(0,1)$ part with respect to $J(r, \ulu)$ and $j(r)$.  By direct calculation we get 
  \bea
   \frac{d}{d\lambda}\Big\lvert_{\lambda=0}\Phi^Q_{\ulu}(\lambda\ul{\xi})^{-1}\ol{\partial}(r_\lambda, \ulu_\lambda)  & = &  \frac{d}{d\lambda}\Big\lvert_{\lambda=0}\Phi^Q_{\ulu}(\lambda\ul{\xi})^{-1} \frac{1}{2}\left( d\ulu_\lambda + J(r_\lambda, \ulu_\lambda)\circ d\ulu_\lambda \circ j(r_\lambda)\right)\\
   & = & \frac{1}{2}\left( \widetilde{\nabla}_\lambda d\ulu_\lambda +  \partial_\lambda J(r_\lambda, \ulu)\circ d\ulu \circ j(r) \right.\\
   & & \left. + J(r,u)\circ \widetilde{\nabla}_\lambda d\ulu_\lambda \circ j(r)  + J(r,\ulu)\circ d\ulu\circ \partial_\lambda j(r_\lambda)\right)\\
   & = & \left[  \widetilde{\nabla}_\lambda d\ulu_\lambda\right]^{0,1} + \frac{1}{2}\partial_\rho J \circ d\ulu \circ j(r) + \frac{1}{2}J(r,\ulu)\circ d\ulu \circ \partial_\rho j.
  \eea
Thus the linearized operator $\widetilde{D}_{\sS, r, \ulu}$ splits into the two parts
\begin{eqnarray}
\widetilde{D}_{\sS, r, \ulu} (\rho, \ul{\xi}) =  \underset{=: \widetilde{D}_{\ulu}^{(r)}(\ul{\xi})}{\underbrace{\left[  \widetilde{\nabla}_\lambda d\ulu_\lambda\right]^{0,1}} }+ \underset{=:\widetilde{D}_r^{(\ulu)}(\rho)}{\underbrace{\left[  \frac{1}{2}\partial_\rho J \circ d\ulu \circ j(r) + \frac{1}{2}J(r,\ulu)\circ d\ulu \circ \partial_\rho j \right]^{0,1}}},\label{lin_delbar1}
\end{eqnarray}
and identical arguments as in \cite[Proposition 3.1.1]{mcd-sal} give the explicit local formula 
\begin{equation}\label{lin_delbar}
\widetilde{D}_{\ulu}^{(r)} (\ul{\xi}) = [\nabla \ul{\xi}]^{0,1} - \frac{1}{2}J(r, \ulu)(\nabla_{\ul{\xi}}J)(r,\ulu)\partial_{\ul{J}(r)}(\ulu),
\end{equation}
in terms of the Levi-Civita connection $\nabla$ of the metrics $\ul{g}_{\ul{J}(r)}$.  Similarly, 
\begin{eqnarray}
P_{\sS, r, \ulu}(\rho, \ul{\xi}) & = & \left[ \frac{d}{d\lambda}\Big\lvert_{\lambda = 0} \Phi_u^Q(\lambda\ulxi)^{-1} \frac{1}{2}\left(  Y(r_\lambda, \ulu_\lambda) + J(r_\lambda, \ulu_\lambda)\circ  Y(r_\lambda, \ulu_\lambda) \circ j(r_\lambda)\right) \right]^{0,1}\nonumber\\
& = & \underset{=:P_{\ulu}^{(r)}(\ul{\xi})}{\underbrace{\left[  \widetilde{\nabla}_{\lambda} Y(r, \ulu_\lambda) \right]^{0,1}}}  + \underset{=: P_r^{(\ulu)}(\rho)}{\underbrace{\left[ \partial_\rho Y  \right]^{0,1} + \left[   \partial_\rho J \circ Y \circ j + J \circ Y\circ \partial_\rho j \right]^{0,1}}}\label{lin_pert1}
\end{eqnarray}
and $P_{\ulu}^{(r)}$ is explicitly given locally by
\begin{equation}\label{lin_pert}
P_{\ulu}^{(r)}(\ul{\xi}) = \left[ \nabla_{\ul{\xi}}Y \right]^{0,1} - \frac{1}{2}\left[ J\circ \nabla_{\ul{\xi}}J\circ Y\right]^{0,1}.
\end{equation}

\subsection{Gromov convergence}

\begin{definition} 
Consider a sequence $\{(r_n, \ul{u}_n)\}_{n=1}^\infty \subset \M_{d,0}(\ul{x}_0, \ldots, \ul{x}_d)$. 

\begin{enumerate}
\item For $2 \leq e \leq d-1$, we say that the sequence {\em Gromov converges to the broken pair} 
\bea
(r_1, \ulu_1) & \in & \M_{d-e+1, 0}(\ul{x}_0, \ul{x}_1, \ldots, \ul{x}_i, \ul{y}, \ul{x}_{i+e +1}, \ldots, \ul{x}_d)\\
(r_2, \ulu_2) & \in & \M_e(\ul{y}, \ul{x}_{i+1}, \ldots, \ul{x}_{i+e})
\eea
if 
\begin{itemize}
\item $r_n \lra r_1 \#_0 r_2$ in the topology of $\RR^{d,0}$ near the boundary point $r_1 \#_0 r_2 \in \partial \RR^{d,0}$,
\item $E(\ulu_n) \lra E(\ulu_1) + E(\ulu_2)$,
\item $\ulu_n$ converges uniformly on compact subsets of $\sS_{r_1}$ to $\ulu_1$, and converges uniformly on compact subsets of $\sS_{r_2}$ to $\ulu_2$.  
\end{itemize}
\item For $ 1 \leq s_1, \ldots, s_k \leq d-1$ such that $s_1 + \ldots + s_k = d$, we say the sequence {\em Gromov converges to the broken tuple} 
\bea
(r_0, \ulu_0) & \in & \M_{k}(\ul{x}_0, \ul{y}_1, \ldots, \ul{y}_k)\\
(r_1, \ulu_1) & \in & \M_{s_1, 0}(\ul{y}_1, \ul{x}_1, \ldots, \ul{x}_{s_1})\\
(r_2, \ulu_2) & \in & \M_{s_2, 0}(\ul{y}_2, \ul{x}_{s_1 + 1}, \ldots, \ul{x}_{s_1 + s_2})\\
\ldots & & \\
(r_k, \ulu_k) & \in & \M_{s_k, 0}(\ul{y_k}, \ul{x}_{d-s_k + 1}, \ldots, \ul{x}_{d}) 
\eea 
if
\begin{itemize}

\item $r_n \lra r_0 \#_0 (r_1, \ldots, r_k) \in \partial \RR^{d,0}$ in the topology of $\RR^{d,0}$ near the boundary,

\item $E(u_n) \lra E(\ulu_0) +E(\ulu_1) + \ldots + E(\ulu_k)$,

\item $\ulu_n$ converges uniformly on compact subsets of $\sS_{r_j}$ to $\ulu_j$, for $j = 0, \ldots, k$.

\end{itemize}

\item For $i \in \{1, \ldots, d\}$, we say that the sequence {\em Gromov converges to the broken pair}
\bea
(r, \ulu) & \in & \M_{d,0}(\ul{x}_0, \ul{x}_1, \ldots, \ul{x}_{i-1}, \ul{y}, \ul{x}_{i+1}, \ldots, \ul{x}_d)\\
\ul{v} & \in & \widetilde{\M}_1(\ul{y}, \ul{x}_i)
\eea
if 
\begin{itemize}
\item $r_n \to r$ in $\RR^{d,0}$, where $r$ is in the interior of $\RR^{d,0}$,

\item $E(\ulu_n) \lra E(\ulu) + E(\ul{v})$,

\item $\ulu_n$ converges uniformly on compact subsets of $\sS_{r}$ to $\ulu$, and there is a sequence $\tau_n \in \R$ (shift parameters) such that if $(s,t)$ denote coordinates on the strip $\R \times [0,1]$, and $\epsilon_i : \R_{\geq 0} \times [0,1] \to \sS_{r}$ is the $i$-th striplike end of $\sS_{r}$, then the sequence of shifted  maps $\ulu_n( \epsilon_i(s+\tau_n, t))$ converges uniformly on compact subsets of $\R \times [0,1]$ to a fixed parametrization of the Floer trajectory $\ul{v}$. 
\end{itemize}

\item We say that the sequence {\em Gromov converges to the broken pair}
\bea
(r,u) & \in & \M_{d,0}(\ul{y}, \ul{x}_1, \ldots, \ul{x}_d)\\
\ul{v} & \in & \widetilde{\M}_1(\ul{x}_0, \ul{y})
\eea
if 
\begin{itemize}
\item $r_n \to r$ in $\RR^{d,0}$, where $r$ is in the interior of $\RR^{d,0}$,

\item $E(\ulu_n) \lra E(\ulu) + E(\ul{v})$,

\item $\ulu_n$ converges uniformly on compact subsets of $\sS_{r}$ to $\ulu$, and there is a sequence $\tau_n \in \R$ (shift parameters) such that if $(s,t)$ denote coordinates on the strip $\R \times [0,1]$, and $\epsilon_0 : \R_{\geq 0} \times [0,1] \to \sS_{r}$ is the $0$-th striplike end of $\sS_{r}$, then the sequence of shifted  maps $\ulu_n( \epsilon_0(s+\tau_n, t))$ converges uniformly on compact subsets of $\R \times [0,1]$ to a fixed parametrization of the Floer trajectory $\ul{v}$. 
\end{itemize}

\end{enumerate}
\end{definition}

\subsection{Gromov neighborhoods}\label{gnbds}

We now define what we call {\em Gromov neighborhoods} of a broken quilt of Type 1, 2 or 3.  For $\epsilon > 0$, we will define a subset $U_\epsilon \subset \B^{d,0}$.  Under these definitions, a sequence $(r_\nu, \ulu_\nu) \in \M_{d,0}(\ul{x}_0,\ldots, \ul{x}_d)$ will Gromov converge to a broken quilt if, and only if, given $\epsilon > 0$ there is a $\nu_0$ such that $(r_\nu, \ulu_\nu) \in U_{\epsilon}$ for all $\nu \geq \nu_0$.  

\subsection*{Type 1}  

Let $(r_1, \ulu_1)$ and $(r_2, \ulu_2)$ be a broken pair, 
\bea
(r_1, \ulu_1) & \in & \M_{d-e+1,0}(\ul{x}_0, \ul{x_1}\ldots, \ul{x}_{i-1}, \ul{y}, \ul{x}_{i + e+1}, \ldots, \ul{x}_d)^0\\
(r_2, \ulu_2) & \in & \M_{e}(\ul{y}, \ul{x}_i, \ul{x}_{i+1},\ldots, \ul{x}_{i+e})^0.
\eea
A small neighborhood of the point $r_1\#_0 r_2 \in \partial \RR^{d,0}$ is of the form 
\[
U \cong U_1\times U_2 \times [0, \epsilon)
\] 
where $U_1 \subset \RR^{d-e+1, 0}$ is a neighborhood of $r_1$, and $U_2 \subset \RR^e$ is a neighborhood of $r_2$, and the interval $[0,\epsilon)$ represents the gluing parameter.  Recall that a gluing parameter $\delta$ corresponds to a gluing length $R(\delta) = - \log(\delta)$.   

Fix a metric on $U_1 \subset \RR^{d-e+1,0}$ and a metric on $U_2 \subset \RR^{e}$, and define a metric topology on $U \cong U_1\times U_2 \times [0, \epsilon)$ by 
\[
\dist_U(r_1 \#_\delta r_2, r_1^\prime \#_{\delta^\prime} r_2^\prime) := \sup \{ \dist_{U_1}(r_1, r_1^\prime), \dist_{U_2}(r_2, r_2^\prime), |\delta - \delta^\prime| \}. 
\]  
By the construction of the surface bundles $\sS \lra \RR$, we can suppose that the neighborhood $U$ is sufficiently small that the corresponding neighborhoods $U_1 \subset \RR^{d-e+1,0}$ and $U_2 \subset \RR^{e}$ are also small enough that all surfaces in the bundles over them are diffeomorphic to each other by diffeomorphisms preserving the striplike ends.  Write $\sS_{r_1\#_\delta r_2} = \sS_{r_1}^\delta \cup \sS_{r_2}^\delta / \sim$, where $\sS_{r_i}^\delta$ represents the truncation of $\sS_{r_i}$ along the prescribed striplike end at $s =R(\delta)$, and $\sim$ is the identification of the two truncated surfaces along the cuts.

\begin{definition} Let $\epsilon > 0$ be given.  Define a {\em Gromov neighborhood} $U_\epsilon \subset \bB^{d, 0}$ of the pair $(r_1, \ulu_1), (r_2, \ulu_2)$ as follows:  $(r, \ulu) \in U_\epsilon$ if 
\begin{itemize}

\item $r = \widetilde{r_1} \#_{\delta} \widetilde{r_2} \in U$ with $\dist_U(  \widetilde{r_1} \#_{\delta} \widetilde{r_2}  , r_1\#_0 r_2) < \epsilon$,

\item $|E(\ulu_1) + E(\ulu_2) - E(\ulu)| < \epsilon$,

\item $\dist_{\ul{M}}(\ulu(z), \ulu_1(z)) < \epsilon$ for all $z \in \sS_{r_1}^\delta$, 

\item $\dist_{\ul{M}}(\ulu(z), \ulu_2(z)) < \epsilon$ for all $z \in \sS_{r_2}^\delta$. 
\end{itemize}
The metrics on the target manifolds $\ul{M}$ are those induced by their symplectic forms $\ul{\omega}$ and the choice of compatible almost complex structures $\ul{J}=\ul{J}(z)$.
\end{definition}

\subsection*{Type 2}

Let $(r_0, \ulu_0), \ldots, (r_k, \ulu_k)$ be a broken tuple of the form
\bea
(r_0, \ulu_0) & \in & \M_{k}(\ul{x}_0, \ul{y}_1, \ldots, \ul{y}_k)\\
(r_1, \ulu_1) & \in & \M_{s_1, 0}(\ul{y}_1, \ul{x}_1, \ldots, \ul{x}_{s_1})\\
(r_2, \ulu_2) & \in & \M_{s_2, 0}(\ul{y}_2, \ul{x}_{s_1 + 1}, \ldots, \ul{x}_{s_1 + s_2})\\
\ldots & & \\
(r_k, \ulu_k) & \in & \M_{s_k, 0}(\ul{y_k}, \ul{x}_{d-s_k + 1}, \ldots, \ul{x}_{d}). 
\eea 
A small neighborhood of the point $r_1\#_0 r_2 \in \partial \RR^{d,0}$ is of the form 
\[
U \cong U_0\times U_1 \times \ldots U_k \times [0, \epsilon)
\] 
where $U_0 \subset \RR^{k}$ is a neighborhood of $r_0$, $U_i \subset \RR^{s_i, 0}$ is a neighborhood of $r_i$ for $i = 1, \ldots, k$, and the interval $[0,\epsilon)$ represents the gluing parameter. Fixing a metric on $U_0, \ldots, U_i$ determines a metric topology on $U \cong U_1\times U_2 \times [0, \epsilon)$ by 
\[
\dist_U(r_0 \#_\delta (r_1, \ldots, r_k),  r_0^\prime \#_{\delta^\prime} (r_1^\prime, \ldots, r_k^\prime)) := \sup \{ \dist_{U_0}(r_0, r_0^\prime),  \ldots, \dist_{U_k}(r_k, r_k^\prime), |\delta - \delta^\prime| \}. 
\]
Taking the neighborhood $U$ to be sufficiently small we can assume that all surfaces parametrized by $U_0, \ldots, U_k$ are diffeomorphic via diffeomorphisms that are constant on the striplike ends.   Write $\sS_{r_0} \#_\delta (r_1, \ldots, r_k) = \sS_{r_0}^\delta \cup \sS_{r_1}^\delta \cup \ldots \cup \sS_{r_k}^\delta/\sim$, where each $\sS_{r_i}^\delta$ is the truncation of the surface $\sS_{r_i}$ along the prescribed striplike end at $s = R(\delta) = -\log(\delta)$, and $\sim$ is the identifications of the surfaces along the truncated ends.  

\begin{definition} Let $\epsilon > 0$ be given.  Define a {\em Gromov neighborhood} $U_\epsilon \subset \bB^{d, 0}$ of the tuple $(r_0, \ulu_0), \ldots, (r_k, \ulu_k)$ as follows:  $(r, \ulu) \in U_\epsilon$ if 
\begin{itemize}

\item $r = \widetilde{r_0} \#_{\delta} \{\widetilde{r_1}, \ldots, \widetilde{r_k}\} \in U$ with $\dist_U( r  , r_0\#_0 \{r_1, \ldots, r_k\}) < \epsilon$,

\item $|E(\ulu_0) + E(\ulu_1) + \ldots + E(\ulu_k) - E(\ulu)| < \epsilon$,

\item $\dist_{\ul{M}}(\ulu(z), \ulu_i(z)) < \epsilon$ for all $z \in \sS_{r_i}^\delta$, $i = 0, \ldots, k$.

\end{itemize}

\end{definition}

\subsection*{Type 3}

Let $(r_0, \ulu_0) \in \M_{d,0}(\ul{x}_0, \ldots, \ul{x}_{i-1}, \ul{y}, \ul{x}_{i+1}, \ldots, \ul{x}_d)$ be a pseudoholomorphic quilted disk, and let $\ul{v} \in \widetilde{\M}_1(\ul{y},\ul{x}_i)$ be a quilted Floer trajectory.  We fix a parametrization of the Floer trajectory $\ul{v}: \R\times [0,1] \to \ul{M}$.   As in previous sections we write $\epsilon_i: [0,\infty) \times [0,1] \to \sS_{r}$ for the $i$-th striplike end, and  $Z_i$ for the image of this striplike end in $\sS_r$.  

\begin{definition}
Let $\epsilon > 0$ be given, and define $R(\epsilon) = -\log (\epsilon)$.  Define a {\em Gromov neighborhood} $U_\epsilon \subset \bB^{d,0}$ of the pair $(r_0, \ulu_0), \ul{v}$ as follows:  $(r,\ulu) \in U_\epsilon$ if
\begin{itemize}

\item $\dist_{\RR^{d,0}}(r, r_0) < \epsilon$,

\item $|E(\ulu_0) + E(\ul{v}) - E(\ulu)| < \epsilon$,

\item for $z \in \sS_{r_0}^{R(\epsilon)}$, $\dist_{\ul{M}}(\ulu(z), \ulu_0(z)) < \epsilon$,

\item there exists some $\tau \geq 2 R(\epsilon)$ such that $\dist_{\ul{M}}(\ulu(\epsilon_i(s+\tau, t), \ul{v}(s,t)) < \epsilon$ for $(s,t) \in [-R(\epsilon), R(\epsilon)]\times [0,1]$. 

\end{itemize}

\end{definition}

\section{Gluing}\label{gluing_section}

 \begin{definition} A pseudoholomorphic quilt $(r,\ulu)$ is {\em regular} if the linearized operator $D_{\sS,r,\ulu}$ is surjective.  Similarly we say that a generalized Floer trajectory $\ul{v}$ is {\em regular} if the associated linearized operator $D_{\ul{v}}$ is surjective.  
\end{definition}
The strategy of proof for gluing is standard:

\begin{itemize}

\item[1:] Define a {\em pre-glued curve}, $(r_R, \ulu_R)$, for gluing lengths $R> >0$.  

\item[2:] Compute that $\|(\ol{\partial} - \ul{\nu})(r_R, \ulu_R)\|_{0,p} \leq \varepsilon(R)$ where $\varepsilon(R) \to 0$ as $R\to \infty$.  

\item[3:] Show that $D_{\sS, r_R, \ulu_R}$ is surjective, and construct a right inverse $Q_R$ by first constructing an approximate right inverse, $T_R$, with the same image. 

\item[4:] Show there is a uniform bound $\|Q_R\| \leq C$ for sufficiently large $R$.

\item[5:] Show that for each $R$, the function $\F_{\sS, r_R, \ulu_R}$ satisfies a quadratic estimate 
\begin{equation}\label{quad_est}
\| d\F_{\sS,r_R,\ulu_R}(\rho, \ulxi) - D_{\sS,r_R,\ulu_R}\| \leq c (|\rho| + \|\ulxi\|_{W^{1,p}})
\end{equation}
with the constant $c$ independent of $R$.
\end{itemize}

These steps are the content of Sections \ref{pregluing_section} through \ref{quad_est_section} respectively. In Section \ref{gluing_map} these ingredients are used to define, with the help of an implicit function theorem, a gluing map associated to each regular tuple, and we show that the image of the gluing map is contained in the one-dimensional component of the relevant moduli space of quilted disks, $\M_{d,0}(\ul{x}_0, \ldots, \ul{x}_d)^1$.  In Section \ref{surjectivity_section} we show that the gluing map is surjective, in the sense that if a pseudoholomorphic quilted disk $(r, \ulu) \in \M_{d,0}(\ul{x}_0, \ldots, \ul{x}_d)^1$ is in a sufficiently small Gromov neighborhood of the broken tuple, then it is in the image of the gluing map for that tuple.  

\subsection{Pregluing}\label{pregluing_section}

There are three types of pre-gluing to consider - two types arise from the two types of facet of codimension one in the boundary of $\ol{\RR}^{d,0}$ (corresponding to whether the inner circle has bubbled through or not), and the other type arises from a Floer trajectory breaking off. 
\begin{definition}Given a {\em gluing parameter} $\delta > 0$,  define $R(\delta) := - \log(\delta)$ to be the {\em gluing length} corresponding to $\delta$.  So $R(\delta) \to \infty$ as $\delta \to 0$. 
\end{definition}

\subsection*{Type 1} Assume that we have a regular pair 
\bea
(r_1, \ulu_1) & \in & \M_{{d_1},0}(\ul{x}_0, \ldots, \ul{x}_{i-1}, \ul{y}, \ul{x}_{i+1}, \ldots, \ \ldots \ul{x}_{d_1})^0, \ \mbox{where}\ i \in \{1, \ldots, d\},\\
(r_2, \ulu_2) & \in & \M_{d_2}(\ul{y}, \ul{z}_1, \ldots, \ul{z}_{d_2})^0.
\eea 
The surface parametrized by $r_1$ is a quilted disk, and the surface parametrized by $r_2$ is a marked, unquilted disk.  A marked point labeled $\zeta^-$ on $r_1$ is identified with a marked point $\zeta^+$ on $r_2$, identifying the pair $r_1 \# r_2$ with a nodal quilted disk.  Along the strip-like ends labelled by $\zeta^\pm$, $\ulu_1$ and $\ulu_2$ converge exponentially to $\ul{y}$.

\begin{figure}[ht]
\center{\includegraphics[height=3in]{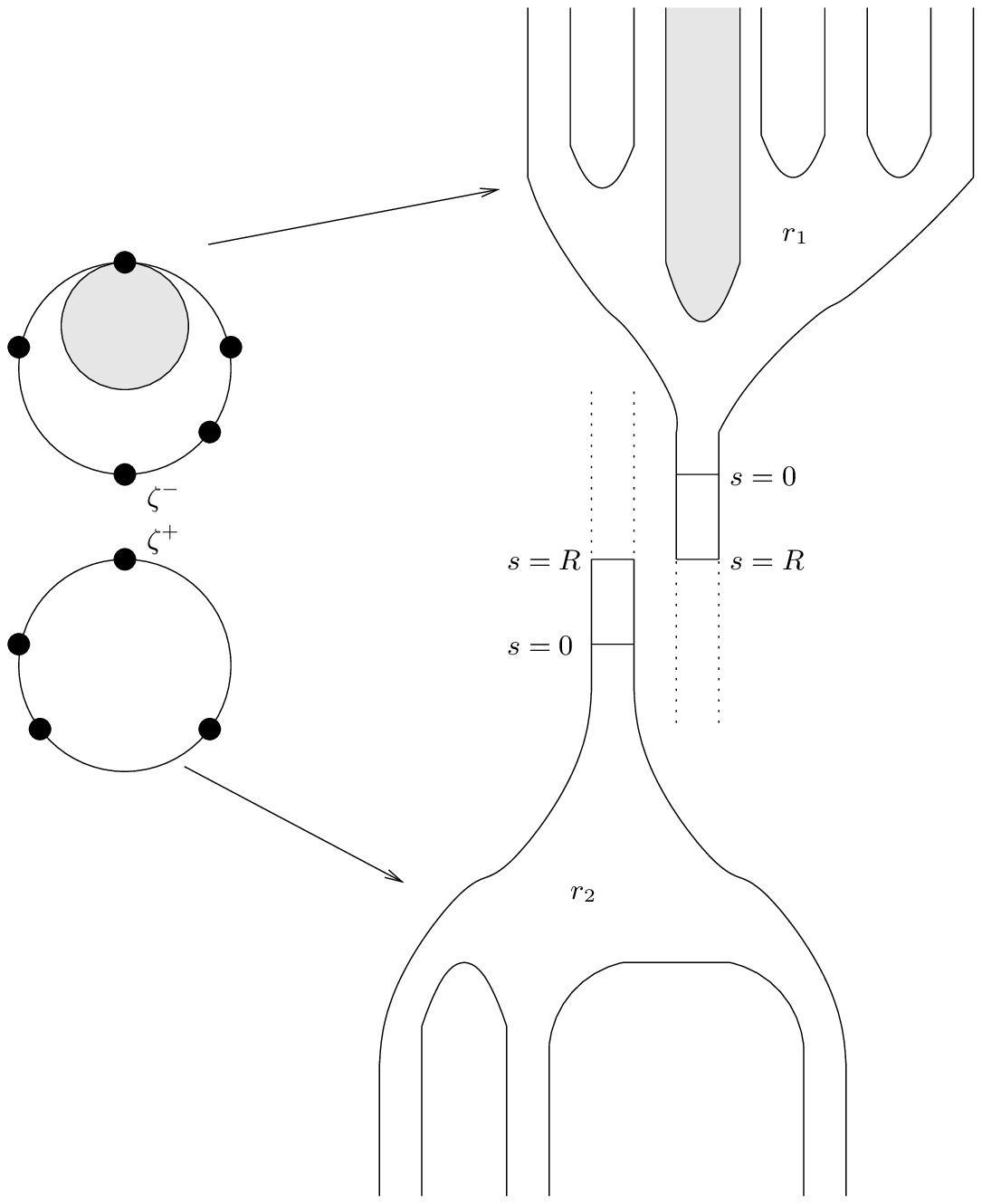} \includegraphics[height=3in]{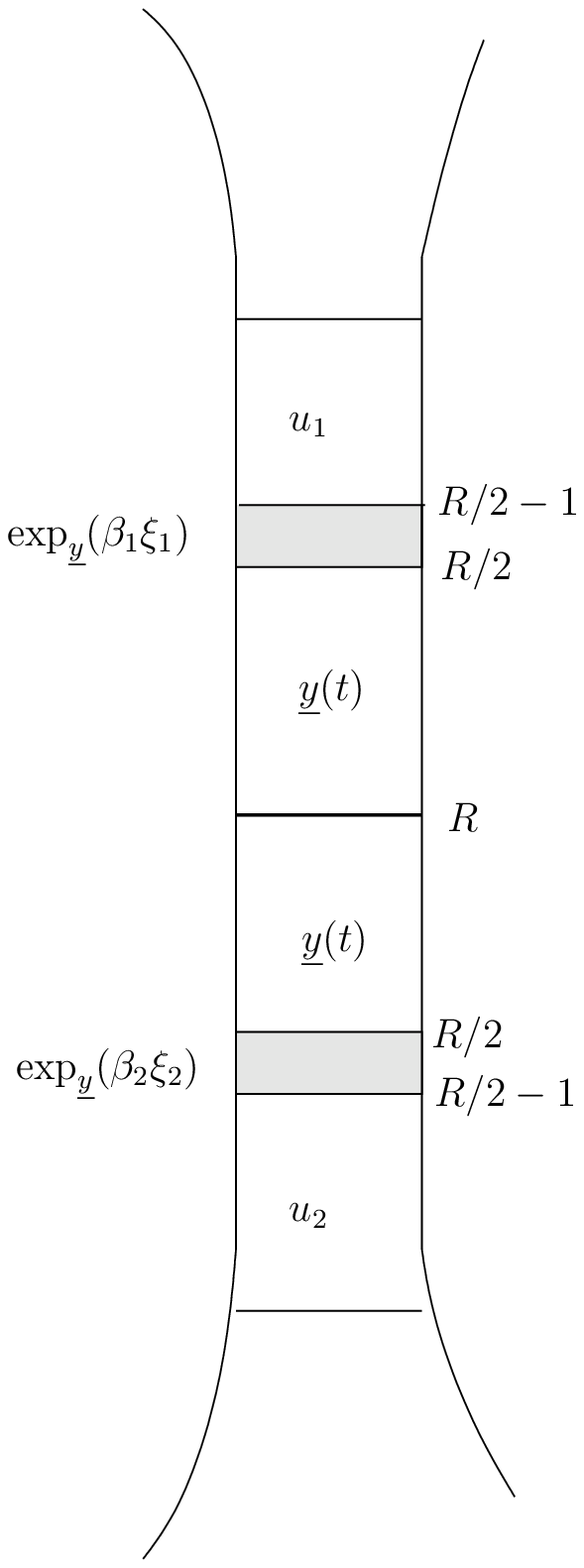}}
\caption{Case 1 of pregluing (left), and gluing $\ulu_1$ and $\ulu_2$ along the neck(right).}\label{gluing_neck}\label{case1}
\end{figure}

\subsubsection*{The quilted surface $r_R:=r_1 \#_R r_2$:} Truncate the surface $\sS_{r_1}$ along the striplike end labeled by $\zeta^-$ at $s=R$, and truncate the surface $\sS_{r_2}$ along the striplike end labeled by $\zeta^+$ at $s=R$, then identify the two truncated surfaces along $s=R$. Explicitly, one identifies $\epsilon_{\zeta^+}(R,t) \sim \epsilon_{\zeta^-}(R,1-t)$; see Figure \ref{case1}.

Let $z = (s,t)$ denote the variables on the striplike end of $\sS_{r_1}$, where $s \in [0,\infty)$ and $t \in [0,1]$.  For $s >>0$,  $\ulu_1(s, t)$ is exponentially close to $\ul{y}(t) \in \I(\ul{L}, \ul{L}^\prime)$, and we 
define $\ul{\xi}_1(s,t) \in T_{\ul{y}(t)} \ulM$ by
\[
\ul{\exp}^Q_{\ul{y}(t)}(\ulxi_1(s,t)) = \ulu_1(s,t)
\]
where $\ul{\exp}^Q$ is the tuple of $t$-parametrized quadratically corrected exponential maps for the pair $\ul{L}, \ul{L}^\prime$, such that $\ul{L}$ is totally geodesic when $t=0$, and $\ul{L}^\prime$ is totally geodesic when $t=1$.  Similarly, we can define $\ul{\xi}_2(s,t) \in T_{\ul{y}(t)}\ulM$, for very large $s$, by
\[
\ul{\exp}^Q_{\ul{y}(t)}(\ul{\xi}_2(s,t)) = \ulu_2(s,t).
\]
Let $\beta: \R \to [0,1]$ be a smooth cut-off function such that $\beta(s) = 1, s \leq -1 $, and $\beta(s) = 0, s \geq 0$.  Define a pair of intermediate approximate pseudoholomorphic quilted surfaces $(r_i, \ulu_i^R)$ for $i = 1,2$, by
\[
\ulu_i^R(z) = \left\{\begin{array}{ll}
                 \ulu_1(z), & z \in \sS_{r_i}\setminus \epsilon_{\zeta^\pm}(s \geq  R/2)\\
                 \ul{\exp}^Q_{\ul{x}}(\beta(s-R/2) \ul{\xi}_i(s,t)), &  R/2 -1 \leq s \leq R/2 \\
                  \ul{y}(t), & s\geq R/2.
                 \end{array}\right.
\]

The pre-glued map $\ulu_1 \#_R \ulu_2 : \sS_{r_1\#_R r_2} \to \ul{M}$ is (Figure \ref{gluing_neck})
\[
\ulu_1\#_R \ulu_2 (z) = \left\{ \begin{array}{ll}
                 \ulu_1^R(z), & z \in \sS_{r_1} \setminus \epsilon_{\zeta^-}((R, \infty)\times [0,1])\\
                 \ulu_2^R(z), & z \in \sS_{r_2} \setminus \epsilon_{\zeta^+}((R, \infty)\times [0,1]).
\end{array}\right.
\]

\subsection*{Type 2} Consider a collection 
\bea
(r_0, \ulu_0) & \in & \M_{k}(\ul{x}_0, \ul{y}_1, \ldots, \ul{y}_{k})^0\\
(r_1, \ulu_1) & \in & \M_{d_1, 0}(\ul{y}_1, \ul{z}_{1}, \ul{z}_2, \ldots, \ul{z}_{d_1})^0\\
(r_2, \ulu_2) & \in &  \M_{d_2, 0}(\ul{y}_2, \ul{z}_{d_1+1}, \ul{z}_{d_1+2}, \ldots, \ul{z}_{d_1+d_2})^0\\
\ldots& & \\
(r_k, \ulu_k) & \in &  \M_{d_k, 0}(\ul{y}_{k}, \ul{z}_{d_1+d_2 + \ldots + d_{k-1}+1}, \ul{z}_{d_1+d_2 + \ldots + d_{k-1}+2}, \ldots,  \ul{z}_{d_1 + d_2 + \ldots d_{k-1}+d_k})^0
\eea
of regular pseudoholomorphic quilts. Let $\zeta^{(0)}, \zeta^{(1)}, \ldots, \zeta^{(k)}$  label the striplike ends of $\sS_{r_0}$ where the quilt map $\ul{u}_0$ converges to $\ul{x}_0,\ul{y}_1, \ldots, \ul{y}_k$ respectively.  For $i = 1, \ldots, k$, each quilt $\sS_{r_i}$ has a distinguished strip-like end $\eta_i^{(0)}$ on which $\ul{u}_i: \sS_{r_i} \to \ul{M}$ converges to $\ul{y}_i$. 

\subsubsection*{The quilted surface $r_R:=r_0 \#_R \{r_1, \ldots, r_k\}$:} Strictly speaking a gluing parameter $\delta >0$ and $R = -\log(\delta)$ gives rise to a tuple of admissible gluing lengths, $(R + a_1, \ldots, R+a_k)$, where the $a_1, \ldots, a_k$ are determined by the coordinates of $r_0, r_1, \ldots, r_k$ in their respective moduli spaces.  However, for sufficiently large $R$ we can shift the position of ``$s=0$" on the striplike ends so as to take the gluing lengths along all the striplike ends to be $R$.  Truncate $\sS_{r_0}$ along each of the striplike ends labeled $\zeta^{(1)}, \ldots, \zeta^{(d)}$ at $s = R$, and truncate the surfaces $\sS_{r_1}, \ldots, \sS_{r_k}$ along their respective striplike ends labeled $\eta_1^{(0)}, \ldots, \eta_k^{(0)}$ at $s = R$,  then identify truncations labeled by pairs $\zeta^{(i)}, \eta_i^{(0)}$ along $s = R$ (Figure \ref{pregluing2}).
\begin{figure}[h]
\center{\includegraphics[height=4in]{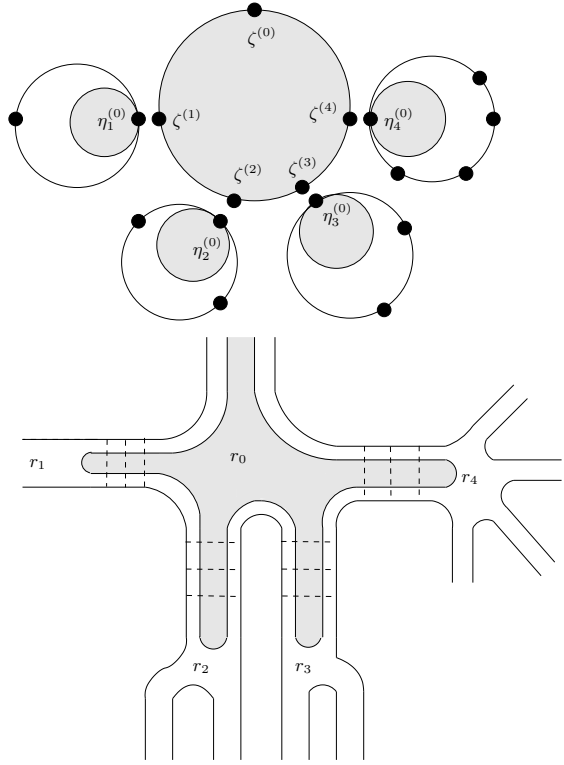}}
\caption{Pregluing in Case 2.}\label{pregluing2}
\end{figure}

\subsubsection*{The approximate pseudoholomorphic map} 
Let $z = (s,t)$ be coordinates along the striplike ends.  For $s >> 0$, define $\ul{\xi}_{i}(s,t) \in T_{\ul{y}_i(t)}\ul{M}$ by
\[
\ul{\exp}^Q_{\ul{y}_i(t)}(\ulxi_i(s,t)) = \ulu_0(\epsilon_{\zeta^{(i)}}(s,t)).
\]
Similarly, for $i = 1, \ldots, k$ and $s>>0$ define $\ulxi_i^{(0)}(s,t) \in T_{\ul{y}_i(t)}\ul{M}$ by 
\[
\ul{\exp}^Q_{\ul{y}_i(t)}(\ul{\xi}_i^{(0)}(s,t)) = \ulu_i(\epsilon_{\eta_i^{(0)}}(s,t)).
\]
We introduce the intermediate approximate pseudoholomorphic quilted surfaces $(r_0, \ul{u}_0^R), (r_1, \ul{u}_1^R), \ldots, (r_k, \ul{u}_k^R)$ defined as follows. The map $\ul{u}_0^R: \sS_{r_0} \to \ul{M}$ is defined piecewise by
\[
\ulu_0^R(z) = \left\{\begin{array}{ll}
                 \ulu_0(z), & z \in \sS_{r_0}\setminus \bigcup\limits_{i=1}^k \epsilon_{\zeta^{(i)}}(s \geq  R/2-1)\\
                 \ul{\exp}^Q_{\ul{y}_i(t)}(\beta(s-R/2) \ul{\xi}_i(s,t)), & z = \epsilon_{\zeta^{(i)}}(s,t), s \in [R/2 -1, R/2] \\
                  \ul{y}_i(t), & z = \epsilon_{\zeta^{(i)}}(s,t), s \geq R/2.
                 \end{array}\right.
\]
For $i = 1, \ldots, k$ the intermediate maps $\ul{u}_i^R: \sS_{r_i} \to \ul{M}$ are defined by
\[
\ulu_i^R(z)= \left\{\begin{array}{ll}
           \ulu_i(z), & z \in \sS_{r_i}\setminus \epsilon_{\eta_i^{(0)}}(s \geq  R/2-1)\\
           \ul{ \exp}^Q_{\ul{y}_i(t)}(\beta(s-R/2) \ul{\xi}_i(s,t)), & z = \epsilon_{\eta_i^{(0)}}(s,t), s \in [R/2 -1, R/2]\\
\ul{y}_i(t), & z = \epsilon_{\eta_i^{(0)}}(s,t), s \geq R/2.
\end{array}\right.
\]
Then the pre-glued quilt map is
\[
\ulu_0 \#_R \{\ulu_1, \ldots, \ulu_k\}(z) = \left\{ \begin{array}{ll}
\ulu_0^R(z), & z \in \sS_{r_0}\setminus \bigcup\limits_{i=1}^k \epsilon_{\zeta^{(i)}}(s \geq  R)\\
\ulu_i^R(z), & z \in \sS_{r_i}\setminus \epsilon_{\eta_i^{(0)}}(s \geq R).
\end{array}\right.
\]

\subsection*{Type 3} (A Floer trajectory breaks off.) Consider a pair 
\bea
(r_1, \ulu_1)  \in  \M_{d, 0}(\ul{x}_0, \ul{x}_1, \ldots, \ul{x}_d)^0, \ \ \ 
\ul{v}  \in   \widetilde{\M}(\ul{x}_i, \ul{y})^0
\eea  
of a regular pseudoholomorphic quilted disk and a regular Floer trajectory. Assume without loss of generality that $\lim_{s\to-\infty} \ul{v}(s,t) =  \ul{x}_i(t), \ \ \ \lim_{s\to \infty} \ul{v}(s,t) = \ul{y}(t)$, 
and let $\zeta$ label the striplike end of $\sS_{r_1}$ for which $\lim_{s\to\infty} \ul{u}_1(\epsilon_{\zeta}(s,t)) = \ul{x}_i(t)$. The Floer trajectory $\ul{v}$ is defined only up to an $\R$ translation, so fix a parametrization.

\subsubsection*{The quilted surface} In this case, $r_R = r_1$. 

\subsubsection*{The approximate pseudoholomorphic map} For $s >> 0$ we know that $\ul{u}_1(\epsilon_\zeta(s,t))$ and $v(-s, t)$ are exponentially close to $\ul{x}_i(t)$.  For such $s$, define $\ul{\xi}(s,t), \ul{\eta}(-s,t) \in T_{\ul{x}_i(t)}\ul{M}$ by the conditions that
\[
\ul{\exp}^Q_{\ul{x}_i(t)}(\ul{\xi}(s,t))  =  \ul{u}_1(\epsilon_{\zeta}(s,t)),  \ \ \ \ 
\ul{\exp}^Q_{\ul{x}_i(t)}(\ul{\eta}(-s,t))  =  \ul{v}(-s,t).
\] 
Define an approximate pseudoholomorphic quilt and an approximate trajectory by
\bea
\ulu_1^R(z) & = & \left\{\begin{array}{ll}
           \ulu_1(z), & z \in \sS_{r_1}\setminus \epsilon_{\zeta}(s \geq  R/2-1)\\
            \ul{\exp}^Q_{\ul{x}_i(t)}(\beta(s-R/2) \ul{\xi}(s,t)), & z = \epsilon_{\zeta}(s,t), s \in [R/2 -1, R/2]\\
\ul{x}_i(t), & z = \epsilon_{\zeta}(s,t), s \geq R/2.
\end{array}\right.\\
\ul{v}^R(s,t) & = &  \left \{ \begin{array}{ll}
                   \ul{v}(s-2R,t), & s \geq 3R/2 + 1\\
                    \ul{\exp}^Q_{\ul{x}_i(t)}(\beta(-s + 3R/2 )\ul{\eta}(s - 2R,t)), & s \in [3R/2,3R/2 + 1]\\
                    \ul{x}_i(t), & s \leq 3R/2. 
\end{array}\right.
\eea
Then the preglued quilt map is
\[
\ul{u}_R(z) = \left\{\begin{array}{ll}
              \ul{u}_1^R(z), & z \in \sS_{r_1}\setminus \epsilon_{\zeta}(s \geq R)\\
              \ul{v}^R(s,t), & z \in \epsilon_{\zeta}(s\geq R).
\end{array}\right.
\]

\subsection{Estimates for the preglued quilts.}

\begin{proposition}
For sufficiently large $R_0 \geq 0$, there is a monotone decreasing function $\epsilon: [R_0, \infty) \to [0,\infty)$ such that 
\[
\| (\ol{\partial} - \ul{\nu})(r_R, \ulu_R)\|_{0,p} \leq \epsilon(R)
\]
and $\epsilon(R) \to 0$ as $R \to \infty$. 
\end{proposition}
\begin{proof}
We estimate it for each of the three types of pregluing separately. \\
\subsection*{Type 1.}
Let $(\widetilde{\ul{J}}, \widetilde{\ul{K}})$ denote the {\em approximate} perturbation datum on the preglued quilt $\sS_{r_R}$ that is inherited from the perturbation data $(\ul{J}_1, \ul{K}_1)$ and $(\ul{J}_2, \ul{K}_2)$ for $\sS_{r_1}$ and $\sS_{r_2}$ respectively via the pregluing procedure.  Let $(\ul{J}(r_R), \ul{K}(r_R))$ be the perturbation datum on $\sS_{r_R}$ that comes from the universal choice of perturbation data over the family of quilted surfaces parametrized by the multiplihedron. In general $(\widetilde{\ul{J}}, \widetilde{\ul{K}})$ and $(\ul{J}(r_R), \ul{K}(r_R))$ are not the same, but the assumption of consistency implies that for large values of $R$,  the data agree on the ``thin'' part of $\sS_{r_R}$, while on the complement of the thin part, which consists of two compact components coming from $\sS_{r_1}$ and $\sS_{r_2}$, $(\ul{J}(r_R), \ul{K}(r_R))$ converges uniformly to $(\ul{J}_1, \ul{K}_1)$ or $(\ulJ_2,\ulK_2)$ respectively as $R\to \infty$.

Given that
\bea
 \ol{\partial}(r_R, \ulu_R) - \ul{\nu}(r_R, \ulu_R) & = & \frac{1}{2}(d\ulu_R + \ulJ(r_R) \circ d\ulu_R \circ j(r_R) ) \\
 & & - \frac{1}{2}(\ulY(r_R) + \ulJ(r_R)\circ Y(r_R) \circ j(r_R))
\eea
we will estimate this on different parts of the preglued surface $\sS_{r_R}$. 

Consider a striplike end $Z \subset \sS_{r_R}$.  It corresponds to a striplike end on either $\sS_{r_1}$ or $\sS_{r_2}$, that was {\em not} truncated in the pregluing step; let us denote this striplike end by $Z$ too, where $Z \subset \sS_{r_i}$ for $i = 1$ or $2$.  Since $j(r_R, z), \ulJ(r_R, \ulu(z), z)$ and $\ulY(r_R, z)$ are independent of $R$ and are the same as the corresponding data on $\sS_{r_1}$, 
\bea
(\ol{\partial} - \nu)(r_R, \ulu_R)\Big\lvert_{Z}  =   (\ol{\partial} - \ul{\nu})(r_1, \ulu_1)\Big\lvert_Z = 0
\eea
since we assumed that $(\ol{\partial} - \ul{\nu})(r_1, \ulu_1) = 0$.  This argument proves that $(\ol{\partial} - \ul{\nu})(r_R, \ulu_R)$ is zero on {\em all} of the striplike ends of the glued surface $\sS_{r_R}$.  

Next, for $i = 1, 2$ let $S_i$ denote the complement of the striplike ends on $\sS_{r_i}$, and let $S_i$ also denote its image in $\sS_{r_R}$ after pregluing.  Note that $S_i$ is compact, and that $\ulu_R\Big\lvert_{S_i} = \ulu_i\Big\lvert_{S_i}$. Hence,
\bea
(\ol{\partial} - \ul{\nu})(r_R, \ulu_R)\Big\lvert_{S_i} =  (\ol{\partial} - \ul{\nu})(r_R, \ulu_i)\Big\lvert_{S_i}
& = & \frac{1}{2}[d \ulu_i + \ulJ(r_R, \ulu_i) d \ulu_i j(r_R)] \Big\lvert_{S_i}\\
& & \ \ - \frac{1}{2}[\ulY(r_R, \ulu_i) + \ulJ(r_R, \ulu_i) \ulY(r_R,\ulu_i) j(r_R)]\Big\lvert_{S_i}.
\eea
Since we know that $(\ol{\partial} - \ul{\nu})(r_i, \ulu_i) = 0$, and that $j(r_R), J(r_R, \ulu_i)$ and $\ulY(r_R, \ulu_i)$ converge uniformly to $j(r_i), \ulJ(r_i, \ulu_i)$ and $\ulY(r_i, \ulu_i)$ on $S_i$, it follows that for sufficiently large $R$ there is a monotone decreasing function $\epsilon_1(R) \to 0$ as $R\to \infty$ such that the uniform pointwise estimate $|(\ol{\partial} - \ul{\nu})(r_R, \ulu_R)(z)| \leq \epsilon_1(R)$ holds for all $z \in S_i$.

Finally, consider the neck of $\sS_{r_R}$ along which the pregluing took place.  Let $Z_1$ and $Z_2$ denote the striplike ends of $\sS_{r_1}$ and $\sS_{r_2}$ that were truncated along $s=R$; and by slight abuse of notation let $Z_1$ and $Z_2$ also denote the images of the truncations after pregluing.  By symmetry it is enough to consider what happens on $Z_1$. There are three regions to consider,
\bea
 0 \leq s \leq R/2 - 1, \ \ \  R/2 -1 \leq s \leq R/2, \ \ \ R/2 \leq s \leq R.
\eea 
\noindent On the region $0 \leq s \leq R/2 - 1$, $ (\ol{\partial} - \ul{\nu})(r_R, \ulu_R)  = (\ol{\partial} - \ul{\nu})(r_1, \ulu_1)=0$.   \\
\noindent On the region $R/2 \leq s \leq R$, we have that $\ulu_R(s,t) = \ulx_0(t)$, and so
\bea
d\ulx_0(t) + \ulJ(t, \ulx_0(t)d\ulx_0(t) j - X_{\ulH_t} - \ulJ(t,\ulx_0(t))X_{\ulH_t} j & = & 0
\eea 
because $\partial_s \ulx_0(t) = 0$ and $\partial_t \ulx_0(t) = X_{\ulH_t}$ together imply that $d\ulx_0(t)-X_{\ulH_t} = 0$. \\
\noindent On the region $R/2 - 1\leq s \leq R/2$, since $\ulu_1(s,t)$ converges exponentially to $\ul{y}(t)$ as $s \to \infty$, we know that $|\ul{\xi}(s,t)|$ becomes exponentially small in $s$. Therefore, $|\beta(s)\ul{\xi}_1(s,t)|$ is also exponentially small in $s$. Now, $d \ul{y}(t) - X_{\ulH_t}(\ul{y}(t)) = 0$, so
\bea
|d\ul{\exp}^Q_{\ul{y}(t)}(\beta(s) \ulxi_1(s,t)) - X_{\ulH_t}(\exp_{\ul{y}(t)}(\beta(s) \ulxi_1(s,t))|  \leq  |d\exp^Q_{\ul{y}(t)}(\beta(s)\ulxi_1(s,t)) - d\ul{y}(t)| \\
 \ + |X_{\ulH_t}(\ul{y}(t) - X_{\ulH_t}(\ul{\exp}^Q_{\ul{y}(t)}(\beta(s)\ulxi_1(s,t)))|
\eea
and since $\ul{\exp}^Q_{\ul{y}(t)}(\beta(s)\ulxi_1(s,t))$ is exponentially close to $\ul{y}(t)$, there is a monotone decreasing function $\epsilon_2(R) \to 0$ as $R \to \infty$ such that 
\bea
|d\ul{\exp}^Q_{\ul{y}(t)}(\beta(s) \ul{\xi}_1(s,t)) - X_{\ulH_t}(\ul{\exp}^Q_{\ul{y}(t)}(\beta(s) \ul{\xi}_1(s,t))| \leq \epsilon_2(R)
\eea
uniformly in $t$ for all $s \geq R/2 - 1$.

Using all these estimates, we have
\bea
& & \left( \int\limits_{\sS_{r_R}} |(\ol{\partial}-\ul{\nu})(r_R, \ulu_R)|^p \  \dvol_{\sS_{r_R}}  \right)^{1/p}
 \leq  \left( \int\limits_{S_1} |(\ol{\partial}-\ul{\nu})(r_R, \ulu_R)|^p \  \dvol_{S_1}  \right)^{1/p} \\
& & + \left( \int\limits_{S_2} |(\ol{\partial}-\ul{\nu})(r_R, \ulu_R)|^p \  \dvol_{S_2}  \right)^{1/p} 
 + \left( \int\limits_{0}^1 \int\limits_{R/2-1}^{R/2} |(\ol{\partial}-\ul{\nu})(r_R, \ulu_R)|^p \  ds\ dt \right)^{1/p}\\
& &  \ \ \ \ \ \ \ \ \ \ \ \ \ \leq  \epsilon_1(R) \vol(S_1)^{1/p} + \epsilon_1(R) \vol(S_2)^{1/p} + \epsilon_2(R).
\eea
Here $\vol_{S_i}$ is the volume of the compact subset $S_i \subset \sS_{r_i}$ with respect to a fixed volume form on $\sS_{r_i}$.

\subsection*{Type 2} The estimate for Type 2 of the pregluing construction is very similar to that of Type 1.  Repeating the same arguments on the corresponding parts of $\sS_{r_R}$ leads to an estimate
\bea
\|\F_{\sS, r_R, \ulu_R}(0,0)\|_{0,p,R}   \leq  \epsilon_1(R)\sum\limits_{i=0}^k \vol(S_i)^{1/p} + k \epsilon_2(R).
\eea

\subsection*{Type 3} This type of pregluing is slightly different from the previous two types. Let $Z$ denote the striplike end of $\sS_{r_1}$ on which the pregluing takes place. The perturbation data is fixed, so on the complement of $Z$, $
(\ol{\partial} - \ul{\nu})(r_1, \ulu_R)  = (\ol{\partial} - \ul{\nu})(r_1, \ulu_1) = 0.$  Thus $(\ol{\partial} - \ul{\nu})(r_1, \ulu_R)$ is only supported on $Z$. 

Write $z = (s,t)$.  For $s \in [0, R/2-1], (\ol{\partial} - \ul{\nu})(r_1, \ulu_R) = (\ol{\partial} - \ul{\nu})(r_1, \ulu_1) = 0$.  For
$s \in [R/2, 3R/2]$, $(\ol{\partial} - \ul{\nu})(r_1, \ulu_R)  =  (\partial_t \ul{y}(t) - X_{\ulH_t}(\ul{y}(t))^{0,1}
=  0$ since $\ul{y}(t)$ is the Hamiltonian flow of $\ul{H}_t$. For $s \in [3R/2 + 1, \infty)$,
$(\ol{\partial} - \ul{\nu})(r_1, \ulu_R) = \ol{\partial}\ul{v}(s-2R,t) - X_{\ulH_t}(\ul{v}(s-2R,t) = 0$,
as $\ul{v}$ is a Floer trajectory.

Therefore $(\ol{\partial} - \ul{\nu})(r_1, \ulu_R)$ is only supported on the region where $s \in [R/2-1, R/2]$ or $s \in [3R/2, 3R/2+1]$. On the first part, for sufficiently large $R$, $|\ulxi(s,t)|$ is exponentially small in $R$ when $s \geq R/2 - 1$, and so for the same reasons as in the previous two calculations we can find a monotone decreasing $\epsilon_2(R) \to 0$ as $R \to \infty$ such that 
\[
|d\ul{\exp}^Q_{\ul{y}(t)}(\beta(s-R/2) \ulxi(s,t)) - X_{\ulH_t}(\ul{\exp}^Q_{\ul{y}(t)}(\beta(s-R/2) \ulxi(s,t))| \leq \epsilon_2(R)
\]
uniformly in $t$, whenever $s \in [R/2-1, R/2]$.  Similarly, for sufficiently large $R$, $|\ul{\eta}(s-2R,t)|$ is exponentially small and we can find an $\epsilon_2(R)$ as above such that 
\[
|d\ul{\exp}^Q_{\ul{y}(t)}(\beta(-s+3R/2)\ul{\eta}(s-2R,t)) - X_{H_t}(\ul{\exp}^Q_{\ul{y}(t)}(\beta(-s+3R/2)\ul{\eta}(s-2R,t)| \leq \epsilon_3(R))
\]
uniformly in $t$, whenever $s \in [3R/2, 3R/2 + 1]$. Putting the estimates together gives
\bea
\|\F_{\sS, r_1, \ulu_R}(0,0)\|_{0,p,R} & \leq&  \left( \int\limits_{0}^1 \int\limits_{R/2-1}^{R/2}  |(\ol{\partial} - \ul{\nu})(\ulu_R)|^p \ ds\ dt \right)^{1/p}
 + \left( \int\limits_{0}^1 \int\limits_{3R/2}^{3R/2+1}  |(\ol{\partial} - \ul{\nu})(\ulu_R)|^p \ ds\ dt \right)^{1/p}\\
&\leq &  \epsilon_2(R) + \epsilon_3(R).
\eea

\end{proof}

\subsection{Constructing a right inverse}

We construct an approximate right inverse for the linearized operator of the preglued surface and curve, then show that it is sufficiently close to being a right inverse that an actual right inverse can be obtained from it via a convergent power series.  We treat the three types of pregluing separately.

\subsection*{Type 1}
Assume that $(r_1, \ulu_1)$ and $(r_2, \ulu_2)$ are regular and isolated, for the perturbation data $(\ulJ_1, \ulK_1)$ and $(\ulJ_2, \ulK_2)$ respectively.  We will show that $D_R:= D_{\bS, r_1 \#_R r_2, \ulu_1 \#_R \ulu_2}$, defined with perturbation datum $(\ulJ(r_R), \ulK(r_R))$, is also a surjective Fredholm operator, with a right inverse $Q_R$ that is uniformly bounded for sufficiently large $R$. 

Recall the intermediate quilt maps $(r_1, \ulu_1^R), (r_2, \ulu_2^R)$.  For large $R$, the functions $\ulu_i^R, i = 1, 2$ are $W^{1,p}$-small perturbations of $\ulu_i$.  The perturbation data $(\ulJ(r_R), \ulK(r_R))$ are just the Floer data along the neck of $\sS_{r_R}$.  Let $(\ulJ_i(r_R), \ulK_i(r_R))$ denote the perturbation datum on $\sS_{r_i}$ given by the restriction of $(\ulJ(r_R), \ulK(r_R))$ to the truncation of $\sS_{r_i}$, extended trivially over the rest of the striplike end of $\sS_{r_i}$ by the Floer data.   The assumption of consistency implies that for large $R$, the data $(\ulJ_i(r_R), \ulK_i(r_R))$ is a compact perturbation of $(\ulJ_i, \ulK_i)$.   

The property of being a surjective Fredholm operator is stable under $W^{1,p}$-small and compact perturbations, hence for sufficiently large $R$ the pairs $(r_i, \ulu_i^R)$ are regular with respect to $(\ulJ_i(r_R), \ulK_i(r_R))$.  Let $Q_{i, R}$ denote a right inverse for the linearized operator $D_{\sS, r_i, \ulu_i^R}$.  We observe that $Q_{i,R}$ is actually a left {\em and} right inverse, since by assumption $\ker D_{\sS, r_i, \ulu_i^R}$ has dimension 0, so $D_{\sS, r_i, \ulu_i^R}$ is an isomorphism.   

Returning to the linearized operator
\[
D_R: T_{r_R}\RR^{d_1+d_2, 0} \times W^{1,p}(\sS_{r_R}, \ulu_R^* T\ulM) \to L^p(\sS_{r_R}, \Lambda^{0,1}\otimes_{J}\ulu_R^*T\ulM)
\]
we construct an approximate inverse 
\[
T_R: L^p(\sS_{r_R}, \Lambda^{0,1}\otimes_{J}\ulu_R^*T\ulM) \to  T_{r_R}\RR^{d_1+d_2, 0} \times W^{1,p}(\sS_{r_R}, \ulu_R^* T\ulM).
\] 

Let $\ul{\eta} \in L^p(\sS_{r_R}, \Lambda^{0,1}\otimes_{J}\ulu_R^*T\ulM)$. Set 
\bea
\ul{\eta}_1(z)  = \left\{\begin{array}{cc} \ul{\eta}(z), & \ z \in r_1^R,    \\
            0, & \ z \in r_1\setminus r_1^R
            \end{array}, \right.\ \ 
\ul{\eta}_2(z)= \left\{.\begin{array}{cc} \ul{\eta}(z), &\ z \in r_2^R,    \\
           0, &\ z \in r_2\setminus r_2^R
           \end{array}\right.
\eea
The abrupt cut-off does not matter because the norms involving $\ul{\eta}_1$ and $\ul{\eta}_2$ are just  $L^p$ norms, so 
\bea
\ul{\eta}_1 \in  L^p(\sS_{r_1}, \Lambda^{0,1}\otimes_{J}(\ulu_1^R)^*T\ulM),\ \ \ 
\ul{\eta}_2  \in  L^p(\sS_{r_2}, \Lambda^{0,1}\otimes_{J}(\ulu_2^R)^*T\ulM).
\eea
Using $Q_{1,R}$ and $Q_{2,R}$ we get 
\bea
Q_{1,R} \ul{\eta}_1 =:  (\tau_1, \ulxi_1) & \in & T_{r_1}\RR^{d_1+1, 0} \times W^{1,p}(\sS_{r_1}, (\ulu_1^R)^* T\ulM)\\ 
Q_{2,R} \ul{\eta}_2 =:  (\tau_2, \ulxi_2) & \in & T_{r_2}\RR^{d_2+1} \times W^{1,p}(\sS_{r_2}, (\ulu_2^R)^* T\ulM).
\eea 
Our final step is to glue these into a single element of $T_{r_R}\RR^{d_1+d_2, 0} \times W^{1,p}(\sS_{r_R}, \ulu_R^*T\ulM)$.  For large $R$, local gluing charts near the boundary of $\RR^{d_1+d_2,1}$ give an isomorphism 
\[
T_{r_R}\RR^{d_1+d_2,0} \cong T_{r_1}\RR^{d_1+1, 0} \oplus T_{r_2}\RR^{d_2+1}\oplus T_{R}\R
\]
where the last component represents the gluing parameter.  Using this isomorphism we get a well-defined element
\bea
\tau_1 \#_R \tau_2 := (\tau_1, \tau_2, 0) \in T_{r_R}\RR^{d_1+d_2,0}.
\eea
Now fix a smooth cut-off function $\beta: \R_{\geq 0} \to [0,1]$ such that 
$\beta(s) = 1$ for $s \geq 1$ and $\beta(s) = 0$ for $s \leq 1/2$, and $0 \leq \dot{\beta} \leq 2$.  Let $\beta_R(s):= \beta(s/R)$.  
\[
\ulxi_1 \#_R\ \ulxi_2(z) := \left\{ \begin{array}{ll}
                     \ulxi_1(z) & \mbox{if} \ z \in \sS_{r_1}\setminus Z_1\\
                     \ulxi_1(s,t) + \beta_R(s) \ulxi_2(2R - s, 1-t) & \mbox{if} \ z \in Z_1, s \in [0,R]\\
                     \beta_R(s) \ulxi_1(2R-s,1-t) + \ulxi_2(s,t)& \mbox{if} \ z \in Z_2, s \in [0,R] \\
                     \ulxi_2(z) & \mbox{if} \ z \in S_{r_2}\setminus Z_2.  
\end{array}\right.
\]
To simplify notation, define $\beta_{1,R}: \sS_{r_R} \to \R$ by
\[
\beta_{1,R}(z):=\left\{ \begin{array}{ll}
                     1 & \mbox{if}\ z \in \sS_{r_1}\setminus Z_1,\\
                     1 & \mbox{if}\ z \in  Z_1, s \in [0, R],\\
                     \beta_R(s), & \mbox{if} \ z \in Z_2, s \in [0,R],\\
                     0 & \mbox{if} \ z \in \sS_{r_2}\setminus Z_2
                     \end{array}\right.
\]
and make a corresponding definition of $\beta_{2,R}$.  With this notation, we can write $\ulxi_1 \#_R \ulxi_2 = \beta_{1,R} \ulxi_1 + \beta_{2,R} \ulxi_2$.  We define
\[
T_R \ul{\eta} := (\tau_1 \#_R \tau_2, \ulxi_1 \#_R \ulxi_2).
\]   
Now we need to check that for all $\ul{\eta} \in (\E_{\sS_{r_R}})_{\ulu}$, for sufficiently large $R$,
\[
\|D_{\bS, r_R, \ulu_R}T_R \ul{\eta} - \ul{\eta}\|_{0,p,R} \leq \frac{1}{2}\| \ul{\eta}\|_{0,p,R}.
\]
Now,
\bea
D_{\sS, r_R, \ulu_R} T_R \ul{\eta} - \ul{\eta} & = & D_{\sS, r_R, \ulu_R}(\tau_1\#_R\tau_2, \ul{\xi}_1 \#_R \ul{\xi}_2) - \ul{\eta}\\
             & = & D^{(\ulu_R)}_{r_R} \tau_1\#\tau_2 + D^{r_R}_{(\ulu_R)} \ulxi_1\#_R \ulxi_2 -\ul{\eta}.
\eea
By construction $D^{(\ulu_R)}_{r_R}$ and $D^{(\ulu_1^R)}_{r_1}$ agree on the support of $\tau_1$, which is $\sS_{r_1}\setminus Z_1$.  Similarly $D^{(\ulu_2^R)}_{r_2}$ agrees with $D^{(\ulu_R)}_{r_R}$ on the support of $\tau_2$ which is $\sS_{r_2}\setminus Z_2$.  Hence,
\[
D^{(\ulu_R)}_{r_R} \tau_1\#_R \tau_2 = D^{(\ulu_1^R)}_{r_1} \tau_1 + D^{(\ulu_2^R)}_{r_2}\tau_2.
\]
Since $D^{(r_R)}_{\ulu_R}$ and $D^{(r_1)}_{\ulu_1^R}$ agree on the support of $\beta_{1,R} \ulxi_1$, we may write
\bea
D^{(r_R)}_{\ulu_R}  \beta_{1,R} \ulxi_1 & = & D^{(r_1)}_{\ulu_1^R} \beta_{1,R} \ulxi_1\\
                       & = & (\partial_s \beta_{1,R}) \ulxi_1 + \beta_{1,R} D_{\ulu_1^R}^{(r_1)} \ulxi_1,\\
D^{(r_R)}_{\ulu_R}\beta_{2,R} \ulxi_2 & = &  D^{(r_2)}_{\ulu_2^R} (\beta_{2,R}) \ulxi_2\\ 
& = & (\partial_s \beta_{2,R}) \ulxi_2 + \beta_{2,R} D_{\ulu_2^R}^{r_2} \ulxi_2.
\eea
By construction, $D^{(\ulu_i)}_{r_i}\tau_i + D^{(r_i)}_{\ulu_i^R} \ulxi_i = D^{(r_i)}_{\ulu_i^R} Q_{i, R} \ul{\eta}_i = \ul{\eta}_i$, where $Q_{i,R}$ is the inverse of $D_{\sS, r_i, \ulu_i^R}$.  So, 
\bea
D_{\bS, r_R, \ulu_R} T_R \ul{\eta} - \ul{\eta} & = & D^{(\ulu_R)}_{r_R} \tau_1\#_R \tau_2 + D^{(r_R)}_{\ulu_R} \ulxi_1\#_R\ulxi_2 -\ul{\eta}\\
                            & = & D^{(\ulu_1^R)}_{r_1} \tau_1 + D^{(\ulu_2^R)}_{r_2}\tau_2 + D^{(r_R)}_{\ulu_R}(\beta_{1,R} \ulxi_1 +  \beta_{2,R}\ulxi_2) - \ul{\eta}\\
                            & = & D^{(\ulu_1^R)}_{r_1} \tau_1 + D^{(\ulu_2^R)}_{r_2}\tau_2 + D^{(r_1)}_{\ulu_1^R} \beta_{1,R} \ulxi_1 + D^{(r_2)}_{\ulu_2^R} \beta_{2,R} \ulxi_2 - \ul{\eta}\\
                            & = & D^{(\ulu_1^R)}_{r_1} \tau_1 + D^{(\ulu_2^R)}_{r_2}\tau_2 + (\partial_s \beta_{1,R}) \ulxi_1 + \beta_{1,R} D_{\ulu_1^R}^{(r_1)} \ulxi_1 + (\partial_s \beta_{2,R}) \ulxi_1 \\
                            & & \ \ \ \ \ \ \ \ \ \ \ \ \ \ \ \ \ \ \ \ \ \ \ \ \ \ \ \ \ \ \ \ \ \ \ \ \ \ \ \ \ \ \ + \beta_{2,R} D_{\ulu_2^R}^{(r_2)} \ulxi_2 - \ul{\eta}\\
                            & = & \beta_{1,R}( D^{(\ulu_1)}_{r_1} \tau_1 + D^{r_1}_{ \ulu_1^R} \ulxi_1) + \beta_{2,R}( D^{(\ulu_2)}_{r_2}\tau_2 + D^{(r_2)}_{ \ulu_2^R} \ulxi_2) \\ 
                            & & \ \ \ \ \ \ \ \ \ \ \ \ \ \ \ \ \ \ \ \ \ \ \ \ \ \ \  \ + (\partial_s \beta_{1,R}) \ulxi_1 + (\partial_s \beta_{2,R}) \ulxi_2 - \ul{\eta}\\
                           & = & \beta_{1,R} \ul{\eta}_1 + \beta_{2,R} \ul{\eta}_2 +    (\partial_s \beta_{1,R}) \ulxi_1 + (\partial_s \beta_{2,R}) \ulxi_2    - \ul{\eta}\\
                           & = & \ul{\eta}_1 + \ul{\eta}_2 - \ul{\eta} + (\partial_s \beta_{1,R}) \ulxi_1 + (\partial_s \beta_{2,R}) \ulxi_2\\
                          & = & (\partial_s \beta_{1,R}) \ulxi_1 + (\partial_s \beta_{2,R}) \ulxi_2
                          \eea
as the support of $\ul{\eta}_i$ is precisely where $\beta_{i,R} = 1$, and $\ul{\eta}_1 + \ul{\eta}_2 = \ul{\eta}$.  We can find a $c > 0$ and an $R_0 \geq 0$ such that operator norms $\|Q_{i,R}\| \leq c$ for all $R \geq R_0$.  Therefore
\[
\|\ulxi_i\|_{1,p,R} \leq \|Q_{i,R} \ul{\eta}_i\|_R \leq  c \|\ul{\eta}_i\|_{0,p,R} 
\]
so we can estimate
\bea
\| D_{\bS, \ulu_R, r_R} T_R \ul{\eta} - \ul{\eta}\|_{0,p,R} & = &   \|(\partial_s \beta_{1,R}) \ulxi_1 + (\partial_s \beta_{2,R}) \ulxi_2 \|_{0,p,R} \\
                & \leq & 2/R (\|\ulxi_1\|_{0,p,R} +  \|\ulxi_2\|_{0,p,R})\\
                & \leq & \frac{2c}{R} \|\ul{\eta}\|_{0,p,R}
\eea
and $2c/R \leq 1/2$ for sufficiently large $R$.

\subsection*{Type 2}
Assume that $(r_0, \ulu_0), (r_1, \ulu_1) , \ldots, (r_k, \ulu_k)$ are all regular and isolated for the perturbation data prescribed on their strata.  For a gluing length $R$, form the preglued surface and map, abbreviating
\[
(r_R, \ulu_R):= (r_0 \#_R(r_1, \ldots, r_k), \ulu_0 \#_R (\ulu_1, \ldots, \ulu_k)). 
\]
For $i=0,\ldots,k$, denote by $(\ulJ_i(r_R), \ulK_i(r_R))$ the perturbation datum on $\sS_{r_i}$ given by the restriction of $(\ulJ(r_R), \ulK(r_R))$ to the truncation of $\sS_{r_i}$, extended trivially over the rest of the striplike end of $\sS_{r_i}$ by the given Floer data.   The assumption of consistency implies that for large $R$, the data $(\ulJ_i(r_R), \ulK_i(r_R))$ is a compact perturbation of $(\ulJ_i, \ulK_i)$.  Also, by construction the intermediate maps $\ulu_i^R$ are $W^{1,p}$ small perturbations of $\ulu_i$.  The properties of being Fredholm and surjective are stable under $W^{1,p}$-small perturbations, hence for sufficiently large $R$, $(r_0, \ulu_0^R), \ldots, (r_k, \ulu_k^R)$ are regular with respect to $(\ulJ_i(r_R), \ulK_i(r_R) )$.  The linearized operators $D_{\sS, r_i, \ulu_i^R}$ are surjective with zero dimensional kernel hence are isomorphisms, so let $Q_{i,R}$ be the inverse of $D_{\sS, r_i, \ulu_i^R}$.  For convenience, for each $i$ we will denote by $\sS_{r_i}^R$ the truncation of $\sS_{r_i}$ that appears in the pre-glued surface $\sS_{r_R}$, denote by $Z_1, \ldots, Z_k$ the striplike ends of $\sS_{r_1}, \ldots, \sS_{r_k}$ that are truncated in the pre-gluing process, and denote by $Z_1^\prime, \ldots, Z_k^\prime$ the corresponding striplike ends of the surface $\sS_{r_0}$.  

We construct an approximate right inverse $T_R$ for the linearized operator $D_R$ as follows.  Let $\ul{\eta} \in L^p(\sS_{r_R}, \Lambda^{0,1}\otimes_{J}\ulu_R^*T\ulM)$. Set
\bea
\ul{\eta}_0(z)  =  \left\{ \begin{array}{rl}
                                \ul{\eta}(z), & z \in \sS_{r_0}^R\,    \\
                                0, & \mbox{else}
                          \end{array}\right., \ \  \ldots, \ \   
\ul{\eta}_k(z) =  \left\{ \begin{array}{rl} 
                \ul{\eta}(z), &  z \in \sS_{r_k}^R,    \\
               0, & \mbox{else}.
                \end{array}\right.
\eea
with $\ul{\eta}_i \in  L^p(\sS_{r_i}, \Lambda^{0,1}\otimes_{J}(\ulu_i^R)^*T\ulM)$ for $i = 0, \ldots, k$. 
Define
\bea (\tau_0, \ulxi_0)& := & Q_{0,R}(\ul{\eta}_0) \in T_{r_0}\RR^{k+1} \times W^{1,p}(\sS_{r_0}, (\ulu_0^R)^* T\ulM)  \\
   & \ldots &\\
 (\tau_k, \ulxi_k) & := & Q_{k,R}(\ul{\eta}_k) \in T_{r_k}\RR^{d_k+1, 0} \times W^{1,p}(\sS_{r_k}, (\ulu_k^R)^* T\ulM).
\eea
The final step is to glue these $k+1$ things together to get a single element of $T_{r_R}\RR^{d_1+d_2, 0} \times W^{1,p}(\sS_{r_R}, \ulu_R^*T\ulM)$.  For large $R$, $r_R$ is near the boundary of $\RR^{d_1 + \ldots + d_k, 0}$, where local charts identify
\[
T_{r_R}\RR^{d_1 + \ldots + d_k, 1} \cong T_{r_0}\RR^{k +1} \oplus T_{r_1}\RR^{d_1 + 1, 0} \oplus \ldots \oplus T_{r_k}\RR^{d_k + 1, 0} \oplus T_R \R,
\]
the last component coming from the gluing parameter. With this identification set
\[
\tau_0 \#_R (\tau_1, \ldots, \tau_k):= (\tau_0, \tau_1, \ldots, \tau_k, 0).
\]
Fix a smooth cutoff function $\beta: \R_{\geq 0} \to [0,1]$ such that 
$\beta(s) = 1$ for $s \geq 1$ and $\beta(s) = 0$ for $s \leq 1/2$, and $0 \leq \dot{\beta} \leq 2$.  Let $\beta_R(s):= \beta(s/R)$.  Define $\beta_{0,R}: \sS_{r_R} \to [0,1]$ by the condition that
\bea
\beta_{0,R} (z) & = & \left\{ \begin{array}{ll}
                     1, & z \in \sS_{r_0}^R,\\
                     \beta_R(s), & z \in Z_i, s \in [0,R], i = 1, \ldots, k,\\
                     0, & \mbox{else},                   
\end{array}\right.
\eea
and for each $i \in \{1, \ldots, k\}$, let
\bea
\beta_{i,R}(z) & = & \left\{\begin{array}{ll}
                     1, & z \in \sS_{r_i}^R,\\
                     \beta_R(s), & z \in Z_i^{\prime}\subset \sS_{r_0}, s \in [0,R],\\
                     0, & \mbox{else},
\end{array}\right.
\eea
then set $ \ulxi_0 \#_R \ (\ulxi_1, \ldots, \ulxi_k) := \sum_{i=0}^{k} \beta_{i,R} \ulxi_{i}.$
We define the approximate inverse by
$
T_R \ul{\eta} := (\tau_0\#_R(\tau_1, \ldots, \tau_k), \ulxi_0 \#_R (\ulxi_1, \ldots, \ulxi_k).
$ 
The construction leads to the estimate
\bea
D_{\bS, {r_R}, \ulu_R} T_R \eta - \eta & = & D^{(\ulu_R)}_{r_R} \tau_0\#(\tau_1, \ldots, \tau_k) + D^{(r_R)}_{\ulu_R} \ulxi_0 \# (\ulxi_1, \ldots, \ulxi_k) -\ul{\eta} \\
& = &  \sum\limits_{i = 0}^k D^{(\ulu_i^R)}_{r_i} \tau_i + \sum\limits_{i=0}^k D^{(r_i)}_{\ulu_i^R} \beta_{i,R} \ulxi_i - \ul{\eta}\\
& = &  \sum\limits_{i = 0}^k D^{(\ulu_i^R)}_{r_i} \tau_i + \sum\limits_{i=0}^k  \beta_{i,R}D^{(r_i)}_{\ulu_i^R}\ulxi_i +      \sum\limits_{i=0}^k (\partial_s \beta_{i,R})\ulxi_i - \ul{\eta}\\
& = & \sum\limits_{i=0}^k \beta_{i,R} ( D^{(\ulu_i^R)}_{r_i} \tau_i  + D^{(r_i)}_{\ulu_i^R}\ulxi_i ) + \sum\limits_{i=0}^k (\partial_s \beta_{i,R})\ulxi_i - \ul{\eta}\\
& = & \sum\limits_{i=0}^k \beta_{i,R} \ul{\eta}_i - \ul{\eta} + \sum\limits_{i=0}^k (\partial_s \beta_{i,R})\ulxi_i\\
& = & \sum\limits_{i=0}^k (\partial_s \beta_{i,R})\ulxi_i.
\eea
Since $(\tau_i, \ulxi_i) = Q_{i,R} \ul{\eta}_i$ there is a $c > 0$ such that
\[
\| \ulxi_i \|_{0,p,R} \leq \|Q_{i,R} \ul{\eta}_i \|_{1,p,R} \leq c \|\ul{\eta}_i\|_{0,p,R},
\]
for all $i$. Combine everything into a total estimate
\bea
\| D_{\bS, \ulu, r} T \ul{\eta} - \ul{\eta}\|_{0,p,R}  \leq  \sum\limits_{i=0}^k \| \dot{\beta}/R\ \ul{\xi}_i \|_{0,p,R}
                 \leq  \| \partial_s \beta_{i,R} \|_{\infty} \sum\limits_{i=0}^k\|\ulxi_i\|_{0,p,R}
                & \leq & \frac{2}{R} \sum\limits_{i=0}^k c \|\ul{\eta}_i\|_{0,p,R}\\
                & \leq & \frac{2c}{R} \|\ul{\eta}\|_{0,p,R}
\eea
and for sufficiently large $R$, $2c/R \leq 1/2$.

\subsection*{Type 3} Assume that $(r_1, \ulu_1)$ is a regular isolated pseudoholomorphic quilted disk, and $\ul{v}$ is a regular isolated Floer trajectory, for their respective perturbation data.  For large $R$, the intermediate maps $(r_1, \ulu_1^R)$ and $\ul{v}^R$ are $W^{1,p}$-small perturbations of $(r_1, \ulu_1) $ and $\ulv$ respectively, and since being Fredholm and surjective are properties that are stable under small perturbations, it follows that $(r_1, \ulu_1^R)$ and $v^R$ are also regular for sufficiently large $R$.  Let $Q_{1,R}$ be the inverse of $D_{\sS, r_1, \ulu_1^R}$ (which is invertible), and let $Q_{\ul{v}^R}$ be a right inverse for  $D_{\ul{v}^R}$ respectively, whose image is the $L^2$ orthogonal complement of the kernel of $D_{\ulv^R}$.  Let $\Sigma$ denote the infinite strip $\R \times [0,1]$, and label by $Z$ the striplike end of $\sS_{r_1}$ on which the pregluing takes place.  
Construct an approximate right inverse 
\[
T_R: L^p(\sS_{r_1}, \Lambda^{0,1}\otimes_{J}\ulu_R^*T\ulM) \to T_{r_1}\RR^{d + 1, 0} \times W^{1,p}(\sS_{r_1}, \ulu_R^* T\ulM)
\] 
of $D_{\sS, r_1, \ulu_R}$ as follows.   Let $\ul{\eta} \in  L^p(\sS_{r_1}, \Lambda^{0,1}\otimes_{J}\ulu_R^*T\ul{M})$. Set
\bea
\ul{\eta}_1(z) & = & \ul{\eta}(z), \ z \in \sS_{r_1}^R, \\
& & 0, \mbox{else},
\eea
and $\ul{\eta}_v = 1 - \ul{\eta}_1$.  Then $\ul{\eta}_1 \in L^p((\sS_{r_1}, \Lambda^{0,1}\otimes_{J}(\ulu_1^R)^*T\ulM)$, and $\ul{\eta}_v \in L^p(\Sigma, \Lambda^{0,1}\otimes_J(\ulv^R)^*T\ulM)$, and we set
\bea
Q_{1,R} \ul{\eta}_1 & := & (\tau_1, \ulxi_1) \in T_{r_1}\RR^{d + 1, 0}\times W^{1,p}(\sS_{r_1}, (\ulu_1^R)^* T\ulM)\\
Q_{\ulv^R} \ul{\eta}_v  & := & \ulxi_2 \in W^{1,p}(\Sigma, (\ulv^R)^* T\ulM).
\eea
We need to glue $\ulxi_1$ and $\ulxi_2$ together to get an element of 
$W^{1,p}(\sS_{r_1}, (\ulu_1 \#_R \ulv)^*T\ul{M})$.  Fix a smooth cut-off function $\beta: \R_{\geq 0} \to [0,1]$ such that $\beta(s) = 0$ for $s \leq 1/2 $ and $\beta(s) = 1$ for $s \geq 1$, and $0 \leq \dot{\beta} \leq 2$.  Let $\beta_R(s):= \beta(s/R)$.  Define $\beta_{1,R}: \sS_{r_1}\to [0,1]$ by
\bea
\beta_{1,R}(z) & = & 1, z \in \sS_{r_1}\setminus Z,\\
                           &  & 1 - \beta_R(s - R/2), z = (s,t) \in Z, 
\eea
and 
\bea
\beta_{2,R}(z) & = & 0, z \in \sS_{r_1}\setminus Z,\\
                          & & \beta_R(s), z = (s,t) \in Z.
                          \eea
Putting $\ulxi_1 \#_R \ulxi_2 := \beta_{1,R} \ulxi_1 + \beta_{2,R} \ulxi_2$
 we define $T_R \ul{\eta} := (\tau_1, \ulxi_1 \#_R \ulxi_2)$. 
Now,
\bea
D_{\sS, r_1, \ulu_R} T_R \ul{\eta} - \ul{\eta} & = & D_{\sS, r_1, \ulu_R}(\tau_1, \ulxi_1 \#_R \ulxi_2) - \ul{\eta}\\
& = & D_{r_1}^{(\ulu_1^R)}\tau_1 + D_{\ulu_1^R}^{(r_1)}  \beta_{1,R} \ulxi_1 + D_{\ulv^R}  \beta_{2,R} \ulxi_2 - \ul{\eta}\\
& = & D_{r_1}^{(\ulu_1^R)}\tau_1 + (\partial_s  \beta_{1,R}) \ulxi_1 + \beta_{1,R} D_{\ulu_1^R}^{(r_1)} \ulxi_1 +  (\partial_s  \beta_{2,R}) \ulxi_2 + \beta_{2,R} D_{\ulv^R}\ulxi_2 - \ul{\eta}\\
& = & \beta_{1,R} (D_{r_1}^{(\ulu_1^R)}\tau_1 + D_{\ulu_1^R}^{(r_1)} \ulxi_1) + \beta_{2,R} D_{\ulv^R}\ulxi_2 - \ul{\eta} + (\partial_s  \beta_{1,R}) \ulxi_1  +(\partial_s  \beta_{2,R}) \ulxi_2 \\
& = & (\partial_s  \beta_{1,R}) \ulxi_1  +(\partial_s  \beta_{2,R}) \ulxi_2,
\eea
and choosing a $c > 0$, $R_0 \geq 0$ such that the operator norms $\|Q_{1,R} \| \leq c, \|Q_{\ulv^R}\|\leq c$ for all $R \geq R_0$, we have an estimate
\bea
\| D_{\sS, r_1, \ulu_R} T_R \ul{\eta} - \ul{\eta} \|_{0,p,R} & \leq & \frac{2}{R} (\| \ulxi_1   \|_{0,p,R} +  \| \ulxi_2   \|_{0,p,R})\\
& \leq &  \frac{2}{R} (c \|\ul{\eta}_1\|_{0,p,R} + c\|\ul{\eta}_2\|_{0,p,R}) =  \frac{2c}{R} \| \ul{\eta}\|_{0,p,R}
\eea
so $2c/R \leq 1/2$ for sufficiently large $R$.  We obtain an actual right inverse $Q_R = T_R(D_RT_R)^{-1}$ using a power series
\bea
(D_RT_R)^{-1}  =  (\Id + (D_RT_R - \Id))^{-1}  =  \sum\limits_{k=0}^\infty (-1)^k(D_RT_R - \Id)^k,
\eea 
which is convergent for large $R$ because of the estimate $\|D_RT_R -\Id\| \leq 1/2$.  

\subsection{Uniform bound on the right inverse}
By construction, the image of $Q_R = T_R(D_RT_R)^{-1}$ is the same as the image of $T_R$.  The operator norm of $(D_RT_R)^{-1}$ can be uniformly estimated, by the results of the previous section. Thus it remains to find a uniform bound for $\|T_R\|$ in order to get a uniform bound $\|Q_R\|$.  

\subsection*{Type 1} $T_R$ is a composition of operations.  The initial cut-offs define a map 
\[
\triangle_1\times \triangle_2: L^p(\sS_{r_R}, \ulu_R^*T\ulM) \to L^p(\sS_{r_1}, (\ulu_1^R)^*T\ulM) \times L^p(\sS_{r_2}, (\ulu_2^R)^*T\ulM)
\]
where the norm on the product $ L^p(\sS_{r_1}, (\ulu_1^R)^*T\ulM) \times L^p(\sS_{r_2}, (\ulu_2^R)^*T\ulM)$ is the sum of the norms.  So by construction, the operator norm $\|\triangle_1\times \triangle_2\| = 1$.  The next step is the map $Q_{1,R} \times Q_{2, R} $ whose domain is the product $ L^p(\sS_{r_1}, (\ulu_1^R)^*T\ulM) \times L^p(\sS_{r_2}, (\ulu_2^R)^*T\ulM)$ and range is the product $  T_{r_1}\RR \times W^{1,p}(\sS_{r_1}, (\ulu_1^R)^*T\ulM) \times T_{r_2}\RR \times W^{1,p}(\sS_{r_2}, (\ulu_2^R)^*T\ulM)$.  The operator norm is estimated $\|Q_{1,R} \times Q_{2,R}\| \leq \|Q_{1,R}\| + \|Q_{2,R}\| \leq 2c$ where $c$ is a uniform bound on the operator norms of $Q_{1,R}, Q_{2,R}$ for large $R$. (Such a uniform bound $c$ exists because for large $R$, $Q_{1,R}$ and $Q_{2,R}$ converge  to $Q_1$ and $Q_2$, the respective  inverses of $D_{\sS, r_1, \ulu_1}$ and $D_{\sS, r_2, \ulu_2}$.) The final step in the construction of $T_R$ uses an operator
\[
\beta_1 \times \beta_2: W^{1,p}(\sS_{r_1}, (\ulu_2^R)^*T\ulM) \times W^{1,p}(\sS_{r_2}, (\ulu_2^R)^*T\ulM) \to W^{1,p}(\sS_{r_R}, (\ulu_R)^*T\ulM)
\]
 defined using the cut-off functions $\beta_{1,R}$ and $\beta_{2,R}$, which by construction satisfies 
 \[
  \| ( \beta_1 \times \beta_2)(\ulxi_1, \ulxi_2)\|_{1,p,R} \leq \|\ulxi_1\|_{1,p,R}  + \|\ulxi_2\|_{1,p,R}
  \] 
   where the right hand side is the norm of $(\ulxi_1, \ulxi_2)$ on the product $ W^{1,p}(\sS_{r_1}, (\ulu_2^R)^*T\ulM) \times W^{1,p}(\sS_{r_2}, (\ulu_2^R)^*T\ulM)$.  Hence $\|\beta_1 \times \beta_2\| = 1$.  Putting everything together we get that $\|T_R\| \leq 2c$. 
 
 \subsection*{Type 2} By almost identical arguments as above we can conclude that $\|T_R\| \leq (k+ 1) c$, where $c$ is a uniform bound for large $R$ on the operator norms $\|Q_{i,R}\|$, for $i = 0, 1, \ldots, k$.  

\subsection*{Type 3} By almost identical arguments as above we can conclude that $\|T_R\| \leq 2c$, where $c$ is a uniform bound for large $R$ on the operator norms $\| Q_{1,R}\|$ and $\|Q_{v,R}\|$.

\subsection{Quadratic estimate}  \label{quad_est_section}
The goal of this section is to establish a quadratic estimate, 
\[
\|d \F_{\sS, r, \ulu} (\rho, \ulxi) - D_{\sS, r, \ulu}\| \leq c (\|\ulxi\|_{W^{1,p}} + |\rho|)
\]
where the norm on the left is the operator norm.   The $W^{1,p}$ norm on the right will depend on a choice of volume form on $S$, and in the proof of the gluing theorem it will be a different volume form for different values of the gluing parameter $R$.  Note, however, that the definition of the operator $\F_{\sS, r, \ulu}$ doesn't depend in any way on the volume form on $S$.  The crucial thing to ensure is that the constant $c$ in the estimate can be chosen to be independent of the gluing parameter $R$. 

We assume that we are working in a local trivialization of the fiber bundle $\sS \to \RR$, in a neighborhood $U \subset \RR$ of $r$.  We will write $\ul{S}:=\sS_r$.  Recall that in this neighborhood of $r$, all the fibers $\sS_{r^\prime}$ are diffeomorphic to $\ul{S}$, and so the varying almost complex structures can be pulled back to $\ul{S}$, as can the perturbation data.   

 \begin{definition}
For a fixed volume form on $\ul{S}$, let
\[
c_p(\dvol_{\ul{S}}) : = \sup\limits_{0\neq f \in C^\infty(\ul{S})\cap W^{1,p}(\ul{S}) } \frac{\|f\|_{L^\infty}}{\|f\|_{W^{1,p}}}.
\]
\end{definition}
\begin{lemma} For each $d \geq 1$, there is a constant $c_0$ such that $c_p \leq c_0$ for all
the quilted surfaces $\sS_r$ in the bundle $\sS^d \to \RR$.
\end{lemma}
\begin{proof} If one fixes volume forms for each patch in the quilts in a given family $\sS^{d}$ or $\sS^{d,0}$, then such a constant exists for each patch by the embedding statements of Appendix A, which can be uniformly bounded over a given family $\sS^{d,0}$ or $\sS^d$. 
\end{proof}

\begin{proposition}[c.f. Proposition 3.5.3 in \cite{mcd-sal}]
Let $p > 2$, $r\in \RR$ and let $\ul{S}:= \sS_r$ be its quilted surface with strip-like ends.   Let $\ul{S}^{thick (thin)}$ denote the thick (resp. thin) part of a thick-thin decomposition.  Then, for every constant $c_0>0$, there exists a constant $c>0$ such that the following holds for every volume form $\dvol_{\ul{S}}$ on $\ul{S}$ such that $c_p(\dvol_S) \leq c_0$ and $\mathrm{vol}(S^{thick}) \leq c_0^p$.  If $\ulu \in W^{1,p}(S,\ulM)$, $\ulxi \in W^{1,p}(S, \ulu^*T\ulM)$, $r \in \RR$, and $\rho \in T_r\RR$ satisfy
\begin{equation}
\|d\ulu\|_{L^p}\leq c_0, \ \|\ulxi\|_{L^\infty} \leq c_0, \ |\rho| \leq c_0,
\end{equation}
then
\begin{equation}
\|d\F_{\sS, r, \ulu}(\rho, \ulxi) - D_{\sS, r, \ulu}\| \leq c (\|\ulxi\|_{W^{1,p}} + |\rho|).\label{est}
\end{equation}
Here $\| \cdot \|$ denotes the operator norm on the space of bounded linear operators from $T_r \RR \times W^{1,p}(S, u^*TM)$ to $L^{p}(S, \Lambda^{0,1}\otimes_{J} u^*TM)$.  
\end{proposition}

\begin{proof}

For each component manifold $M$ of $\ul{M}$, a point $x\in M$ and a vector $\xi\in T_x M$ determine a linear map
\bea
E_x(\xi) : T_x M  \to  T_{\exp_x\xi}M, \ \ \ 
\widetilde{\xi}  \mapsto  \frac{d}{d\lambda}\Big\lvert_{\lambda=0} \exp_x(\xi + \lambda \widetilde{\xi})
\eea
and a bilinear map
\bea
\Psi_x(\xi): T_x M\times T_x M  \to  T_{\exp_x\xi}M, \ \ \ \ 
(\widetilde{\xi}, \widetilde{\eta})  \mapsto \widetilde{\nabla}_\lam \Phi_x(\xi + \lambda \tilde{\xi})\eta \Big\lvert_{\lambda = 0} 
\eea
Similarly, for each pair of manifolds $(M_1, M_2)$ in $\ul{M}$ labeling the patches on either side of a seam (we include true boundary components by considering them as a seam with the other side labeled by $\{pt\}$), a point $x := (x_1, x_2) \in M_1^-\times {M_w} =: M$ 
together with a vector $\xi:=(\xi_1, \xi_2) \in T_{x}M$ determine a linear map
\bea
E^Q_{x}(\xi): T_xM   \to  T_{\exp^Q_x(\xi)}M, \ \ \ \ 
                                 \tilde{\xi} \mapsto \frac{d}{d\lambda}\Big\lvert_{\lambda=0} \exp^Q_x(\xi + \lambda \tilde{\xi})
\eea
and a bilinear map
\bea
\Psi^Q_x(\xi): T_x M\times T_x M  \to  T_{\exp_x\xi}M, \ \ \ \ 
(\widetilde{\xi}, \widetilde{\eta})  \mapsto  \widetilde{\nabla}_\lam  \Phi^Q_x(\xi + \lambda \tilde{\xi})\eta \Big\lvert_{\lambda = 0}
\eea
We use the notation $\Phi_r(\rho)$ to denote the inverse of the projection $[ \ ]^{0,1}_{(r)}$.  That is, $\Phi_r(\rho)^{-1}(\psi) = [\psi]^{0,1}_{(r)} = \Pi_r \psi = \frac{1}{2}\left( \psi + J(r)\circ \psi \circ j(r)\right)$.  We therefore have, for each $r \in \RR$, and $\rho \in T_r \RR$, a linear map
\bea
E_r(\rho) : T_{r}\RR  \to     T_{r_\rho}\RR , \ \ \ \ 
                          \tilde{\rho}  \mapsto   \frac{d}{d\lambda}\Big\lvert_{\lambda=0} \exp_r(\rho + \lambda \tilde{\rho})
\eea
and a family of bilinear maps (parametrized by the quilted surface) given by
\bea
\Psi_{r,z}(\rho): T_r\RR \times \Lambda^{0,1}_{(r)}(T^*_z S) \otimes_{J}  T_{x} M & \to & \Lambda^{0,1}_{(r_\rho)}(T^*_z S) \otimes_{J}  T_{x} M \\
       (\tilde{\rho}, \eta\otimes \xi) & \mapsto & \frac{d}{d\lambda} \Big\lvert_{\lambda =0} \Pi_{r_\rho} \circ \Phi_r(\rho + \lambda \tilde{\rho})(\eta \otimes \xi).
\eea 

All the manifolds are compact, and by construction the almost complex structures $J(r,z)$ (which determine the local metrics and exponential maps) vary only over a compact subset of the bundle $\sS$, and the quadratic corrections $Q$ near each seam/boundary component only vary over a compact interval.  Therefore, there is a constant $c_1$ such that for all $|\rho| \leq c_0$ and $|\xi| \leq c_0$, 
\bea
& & |E_x(\xi)\tilde{\xi}|  \leq  c_1 |\tilde{\xi}|, \ \ 
|\Psi_x(\xi)(\tilde{\xi},\eta)|   \leq  c_1 |\xi| |\tilde{\xi}| |\eta|, \ \ 
|E_x^Q(\xi)\tilde{\xi}|  \leq  c_1 |\tilde{\xi}|, \ \ 
|\Psi^Q_x(\xi)(\tilde{\xi}, \eta)  \leq  c_1 ||\xi| |\tilde{\xi}| |\eta|, \\
& & |E_r(\rho)\tilde{\rho}| \leq  c_1 |\tilde{\rho}|, \ \ 
|\Psi_{r,z}(\rho; \tilde{\rho}, \eta \otimes \xi)| \leq  c_1 |\rho| |\tilde{\rho}| |\eta\otimes \xi|
\eea
for all norms determined by the metrics $g_{\ul{J}(r,z)}$. The inequalities for the bilinear forms are possible because both are 0 when $\xi = 0$ and $\rho = 0$.  

From here on in the calculations, we will simplify notation and write $u$ instead of $\ulu$,  and $\xi$ instead of $\ulxi$; but it is important to remember that they are defined on quilts. Starting from the identity 
\begin{equation*}
\Phi_{\sS,r,u}(\rho + \lambda\tilde{\rho},\xi + \lambda\tilde{\xi}) \F_{\sS, r,u}(\rho + \lambda\tilde{\rho},\xi + \lambda\tilde{\xi}) = \ol{\partial}_J(\exp_r(\rho + \lambda \tilde{\rho}), \exp_u(\xi + \lambda \tilde{\rho}))
\end{equation*}
we covariantly differentiate both sides respect to $\lam$ at $\lam =0$ within our trivialization.  For typesetting reasons, abbreviate $\F:= \F_{\sS,r,u}, u_\xi:= \exp_u(\xi), r_\rho:= \exp_r(\rho)$. 
\begin{equation}
\widetilde{\nabla}_\lambda \Big\lvert_{\lam=0} \Pi_{r_\rho} \Phi_{\sS,r,u}(\rho + \lambda\tilde{\rho},{\xi} + \lambda\tilde{{\xi}}) \F(\rho + \lambda\tilde{\rho},{\xi} + \lambda\tilde{{\xi}}) = D_{\sS,r_\rho,{u}_{{\xi}}}(E_r(\rho)\tilde{\rho}, E_{u}({\xi})\tilde{{\xi}})\label{covariant}.
\end{equation}
 Expanding the left hand side of (\ref{covariant}) yields 
\begin{eqnarray*} 
\left[ \Psi_r(\rho; \tilde{\rho}, \Phi_{u}({\xi}) \F(\rho,{\xi}) \right]^{0,1}_{r_\rho} 
+   [\Psi_{u}({\xi}; {\widetilde{\xi}}, \Phi_r(\rho)\F(\rho,{\xi}))]^{0,1}_{r_\rho}  + \Phi_{\sS, r, u}(\rho, {\xi}) d\F(\rho, {\xi})(\tilde{\rho}, \tilde{{\xi}})].
\end{eqnarray*}
Therefore we have
\begin{eqnarray*}
d\F_{\sS,r,u}(\rho,\xi)(\tilde{\rho},\tilde{\xi})  -  D_{\sS,r,u}(\tilde{\rho},\tilde{\xi})  & & - \  \underset{A}{\underbrace{\Phi_{\sS,r,u}(\rho,\xi)^{-1} \left[\Psi_u(\xi; \tilde{\xi},\Phi_r(\rho)\F(\rho,\xi))   \right]^{0,1}_{r_\rho} }} \\
&  = &  - \ \underset{B}{\underbrace{\Phi_{\sS,r,u}(\rho,\xi)^{-1} \left[ \Psi_r(\rho; \tilde{\rho},\Phi_u(\xi)\F(\rho,\xi))   \right]^{0,1}_{r_\rho}}} \\
& & + \ \underset{C}{\underbrace{\Phi_{\sS, r, u}(\rho, \xi)^{-1}D_{\sS,r_\rho,u_\xi}(E_r(\rho)\tilde{\rho}, E_x(\xi)\tilde{\xi}) - D_{\sS,r,u}(\tilde{\rho},\tilde{\xi})}}.
\end{eqnarray*}
We estimate $A$, $B$ and $C$ in turn.

\noindent {\it Estimating A.}  By definition $\Phi_{\sS, r, u}(\rho,\xi)^{-1} = \Pi_r \circ \Phi_u(\xi)^{-1}$.  By construction, the projection $\Pi_r$ satisfies $|\Pi_r \psi| \leq |\psi|$ when using the induced norm at $r$, and $\Phi_u(\xi)$ is an isometry.  Under the assumption that $|\rho|\leq c_0$, we can find a constant $c_2$ (which depends on $c_0$) such that the operator norms $\|\Phi_r(\rho)\| \leq c_2$ and $\|\Pi_{r_\rho}\| \leq c_2$.  So we get a pointwise estimate 
\bea
|A| & \leq & |  \Psi_u(\xi ; \tilde{\xi}, \Phi_r(\rho)\F(\rho,\xi))| \leq  c_1c_2 |\xi| |\tilde{\xi}| |\Phi_r(\rho)\F(\rho,\xi)|.
\eea
 Moreover
\bea
\F(\rho,\xi)    & = &  \Phi_{\sS, r, u}(\rho, \xi)^{-1} [d u_\xi - Y(r_\rho, u_\xi)) ]^{0,1}_{r_\rho},\\
\implies |\F(\rho,\xi)| & \leq &c_2 |d u_\xi - Y(r_\rho, u_\xi)| \leq c_2|du_\xi| + c_2|Y(r_\rho,u_\xi)|.
\eea
We can find a constant $c_3$, which depends on the fixed choice of $c_0$, and the fixed choice of almost-complex structures $J(r,z)$ in the perturbation datum, and the fixed choice of quadratic corrections $Q$ for the seams (but given fixed choices, this constant can be made independent of the parameters $(r,z)$), such that
\[
|d\exp^Q_u\xi| \leq c_3( |du| + |\nabla \xi| + |\xi|)
\]
whenever $\|du\|_{L^p} \leq c_0$ and $\|\xi\|_{L^\infty}\leq c_0$.   There is a further constant, call it $c_3$ as well, that gives a uniform bound $|Y(r, z, x)| \leq c_3$ for all $(r,z)\in \sS$, $x \in M$, and all metrics $| \cdot |$ on $TM$ induced by $J(r,z)$.  Taking $L^p$ norms of these estimates, we get
\bea
|A| & \leq & c_1c_2^2c_3|\xi||\txi|(|du| + |\nabla \xi| + |\xi| + 1)\\
\implies \| A \|_{p} & \leq & c_1 c_2^2 c_3 \left(\|\xi\|_{\infty}\|\txi\|_{\infty}\|du\|_{p} + \|\xi\|_\infty \|\txi\|_\infty \|\nabla\xi\|_p  + \|\xi\|_\infty\|\txi\|_\infty \|\xi\|_p+ \|\xi\|_\infty\|\txi\|_p\right)\\
& \leq & c_1 c_2^2 c_3 \left( c_0^3\|\xi\|_{1,p}\|\txi\|_{1,p} + c_0^2\|\xi\|_{1,p}\|\txi\|_{1,p} +  c_0^2 \|\txi\|_{1,p}\|\xi\|_{1,p} + c_0\|\xi\|_{1,p}\|\txi\|_{1,p}\right)\\
& \leq & c_A (\|\xi\|_{1,p} + |\rho|)(\|\txi\|_{1,p} + |\tilde{\rho}|),
\eea
where $c_A = c_A(c_0, c_1, c_2, c_3)$.
 
 \noindent {\it Estimating B.}  First we have
 \bea
 |B| \leq |\Psi_r(\rho; \tilde{\rho},\Phi_u(\xi)\F(\rho,\xi))| \leq  |\Psi_r(\rho; \tilde{\rho},\Phi_u(\xi)\widetilde{\F}(\rho,\xi))| + |\Psi_r(\rho; \tilde{\rho},\Phi_u(\xi)\cP(\rho,\xi))|
 \eea
 where the notation $\F = \widetilde{\F} + \cP$  (from section \ref{lin_op}) just separates the $\delbar$ piece of the inhomogeneous pseudoholomorphic equation from the perturbation piece.  
 For the first term,
 \bea
 |\Psi_r(\rho; \tilde{\rho},\Phi_u(\xi)\widetilde{\F}(\rho,\xi))| & \leq & c_1 |\rho| |\tilde{\rho}|du_\xi| 
 \leq  c_1c_3|\rho| |\tilde{\rho}|(|du| + |\nabla \xi| + |\xi|)\\
 \implies  \|\Psi_r(\rho; \tilde{\rho},\Phi_u(\xi)\widetilde{\F}(\rho,\xi))\|_p & \leq & c_1c_2 |\rho||\tilde{\rho}|(\|du\|_p + \|\nabla{\xi}\|_p + \|\xi\|_p)\\
 & \leq & a_B (\|\xi\|_{1,p} + |\rho|)(\|\txi\|_{1,p} + |\tilde{\rho}|),
 \eea
 where $a_B=a_B(c_0,c_1,c_3)$.  For the other part of $B$, observe that $\Psi_r(\rho; \tilde{\rho},\Phi_u(\xi)\cP(\rho,\xi)) = 0$ on striplike ends and the images of glued striplike ends, i.e., the {\em thin} part of the quilt.  To see that it vanishes on the striplike ends, note that $j$ and $J$ are independent of $r\in \RR$, so $\Phi_r(\rho)$ is the identity.  On the regions of $S$ in the images of the glued striplike ends, i.e. the {\em thin} parts, $J$ is independent of $r\in \RR$, but $j= j(r)$ does change with $r$.  However, it is independent of $r$ in the $t$ direction, so the bilinear form vanishes on $\cP(\rho,\xi)(\partial_t)$ here.  The bilinear form does not automatically vanish in the $\partial_s$ direction, however, the form $\cP$ itself vanishes precisely in the $s$ direction because the Hamiltonian perturbations are always independent of $s$ on these parts, so $Y(r,z)(\partial_s) = 0$ here.  Thus,  $\Psi_r(\rho; \tilde{\rho},\Phi_u(\xi)\cP(\rho,\xi))$ is only supported on the {\em thick} part of $S$.  With this in mind, we have that 
 \bea
 \|\Psi_r(\rho; \tilde{\rho},\Phi_u(\xi)\cP(\rho,\xi))\|_p  & \leq & c_1|\rho||\tilde{\rho}|\|Y(r,z)\|_{L^p(S^{thick})} \\
& \leq & c_1 c_3 |\rho||\tilde{\rho}| (\vol(S^{thick}))^{1/p}\\
& \leq & c_0 c_1 c_3  |\rho||\tilde{\rho}|\\
& \leq &b_B(\|\xi\|_{1,p} + |\rho|)(\|\txi\|_{1,p} + |\tilde{\rho}|),
 \eea
 where $b_B=b_B(c_0,c_1,c_3)$, and so taking $c_B = a_B + b_B$ we have that 
 \[
 \|B\|_p \leq c_B (\|\xi\|_{1,p} + |\rho|)(\|\txi\|_{1,p} + |\tilde{\rho}|).
 \]  
 
{\noindent \it Estimating C.}  Writing $D_{\sS, r_\rho, u_\xi} = D_{u_\xi}^{(r_\rho)} + D_{r_\rho}^{(u_\xi)}$ as in section \ref{lin_op},  we have
\bea
\Phi_ {\sS,r,u}(\rho, \xi)^{-1} D_{\sS, r_\rho, u_\xi} (E_{r,u}(\rho, \xi)(\tilde{\rho},\tilde{\xi})) - D_{\sS,r,u}(\tilde{\rho},\tilde{\xi}) \\
=  \underset{C(a)}{\underbrace{[\Phi_ {\sS,r,u}(\rho, \xi)^{-1} D_{u_\xi}^{(r_\rho)} E_u(\xi)(\tilde{\xi}) - D_u^{(r)}(\tilde{\xi})]}}\\
  +  \underset{C(b)}{\underbrace{[\Phi_ {\sS,r,u}(\rho, \xi)^{-1}D_{r_\rho}^{(u_\xi)} E_r(\rho)(\tilde{\rho}) - D_{r}^{(u)}(\tilde{\rho})]}}.
\eea
\noindent {\it Estimating $C(a)$.}  By definition $\Phi_{\sS, r, u}(\rho,\xi)^{-1} = \Phi_u(\xi)^{-1}\circ \Pi_r$, and $\Phi_u(\xi)$ is an isometry, so
\bea
\Phi_u(\xi) C(a) & = & \Pi_r \circ D_{u_\xi}^{r_\rho} (E_u(\xi) (\tilde{\xi})) - \Phi_u(\xi) D_u^{(r)}(\tilde{\xi})
 =  \left[ D_{u_\xi}^{r_\rho} (E_u(\xi) (\tilde{\xi})) - \Phi_u(\xi) D_u^{(r)}(\tilde{\xi})\right]^{0,1}_{(r)}\\
\implies |C(a)| & \leq & |D_{u_\xi}^{(r_\rho)}(E_u(\xi) (\tilde{\xi}) ) - \Phi_u(\xi) D_u^{(r)}(\tilde{\xi})|\\
& \leq & \underset{C(a)(i)}{\underbrace{|\widetilde{D}_{u_\xi}^{(r_\rho)}(E_u(\xi) (\tilde{\xi}) ) - \Phi_u(\xi) \widetilde{D}_u^{(r)}(\tilde{\xi})|}} + \underset{C(a)(ii)}{\underbrace{|P_{u_\xi}^{(r_\rho)}(E_u(\xi) (\tilde{\xi}) ) - \Phi_u(\xi) P_u^{(r)}(\tilde{\xi})|}} .
\eea
\noindent {\it Estimating $C(a)(i)$.}  By formula (\ref{lin_delbar}) for $\widetilde{D}_{u}^{(r)}$ we have
\begin{eqnarray}
&&|\widetilde{D}_{u_\xi}^{(r_\rho)}(E_u(\xi) (\tilde{\xi}) ) - \Phi_u(\xi) \widetilde{D}_u^{(r)}(\tilde{\xi})|  \leq  |[ \nabla (E_u(\xi)(\txi)) ]^{0,1}_{r_\rho} - \Phi_u(\xi)[\nabla \txi  ]^{0,1}_{r}|\nonumber \\
 &&\ + \  \half |J(r_\rho)(\nabla_{E_u(\xi) (\tilde{\xi})}J)(r_\rho,\ulu_\xi)\partial_{J(r_\rho)}u_\xi - \Phi_u(\xi)J(r)(\nabla_{\txi}J)(r,u) \partial_{J(r)}u|.\label{eqnone}
\end{eqnarray}
 For $|\rho| \leq c_0$, we can find a constant $c_4$ such that at every $(r,z) \in \sS$, the operator norm $\|\Pi_{\rrho} - \Pi_r\| \leq c_4|\rho|$ uniformly for all $(r,z) \in \sS$ and $x\in M$. Hence the first term in (\ref{eqnone}) can be estimated by 
\bea
|[ \nabla (E_u(\xi)(\txi)) ]^{0,1}_{r_\rho} - \Phi_u(\xi)[\nabla \txi  ]^{0,1}_{r}| & \leq & \ |\nabla (E_u(\xi)(\txi)) - \Phi_u(\xi)\nabla \txi| + |(\Pi_{\rrho} - \Pi_r)\circ \Phi_u(\xi)\nabla \txi|\\
& \leq & |\nabla (E_u(\xi)(\txi))-E_u(\xi)\nabla\txi| + |(E_u(\xi) - \Phi_u(\xi))\nabla\txi| + c_4|\rho| |\nabla \txi|\\
& \leq & | (\nabla E_u(\xi)) \txi| + |(E_u(\xi) - \Phi_u(\xi))\nabla\txi| + c_4|\rho| |\nabla \txi|.
\eea
Since $E_x(0) - \Phi_x(0)=0$, there is a uniform constant, again call it $c_4$,  such that $\|E_x(\xi) - \Phi_x(\xi)\| \leq c_4$ for all $|\xi|\leq c_0$.  There is also a uniform constant, again call it $c_4$, such that pointwise, the operator norm $\|\nabla E_u(\xi)\| \leq c_4|\xi|(|du|+|\nabla\xi|)$ whenever $|\xi(z,u(z))|\leq c_0$. Hence,
\bea
\|[ \nabla (E_u(\xi)(\txi)) ]^{0,1}_{r_\rho} - \Phi_u(\xi)[\nabla \txi  ]^{0,1}_{r}\|_p & \leq & c_4\left(\|\xi\|_\infty(\|du\|_p+\|\nabla \xi\|_p )\|\txi\|_\infty + \|\xi\|_\infty\|\nabla\txi\|_p + |\rho| \|\nabla \txi\|_p\right)\\
&\leq & c_4\left( c_0^3 \|\xi\|_{1,p}\|\txi\|_{1,p} + c_0^2 \|\nabla\xi\|_p\|\|\txi\|_{1,p} + c_0\|\xi\|_{1,p}\|\nabla\txi\|_p + |\rho| \|\nabla \txi\|_p   \right)\\
& \leq & C(\|\xi\|_{1,p} + |\rho|)(\|\txi\|_{1,p} + |\tilde{\rho}|).
\eea
where $C=C(c_0,c_4)$ is a uniform constant.  Up to the factor of $\half$, the second term in (\ref{eqnone}) can be expanded to
\bea
&&J(r_\rho)(\nabla_{E_u(\xi) (\tilde{\xi})}J)(r_\rho,\ulu_\xi)\partial_{J(r_\rho)}u_\xi - \Phi_u(\xi)J(r)(\nabla_{\txi}J)(r,u)\partial_{J(r)}u \\
&& \ \ \ \ = \ \ (J(r_\rho)-J(r))(\nabla_{E_u(\xi) (\tilde{\xi})}J)(r_\rho,\ulu_\xi)\partial_{J(r_\rho)}u_\xi \\
&&\ \ \ \  \ \ + \ J(r)((\nabla_{E_u(\xi)(\txi)}J)(r_\rho,u_\xi) - (\nabla_{E_u(\xi)(\txi)}J)(r,u_\xi))\partial_{J(r_\rho)}u_\xi\\
&&\ \ \ \  \ \ + \  J(r)(\nabla_{E_u(\xi)(\txi)}J)(r,u_\xi)) (\partial_{J(r_\rho)}u_\xi - \partial_{J(r)}u_\xi)\\
&&\ \ \ \ \ \  + \  J(r)(\nabla_{E_u(\xi)(\txi)}J)(r,u_\xi)[\partial_{J(r)}u_\xi - \Phi_u(\xi)\partial_{J(r)} u] \\
&&\ \ \ \ \ \  + \  J(r)  [(\nabla_{E_u(\xi)(\txi)}J)(r,u_\xi) -  \Phi_u(\xi)(\nabla_{\txi}J)(r,u)] \Phi_u(\xi) \partial_{J(r)}u
\eea
There is constant $c_5$ so that $\|\nabla_\eta J\| \leq c_5|\eta|$ for all $(r,z)\in \sS$, $x \in M, \eta\in T_xM$.  Furthermore we can choose $c_5$ so that for all $(r,z)\in \sS$, $\rho\in T_r\RR$ satisfying $|\rho|\leq c_0$, $x\in M$ and $\xi \in T_xM$ such that $|\xi| \leq c_0$, we have
\bea
\|J(r_\rho) - J(r)\| \leq c_5 |\rho|, \ 
\|J(r_\rho)\| \leq c_5, \ |j(r_\rho)| \leq c_5, \ |(\nabla_{\xi}J)(r_\rho,x) - (\nabla_{\xi}J)(r,x)\| \leq c_5 |{\xi}| |\rho|, \\
| \partial_{J(r)}u_\xi - \Phi_u(\xi)\partial_{J(r)} u | \leq c_5|\xi|, \ |(\nabla_{E_u(\xi)(\txi)}J)(r,u_\xi) -  \Phi_u(\xi)(\nabla_{\txi}J)(r,u)| \leq c_5|\xi||\txi|, 
\eea
and for all $\psi \in T_z^*S\otimes T_xM$, we have $| [\psi]^{1,0}_{r_\rho} - [\psi]^{1,0}_r| \leq c_5|\rho||\psi|$, where $[ \ ]^{1,0}_r$ denotes projection onto the complex linear part with respect to $J(r)$ and $j(r)$. Thus the above expression can be estimated by
\bea
c_1c_3 c_5^4|\rho | |\tilde{\xi}| (|du| +|\nabla \xi|) + c_1c_3c_5^3|\txi||\rho| (|du| +|\nabla \xi|)  + c_1c_3 c_5^2|\txi||\rho| (|du| +|\nabla \xi|)  \\
+ c_1 c_5^2|\tilde{\xi}| |\xi| + c_5 |\xi||\txi| |du|,
\eea
and it follows readily that there is a constant $C=C(c_0, c_1, c_3, c_5)$ such that the $L^p$ norm of this sum is estimated by $C(|\rho| + \|\xi\|_{1,p})(|\tilde{\rho}| + \|\txi\|_{1,p})$.  Combining this with the estimate for the first term in (\ref{eqnone}) we get that there is a constant $C = C(c_0,c_1,c_3,c_4,c_5)$ such that 
\begin{equation}
\|C(a)(i)\|_p \leq C (|\rho| + \|\xi\|_{1,p})(|\tilde{\rho}| + \|\txi\|_{1,p}).
\end{equation}

\noindent {\it Estimating $C(a)(ii)$.} By formula (\ref{lin_pert}) for $P_{u}^r$, we have
\bea
C(a)(ii)  &=&  [\nabla_{E_u(\xi)(\txi)}Y(r_\rho,u_\xi)]^{0,1}_{r_\rho} - \Phi_u(\xi)[\nabla_{\txi}Y(r,u)]^{0,1}_r\\
& & - \half[J (r_\rho, u_\xi) \circ (\nabla_{E_u(\xi)(\txi)}J)(r_\rho,u_\xi) \circ Y(r_\rho,u_\xi)]^{0,1}_{r_\rho} \\
& & + \half \Phi_u(\xi)[J(r,u)\circ (\nabla_{\txi}J)(r,u) \circ Y(r,u)]^{0,1}_r\\
& = & (\Pi_{r_\rho} - \Pi_r)  \nabla_{E_u(\xi)(\txi)}Y(r_\rho,u_\xi)
 + [ \nabla_{E_u(\xi)(\txi)}Y - \Phi_u(\xi)\nabla_{\txi} Y]^{0,1}_{r}\\
 & & - \half \left( \Pi_{r_\rho} - \Pi_r\right)\circ (J \circ (\nabla_{E_u(\xi)(\txi)}J) \circ Y)(r_\rho,u_\xi)  \\
 & & - \half [  \left(J \circ \nabla_{E_u(\xi)(\txi)}J \circ Y \right)(r_\rho,u_\xi)-  \left(J \circ \nabla_{E_u(\xi)(\txi)}J \circ Y\right)(r,u_\xi) ]^{0,1}_r\\
 & & -\half [ \left(J  \circ \nabla_{E_u(\xi)(\txi)}J \circ Y\right)(r,u_\xi) - \Phi_u(\xi) \left(J \circ \nabla_{\txi}J\circ Y\right)(r,u)]^{0,1}_r
\eea
There is a constant $c_6$ such that for all $(r,z)\in \sS$, $\rho\in T_r\RR$ satisfying $|\rho|\leq c_0$, $x\in M$, $\xi \in T_xM$ such that $|\xi| \leq c_0$, and for all $\eta \in T_x M$, we have
\bea
 |\nabla_\eta Y |  \leq  c_6 |\eta|,  \\
 | \left(J\circ \nabla_\eta J\circ Y\right)(r_\rho,x) - \left(J\circ \nabla_\eta J\circ Y\right)(r,x)|  \leq   c_6 |\rho| |\eta|, \\
 |\nabla_{E_u(\xi)(\eta)} Y - \Phi_u(\xi) \nabla_\eta Y|  \leq   c_6 |\xi| |\eta|,\\
 |\left(\nabla_{E_x(\xi)(\eta)}J \circ Y\right)(r,\exp^Q_x(\xi)) - \Phi_x(\xi) \circ \left(\nabla_{\eta}J)\circ Y\right)(r,x)|  \leq  c_6 |\xi||\eta|.
\eea
Thus we get a pointwise estimate 
\bea
|C(a)(ii)| & \leq & c_1 c_4 c_6 |\rho||\txi| + c_6|\xi||\txi| + \half c_1 c_3c_4 c_5^2 |\rho| |\txi|  + c_0 c_6 |\rho||\txi| + c_6 |\xi||\txi|\\
\implies \|C(a)(ii)\|_p & \leq & C (|\rho| + \|\xi\|_{1,p})(|\tilde{\rho}| + \|\txi\|_{1,p}) 
\eea
where $C = C(c_0, c_1, c_3, c_4,c_5,c_6)$.  This completes the quadratic estimate for $C(a)$.  

{\noindent \it Estimating C(b).}  As in Section \ref{lin_op}, we write $D_r^{(u)}\rho = \widetilde{D}_r^{(u)}\rho + P_r^{(u)}\rho$, \bea
C(b) & = & \underset{C(b)(i)}{\underbrace{\Phi_u(\xi)^{-1} \Pi_r \widetilde{D}_{r_\rho}^{(u_\xi)}E_r(\rho)(\tilde{\rho}) - \widetilde{D}_r^{(u)}\tilde{\rho}}} - \underset{C(b)(ii)}{\underbrace{\Phi_u(\xi)^{-1} \Pi_r P_{r_\rho}^{(u_\xi)}E_r(\rho)(\tilde{\rho})  + P_r^{(u)} \tilde{\rho}}}.
\eea
{\noindent \it Estimating C(b)(i).} Using the explicit formula (\ref{lin_delbar1}), we get
\begin{eqnarray*}
2\Phi_u(\xi) C(b)(i)& = &  \underset{\alpha}{\underbrace{\Pi_r \circ \Pi_{r_\rho} (\partial_{E_r(\rho)(\tilde{\rho})}J\circ du_\xi \circ j)(r_\rho,u_\xi)  - \Pi_r \Phi_u(\xi)\circ (\partial_{\tilde{\rho}} J\circ du\circ j)(r,u) }}\\
& & +\underset{\beta}{\underbrace{ \Pi_r \Pi_{r_\rho}( J\circ du_\xi \circ \partial_{E_r(\rho)(\tilde{\rho})} j)(r_\rho,u_\xi)  - \Pi_r \Phi_u(\xi)(J\circ du \circ \partial_{\tilde{\rho}} j)(r,u).}}
\end{eqnarray*}
We can find a constant $c_7$ such that for all $(r,z) \in \sS$, all $\rho \in T_r\RR$ such that $|\rho|\leq c_0$, all $\theta\in T_r\RR$, and all $x \in M$ and $\xi\in T_xM$ with $|\xi|\leq c_0$, we have $
\|\partial_\theta J\| \leq c_7 |\theta|, \ \|\partial_\theta j\| \leq c_7 |\theta|, \|\partial_\theta Y\|\leq c_7 |\theta|$, $\|\partial_{E_r(\rho)(\theta)}J(r_\rho,x) - \partial_{\theta}J(r,x)\|\leq c_7|\rho||\theta|$, $|\Phi_x(\xi)(\partial_\theta J)(r,x)  - (\partial_\theta J)(r,\exp_x(\xi)) \Phi_x(\xi) |\leq c_7|\theta||\xi|$, $|du_\xi - \Phi_u(\xi) du | \leq c_7 |\xi|$, and $\|j(r_\rho) - j(r)\| \leq c_7|\rho|$.  Expanding $\alpha$ we get
\bea
\alpha & = & \Pi_r (\Pi_{r_\rho} - \Pi_r) (\partial_{E_r(\rho)(\tilde{\rho})}J\circ du_\xi \circ j)(r_\rho,u_\xi)  \\
& & + \Pi_r \left( \partial_{E_r(\rho)(\tilde{\rho})}J(r_\rho, u_\xi) - \partial_{\tilde{\rho}}J(r,u_\xi)\right) \circ du_\xi \circ j(r_\rho) 
\\ 
& &+\Pi_r (\partial_{\tilde{\rho}}J)(r,u_\xi) \circ du_\xi \circ (j(r_\rho)  - j(r)) \\
& & + \Pi_r(\partial_{\tilde{\rho}}J)(r,u_\xi)\circ  \left( du_\xi -  \Phi_u(\xi) du \right)\circ j(r)\\
& & + \Pi_r \left( \partial_{\tilde{\rho}}J(r,u_\xi)\circ \Phi_u(\xi) - \Phi_u(\xi) \partial_{\tilde{\rho}}J(r,u)\right) du\circ j(r)\\
\implies
|\alpha| & \leq & c_1 c_3 c_4 c_5 c_7 |\rho| |\tilde{\rho}| (|du_\xi|) +  c_3c_5c_7|\rho||\tilde{\rho}|(|du_\xi|) \\
& & + c_3 c_7^2|\tilde{\rho}|(|du_\xi|)|\rho| + 2 c_7 |\tilde{\rho}||\xi| \\
\implies \|\alpha\|_p & \leq & C (|\rho| + \|\xi\|_{1,p})(|\tilde{\rho}| + \|\txi\|_{1,p}) 
\eea
where $C = C(c_0,\ldots, c_7)$.  Now expanding $\beta$ we get
\bea
\beta & = & \Pi_r ( \Pi_{r_\rho} - \Pi_r)  J(r_\rho, u_\xi) du_\xi \partial_{E_r(\rho)(\tilde{\rho})} j(r_\rho) \\
& & + \Pi_r ( J(r_\rho,u_\xi) - J(r, u_\xi) ) du_\xi \circ \partial_{E_r(\rho)(\tilde{\rho})} j)(r_\rho)\\
& & + \Pi_r J(r,u_\xi)du_\xi (\partial_{E_r(\rho)(\tilde{\rho})} j(r_\rho) - \partial_{\tilde{\rho}}j)(r)\\
& & + \Pi_r J(r,u_\xi)(du_\xi - \Phi_u(\xi) du)\partial_{\tilde{\rho}}j(r)\\
\implies |\beta| & \leq & c_1c_3c_4 c_5c_7 |\rho| (|du_\xi|)|\tilde{\rho}| + c_1 c_7^2|\rho|(|du_\xi|)|\tilde{\rho}|  + c_3c_7(|du_\xi|)|\rho||\tilde{\rho}| + c_7^2|\xi||\tilde{\rho}|\\
\implies \|\beta\|_p & \leq &C(|\rho| + \|\xi\|_{1,p})(|\tilde{\rho}|+\|\tilde{\xi}\|_{1,p})
\eea
where $C=C(c_0, \ldots, c_7)$. 
{\noindent \it Estimating C(b)(ii).} By the explicit formula (\ref{lin_pert1}), we can write
\bea
\Phi_u(\xi) C(b)(ii) & = & \underset{a}{\underbrace{\Pi_r \Pi_{r_\rho} ( \partial_{E_r(\rho)(\tilde{\rho})}Y)(r_\rho, u_\xi) - \Phi_u(\xi) \Pi_r (\partial_{\tilde{\rho}}Y)(r,u)}}\\
& & + \underset{b}{\underbrace{\Pi_r \Pi_{r_\rho} (\partial_{E_r(\rho)(\tilde{\rho})}J Y j)(r_\rho, u_\xi) -  \Phi_u(\xi) \Pi_r (\partial_{\tilde{\rho}}J Y j)(r,u)}} \\
& & + \underset{c}{\underbrace{\Pi_r \Pi_{r_\rho} (J Y \partial_{E_r(\rho)(\tilde{\rho})}j)(r_\rho, u_\xi) -  \Phi_u(\xi) \Pi_r  (J Y \partial_\rho)(r,u)}}.
\eea
There is a constant $c_8$ such that for all $(r,z) \in \sS$, all $\rho \in T_r\RR$ such that $|\rho|\leq c_0$, all $\theta\in T_r\RR$, and all $x \in M$ and $\xi\in T_xM$ with $|\xi|\leq c_0$, we have $|\partial_{E_r(\rho)(\theta)}Y(r_\rho,x) - \partial_{\theta}(r,x)| \leq c_8|\rho||\theta|, |\partial_{\theta}Y(r, \exp_x \xi) - \Phi_x(\xi)\partial_{\theta}Y(r, x)| \leq c_7 |\xi||\theta|$, $|\partial_\theta Y(r,x)|\leq c_8|\theta|$.  Note also that $Y$ is independent of $r$ on the {\it thin} parts of the surfaces in $\sS$ (i.e., striplike ends and images of striplike ends under gluing), so in particular, $\partial_\theta Y$ is only supported on the thick part of each quilt $\sS_r$.  Therefore, expanding $a$ we have
\bea
a & = & \Pi_r (\Pi_{r_\rho}-\Pi_r)\partial_{E_r(\rho)(\tilde{\rho})}Y(r_\rho, u_\xi)\\
& & + \Pi_r \left( \partial_{E_r(\rho)(\tilde{\rho})}Y(r_\rho, u_\xi) - \partial_{\tilde{\rho}}Y(r, u_\xi)\right)\\
& & + \Pi_r \left( \partial_{\tilde{\rho}}Y(r, u_\xi) - \Phi_u(\xi)\partial_{\tilde{\rho}}Y(r, u)\right)\\
\implies \|a\|_p & \leq & c_1c_4c_8|\rho||\tilde{\rho}(\vol(S^{thick}))^{1/p} + c_8|\rho||\tilde{\rho}|(\vol(S^{thick}))^{1/p} + c_8 |\tilde{\rho}|\|\xi\|_p \\
 & \leq & C(|\rho| + \|\xi\|_{1,p})(|\tilde{\rho}|+\|\tilde{\xi}\|_{1,p})
\eea
where $C=C(c_0,\ldots,c_8)$.  The estimates for $b$ and $c$ are very similar, so we omit them.  Now combining all estimates proves the desired estimate (\ref{est}).  
\end{proof}

\subsection{The gluing map} \label{gluing_map}


To define the gluing map, we use an infinite dimensional implicit function theorem, quoted here from \cite[Appendix A]{mcd-sal}.
\begin{theorem}[Theorem A.3.4 in \cite{mcd-sal}]
Let $X$ and $Y$ be Banach spaces, $U\subset X$ an open subset of $X$, and $f: U \to Y$ a continuously differentiable map.  Suppose that for $x_0 \in U$, $D:= df(x_0): X \to Y$ is surjective, and has a bounded right inverse $Q: Y \to X$, and that $\delta > 0, C>0$ are constants such that
$\|Q\| \leq C$, $B_\delta(x_0) \subset U$, and 
\[
\|x - x_0\| \leq \delta \implies \|df(x) - D\| \leq \frac{1}{2C}.
\] 
 Suppose that $x_1 \in X$ satisfies $\|x_1 - x_0\| \leq \frac{\delta}{8}$, and $\|f(x_1)\| \leq \frac{\delta}{4C}$.  Then there exists a unique $x \in X$ such that 
 \[
 f(x) = 0, \ \ x - x_1 \in \mathrm{im} Q, \ \ \mbox{and} \ \ \|x - x_0\| \leq \delta.
 \]
 Moreover, $\|x - x_1\| \leq 2C\|f(x_1)\|$.
\end{theorem}
We apply it to our situation as follows. For each $R$, take a local trivialization of a small neighborhood of the preglued curve $(r_R, \ulu_R)$. Identify
\bea
&X  :=  T_{r_R}\RR \times W^{1,p}(\sS_{r_R}, \ulu_R^*TM), \ \ 
Y  :=   L^p(\sS_{r_R}, \Lambda^{0,1}\otimes_{J_R} \ulu_R^*T\ulM),&\\
&f  :=  \F_{\sS,r_R,\ulu_R}, \ \ 
x_0  :=  (0,0).&
\eea
For $\delta \leq 1/(2Cc)$,  the estimate (\ref{quad_est}) says that whenever
$ (|\rho| + \|\ulxi\|_{1,p}) \leq \delta $ then
\[
\| d\F_{\sS,r_R,\ulu_R}(\rho, \ulxi) - D_{\sS,r_R,\ulu_R}\| \leq c (|\rho| + \|\ulxi\|_{W^{1,p}}) \leq c  \delta \leq 1/(2C)
\]
For sufficiently large $R$ the constants $C$ and $c$ are independent of $R$, thus $\delta$ can also be chosen to be independent of $R$.  So now let $x_1 := (0, 0)$.  Then
\bea
\|\F_{\sS,r_R, \ulu_R}(0,0)\|_{0,p} & = & \| (\ol{\partial} - \ul{\nu})(r_{R}, \ulu_{R})      \|_{0,p}   
\leq c\  \epsilon(R) \leq   \frac{\delta}{4C}
\eea
for sufficiently large $R$.  Thus by the Implicit Function Theorem there exists a unique $(\rho_R, \ulxi_R) \in T_{r_R}\RR \times W^{1,p}( \sS_{r_R}, \ulu_R^*T\ulM)$ such that 
\bea
 \F_{\sS, r_R, \ulu_R}(\rho_R,\ulxi_R) = 0, \ \ (\rho_R ,\ulxi_R) \in \im Q_R, \ \  |\rho_R| + \|\ulxi_R\|_{W^{1,p}} \leq \delta.
\eea 
Now $\F_{\sS, r_R, \ulu_R} (\rho_R, \ulxi_R) = 0$ if and only if $(\exp_{r_R} \rho_R, \exp_{\ulu_R}\ulxi_R) \in \M_{d,1}(\ul{x}_0, \ldots, \ul{x}_d)$.  Thus we can define a {\em gluing map}
\begin{eqnarray}
g: [R_0, \infty) &\to& \M_{d,0}(\ul{x}_0, \ldots, \ul{x}_d), \ \ 
R  \mapsto  (\exp_{r_R} \rho_R, \exp_{\ulu_R} \ulxi_R) \nonumber
\end{eqnarray}
The implicit function theorem also implies
\[
|\rho_R| + \|\ulxi_R\|_{W^{1,p}} \leq 2 C \|\F_{\sS,r_R, \ulu_R}(0,0)\|_{0,p} \leq  2C\epsilon(R) \to 0
\] 
as $R\to \infty$.  In particular, the glued curves $g(R)$ converge to the preglued curves $(r_R, \ulu_R)$, hence as $R \to \infty$ they Gromov converge to the same limiting broken tuple.     

For a gluing length $R >> 0$, write $(r_R, \ulu_R)$ for the preglued curve, and $(\widetilde{r}_R, \widetilde{\ulu}_R)$ for the corresponding glued curve.  

In a local trivialization about the preglued curve $(r_R, \ulu_R)$, the implicit function theorem implies that the moduli space of pseudoholomorphic quilts in a neighborhood of $(r_R, \ulu_R)$ is modeled on a complement of $\im \ Q_R$.  In particular, to show that the image of the gluing map is contained in the 1-dimensional component of the moduli space of pseudoholomorphic quilted disks, it's enough to check that $\im \ Q_R$ has codimension 1.   Since $\ker D_R$ is a complement to $\im Q_R$, an equivalent statement is that $ \dim \ \ker \ D_R = 1$.  By construction, the right inverse $Q_R$ has the same image as the approximate right inverse $T_R$.    
\begin{proposition}
The images of the gluing maps are contained in the one-dimensional component of the moduli space,
$\M_{d,0}(\ul{x}_0, \ldots, \ul{x}_d)^1$.  
\end{proposition}
The proposition will follow from:
\begin{lemma}\label{lemma_glued_image}
For each type of gluing construction, with approximate right inverse $T_R$, 
$\codim \  \im \ T_R = \dim \ \ker \ D_R = 1$.  
\end{lemma}

\begin{proof}
We prove Lemma \ref{lemma_glued_image} for each type of gluing construction in turn.  Since the constructions behave somewhat differently, the line of proof will be as follows:  for Types 1 and 2, we will prove that $\codim \ \im \ T_R =1$ by finding an explicit description of $\im \ T_R$ and a one-dimensional complement.   For Type 3, we will prove the existence of an isomorphism $\ker \ D_{\ul{v}} \cong \ker \ D_{\sS, r_1, \ulu_R}$; then, since  $\dim \ \ker \ D_{\ul{v}} = 1$, it will follow that $\dim \ \ker \ D_{\sS, r_1, \ulu_R} = 1$ too.  

\subsection*{Type 1}  Suppose that $(\rho, \ulxi) \in T_{r_R}\RR^{d,0} \times W^{1,p}(\sS_{r_R}, (\ulu_R)^*T\ulM)$.  Let $\beta_1 + \beta_{1,2} + \beta_2 = 1$ be a smooth partition of unity on $\sS_{r_R}$ such that the support of $\beta_1$ is on the part of $\sS_{r_R}$ that comes from the truncation of $\sS_{r_1}$ on the neck at $s = R$, the support of $\beta_2$ is on the part of $\sS_{r_R}$ that comes from the truncation of $\sS_{r_2}$ on the neck at $s=R$, and the support of $\beta_{1,2}$ is on the subset of the neck of $\sS_{r_R}$ corresponding to $ R/2 \leq s \leq R$ in $\sS_{r_1}$ as well as the corresponding piece from $\sS_{r_2}$.  For $\ulxi \in W^{1,p}(\sS_{r_R}, (\ulu_R)^*T\ulM)$, it is clear that 
\bea
\beta_1 \ulxi & \in & W^{1,p}(\sS_{r_1}, (\ulu_1^R)^*T\ulM),\\
\beta_2 \ulxi & \in & W^{1,p}(\sS_{r_2}, (\ulu_2^R)^*T\ulM),\\  
\beta_{1,2} \ulxi & \in & W^{1,p}(\sS_{r_1}, (\ulu_1^R)^*T\ulM) \cap W^{1,p}(\sS_{r_2}, (\ulu_2^R)^*T\ulM).
\eea
For $\rho \in T_{r_R}\RR^{d,0}$, we use local charts near the boundary to write $\rho = \rho_g + \rho_1 + \rho_2$, where $\rho_1 \in T_{r_1}\RR, \rho_2 \in T_{r_2}\RR$ and $\rho_g \in \R$ represents the component in the direction of the gluing parameter.   Thus  
\[
(\rho, \ulxi) =  (\rho_g, 0) + (\rho_1, \beta_1\ulxi) + (0, \beta_{1,2} \ulxi) + (\rho_2, \beta_2 \ulxi) 
\]
and the result follows if we can show that the final three terms on the right hand side are all in $\im T_R$. 
The cases of $(\rho_1, \beta_1\ulxi)$ and $(\rho_2, \beta_2 \ulxi)$ have identical proofs, so we prove it for the first case.

Consider $D_{\sS, r_R, \ulu_R}(\rho_1, \beta_1\ulxi) \in L^p(\sS_{r_R}, (\ulu_R)^*T\ulM)$.    It is supported on the image of the truncation $\sS_{r_1}\setminus \{s > R\}$ and so it can be identified with image of $D_{\sS, r_1, \ulu_1^R}(\rho_1, \beta_1 \ulxi)$.   The operators $D_{\sS, r, \ulu}$ preserve basepoints on the surface $\sS_r$, so the support of $\ulxi \in W^{1,p}(\sS_r,\ulu^*T\ulM)$ on $\sS_r$ is the same as the support of $D_{\sS,r,\ulu}(\rho, \ulxi) \in L^p(\sS_r, \ulu^*T\ulM)$.  Applying this fact we can conclude that the support of $D_{\sS, r_R, \ulu_R}(\rho_1, \beta_1\ulxi)$ is precisely the support of the cut-off function used in the construction of $T_R$, and therefore 
\bea
T_R D_{\sS, r_R, \ulu_R}(\rho_1, \beta_1 \ulxi) & = & Q_1 D_{\sS,r_1,\ulu_1^R}(\rho_1, \beta_1 \ulxi) =  (\rho_1, \beta_1 \ulxi)
\eea  
where the last equality follows from the standing assumptions that $D_{\sS,r_1,\ulu_1^R}$ is surjective and has trivial kernel, so that it is an isomorphism; as such its right-inverse $Q_1$ is a left-inverse too.  Hence, $(\rho_1, \beta_1\ulxi) \in \im T_R$.

For the other piece, consider $D_{\sS, r_R, \ulu_R} (0, \beta_{1,2} \ulxi)$.  The support of $\beta_{1,2}\ulxi$ is on the part of the neck where all three of the operators $D_{\sS, r_R, \ulu_R}, D_{\sS, r_1, \ulu_1^R}$ and $D_{\sS_, r_2, \ulu_2^R}$ coincide, i.e.,
\bea
D_{\sS,r_R, \ulu_R}(0, \beta_{1,2}\ulxi)  & = & D_{\sS, r_1, \ulu_1^R}(0, \beta_{1,2}\ulxi)
  =  D_{\sS,r_2, \ulu_2^R}(0, \beta_{1,2}\ulxi).
\eea
By assumption both $D_{\sS, r_1, \ulu_1^R}$ and $D_{\sS_, r_2, \ulu_2^R}$ are isomorphisms (being surjective with trivial kernel).  Writing $\ul{\eta} =  D_{\sS, r_R, \ulu_R} (0, \beta_{1,2} \ulxi)$, recall that in the construction of $T_R$, one uses a cut-off function to write $\ul{\eta} = \ul{\eta}_1 + \ul{\eta}_2$, where the supports of $\ul{\eta}_1$ and $\ul{\eta}_2$ intersect only at the truncation line $s=R$, and then the right inverse $Q_1$ is used on $\ul{\eta}_1$ and the right inverse $Q_2$ is used on $\ul{\eta}_2$.   We know that $Q_1 \ul{\eta} = (0, \ulxi)$ and $Q_2 \ul{\eta} = (0, \ulxi)$.  The operator $T_R$ is defined by $T_R \ul{\eta} = Q_1 \ul{\eta}_1 + Q_2 \ul{\eta}_2$.  So it is enough to show that $Q_1 \ul{\eta}_1 = Q_2 \ul{\eta}_1$.  So suppose for the sake of contradiction that $Q_1 \ul{\eta}_1 - Q_2 \ul{\eta}_1 \neq 0$.  If so, applying $D_{\sS, r_1, \ulu_1^R}$, we could write
\bea
D_{\sS, r_1, \ulu_1^R} ( Q_1 \uleta_1 - Q_2 \uleta_1) & = & D_{\sS, r_1, \ulu_1^R} Q_1 \uleta_1 - D_{\sS, r_1, u_1^R}Q_2 \uleta_1\\
& = & D_{\sS, r_1, \ulu_1^R} Q_1 \uleta_1 - D_{\sS,r_2, \ulu_2^R} Q_2 \uleta_1\\
& = & \uleta_1 - \uleta_1=  0,
\eea  
contradicting the assumption that $\ker D_{\sS, r_1, \ulu_1^R} = 0$.  Hence, $T_R \uleta = (0, \beta_{1,2} \ulxi)$.  In summary, we have:
\begin{enumerate}
\item If $(\rho, \ulxi) \in T_{r_R}\RR^{d,0} \times W^{1,p}(\sS_{r_R}, (\ulu_R)^*T\ulM)$ is such that $\rho = (0, \rho_1, \rho_2) \in \R \times T_{r_1}\RR^{d-e+1,0} \times T_{r_2}\RR^{e} \cong T_{r_R}\RR^{d,1}$, then $(\rho, \ulxi) \in \im \ T_R$,
\item If $(\rho, 0) \in T_{r_R}\RR^{d,0} \times W^{1,p}(\sS_{r_R}, (\ulu_R)^*T\ulM)$ is such that $\rho = (\rho_g, 0, 0) \in \R \times T_{r_1}\RR^{d-e+1,0} \times T_{r_2}\RR^{e} \cong T_{r_R}\RR^{d,0}$, then $(\rho, 0) \notin \im \ T_R$.
\end{enumerate} 
Since these elements span $T_{r_R}\RR^{d,0} \times W^{1,p}(\sS_{r_R}, (\ulu_R)^*T\ulM)$, it follows that 
\begin{equation} \label{image_TR1}
\im \ T_R = \{(\rho, \ulxi) \big\lvert \rho_g = 0\}.
\end{equation}

\subsection*{Type 2}  

Suppose that $(\rho, \ulxi) \in T_{r_R}\RR^{d,0} \times W^{1,p}(\sS_{r_R}, (\ulu_R)^*T\ulM)$.  

Let $\beta_0 + \beta_{1} + \ldots + \beta_k + \beta_{0, 1} + \ldots + \beta_{0,k} = 1$ be a smooth partition of unity on $\sS_{r_R}$ such that for $i = 0, \ldots, k$, the support of $\beta_i$ is on the part of $\sS_{r_R}$ that comes from the truncation of $\sS_{r_i}$ on the neck at $s = R$, and the support of $\beta_{0,i}$ is on the subset of the neck of $\sS_{r_R}$ corresponding to $ R/2 \leq s \leq R$ in $\sS_{r_0}$ as well as the corresponding piece from $\sS_{r_i}$.  Then for $\ulxi \in W^{1,p}(\sS_{r_R}, (\ulu_R)^*T\ulM)$, it is clear that 
\bea
\beta_0 \ulxi & \in & W^{1,p}(\sS_{r_0}, (\ulu_0^R)^*T\ulM),\\
\beta_1 \ulxi & \in & W^{1,p}(\sS_{r_1}, (\ulu_1^R)^*T\ulM),\\
& \ldots & \\
\beta_k \ulxi & \in & W^{1,p}(\sS_{r_k}, (\ulu_k^R)^*T\ulM),\\  
\beta_{0,1} \ulxi & \in & W^{1,p}(\sS_{r_0}, (\ulu_0^R)^*T\ulM) \cap W^{1,p}(\sS_{r_1}, (\ulu_1^R)^*T\ulM),\\
& \ldots &\\
\beta_{0,k}\ulxi & \in &W^{1,p}(\sS_{r_0}, (\ulu_0^R)^*T\ulM) \cap W^{1,p}(\sS_{r_k}, (\ulu_k^R)^*T\ulM).
\eea
For $\rho \in T_{r_R}\RR^{d,0}$, we use local charts near the boundary to write $\rho = \rho_0 + \rho_1 + \ldots + \rho_k + \rho_{g}$, where for $i = 0, \ldots, k$, $\rho_i \in T_{r_i}\RR, $ and $\rho_g \in \R$ represents the component of the tangent vector in the direction of the gluing parameter.   Then as in the previous case we can write 
\[
(\rho, \ulxi)  =   (\rho_g, 0) + (\rho_0, \beta_0\ulxi)  + (\rho_1, \beta_1\ulxi) + \ldots + (\rho_k, \beta_k \ulxi) + (0, \beta_{0,1} \ulxi) + \ldots (0, \beta_{0,k} \ulxi) 
\]
and the result follows if all terms except the first on the right hand side of the above expression are in $\im T_R$.  The same argument as used for Type 1 proves that all terms on the right except for $(\rho_g, 0)$ are in the image of $T_R$, and the result follows.  In summary:
\begin{enumerate}
\item If $(\rho, \ulxi) \in T_{r_R}\RR^{d,0} \times W^{1,p}(\sS_{r_R}, (\ulu_R)^*T\ulM)$ is such that $\rho_g = 0$, then $(\rho, \ulxi) \in \im \ T_R$,
\item If $(\rho, 0) \in T_{r_R}\RR^{d,0} \times W^{1,p}(\sS_{r_R}, (\ulu_R)^*T\ulM)$ is such that $\rho = \rho_g$, then $(\rho, 0) \notin \im \ T_R$.
\end{enumerate} 
Since these elements span $T_{r_R}\RR^{d,0} \times W^{1,p}(\sS_{r_R}, (\ulu_R)^*T\ulM)$, it follows that 
\begin{equation} \label{image_TR2}
\im \ T_R = \{(\rho, \ulxi) \big\lvert \rho_g = 0\}.
\end{equation}

\subsection*{Type 3} In this case we have glued $(r_1, \ulu_1) \in \M_{d,0}(\ul{x}_0, \ldots, \ul{y}, \ldots, \ul{x}_d)^0$ to a Floer trajectory $\ulv \in \widetilde{M}(\ul{y},\ul{x_i})^0$.  Since the complement of $\im \ T_R$ is $\ker \ D_R$, our goal is to show that the vector space $ \ker \ D_R$ as the same dimension as $\ker \ D_{v} $, which by assumption is $1$.  From this it would follow that the image of the gluing map lies in a one dimensional component  of the moduli space.  

To show that the finite dimensional vector spaces $\ker \ D_{v} $ and $ \ker \ D_R$ have the same dimension, it is enough to produce a pair of injective linear maps, $\Phi: \ker \ D_{v} \lra \ker \ D_R$ and $\Psi: \ker \ D_{R} \lra \ker \ D_v$.  Write $f^R$ for the pregluing map.

\noindent {\bf Claim 1:} $\Phi := (1 - Q_R D_R) df^R:  \ker \ D_{v} \lra \ker \ D_R$ is injective. 

By hypothesis $\ker D_v$ is one-dimensional, so we know that an explicit basis is $\{\partial_s \ulv\} $.  It suffices therefore to show that for sufficiently large $R$, there is a constant $c > 0$ such that 
\begin{equation}\label{est}
\|\partial_s \ulv\|_{1,p} \leq c \|(1 - Q_R D_R) df^R (\partial_s \ulv)\|_{1,p}.
\end{equation}
We prove the estimate by proving two separate inequalities
\begin{eqnarray}
\| df^R(\partial_s \ulv)\|_{1,p,R} & \geq & c_1 \| \partial_s \ulv\|_{1,p} \label{est1}\\
\|D_R df^R(\partial_s \ulv)\|_{0,p,R} & \leq & \epsilon(R) \|\partial_s \ulv\|_{1,p}\label{est2}
\end{eqnarray}
where $c_1 > 0$ and $\epsilon(R) \to 0 $ as $R \to \infty$.  Together these imply (\ref{est}), since the uniform bound on the right inverse 
\[
\|Q_R \ulxi\|_{1,p}  \leq  C \|\ulxi\|_{0,p}
\]
holds for $R$ sufficiently large, and so 
\begin{eqnarray}
\|(1 - Q_R D_R) df^R (\partial_s \ulv)\|_{1,p} & \geq & \|df^R(\partial_s \ulv)\|_{1,p} - \|Q_R D_R df^R (\partial_s \ulv)\|_{1,p}\nonumber\\
& \geq & c_1 \|\partial_s \ulv\|_{1,p} - C \|D_R df^R (\partial_s \ulv)\|_{0,p}\nonumber\\
& \geq & c_1 \|\partial_s \ulv\|_{1,p} - C \epsilon(R) \|\partial_s \ulv\|_{1,p}\nonumber\\
& \geq & c^{-1} \|\partial_s \ulv\|_{1,p}
\end{eqnarray}
for some $c > 0$ for sufficiently small $\epsilon(R)$. 

To prove (\ref{est1}), we write $\ulv_\lambda(s,t):= \ulv(s+\lambda,t)$. With this notation, $\partial_\lambda \ulv_\lambda\big\lvert_{\lambda = 0} = \partial_s \ulv$.   Thus,
\bea
df^R(\partial_s \ulv)(s,t) & = & \frac{d}{d\lambda}\Big\lvert_{\lambda = 0} \ulv_\lambda \#_R \ulu_1,
\eea
which by construction is supported only on the region $s \geq 3R/2$ on the striplike end. The pre-gluing map on this region is 
\[
f^R(\ulv_\lambda, \ulu_1) = \ulv_\lambda \#_R \ulu_1(s,t) = \left\{\begin{array}{ll}
\ul{\exp}^Q_{\ul{y}(t)}(\beta(-s + \frac{3R}{2} )\uleta_\lambda(s - 2R,t)), & s \in [\frac{3R}{2},\frac{3R}{2} + 1]\\
                    \ulv_\lambda(s-2R,t), & s \geq \frac{3R}{2} + 1. 
\end{array}
\right.
\]
and we need to take the derivative with respect to $\lambda$.  On the region $s \geq 3R/2 +1$ we have that 
\bea
df^R(\partial_s \ulv)(s,t) & = & \frac{d}{d\lambda}\Big\lvert_{\lambda = 0} \ulv_\lambda(s-2R,t) =  (\partial_s \ulv)(s-2R,t).
\eea
Note that
\bea
\|df^R(\partial_s \ulv)(s,t)\|_{1,p} &\geq& \|df^R(\partial_s \ulv)(s,t)\|_{1,p; [3R/2+1,\infty)}\\
& = &\|(\partial_s \ulv)(s-2R,t)\|_{1,p; [3R/2+1,\infty)}\\
& = & \| \partial_s \ulv\|_{1,p; [-R/2 + 1, \infty)}.
\eea
By exponential convergence of $\partial_s \ulv$, there is a $c_1 > 0$ such that for all sufficiently large $R$, 
\[
\| \partial_s \ulv\|_{1,p; [-R/2 + 1, \infty)} \geq c_1 \|\partial_s \ulv\|_{1,p}
\]
which proves (\ref{est1}).

To prove (\ref{est2}), observe first that by construction, $D_R df^R(\partial_s \ulv)$ is supported only on the interval $s \in [3R/2,3R/2+1]$.  It follows that the $L^p$ norm of $D_R df^R(\partial_s \ulv)$ is controlled by the $W^{1,p}$ norm of $df^R(\partial_s \ulv)$ on that interval.  Let $\beta_R(s): =  \beta(-s+3R/2)$ and $\uleta_R(s,t):= \uleta(s-2R,t)$.  For $s \in [3R/2,3R/2 + 1]$,
\bea
df^R(\partial_s \ulv)(s,t) & = & \frac{d}{d\lambda}\Big\lvert_{\lambda = 0} \ul{\exp}^Q_{\ul{y}(t)}(\beta_R \uleta_\lambda(s-2R,t))\\
& = & \frac{d}{d\lambda}\Big\lvert_{\lambda = 0} \ul{\exp}^Q_{\ul{y}(t)}(\beta_R(s)\uleta_R(s + \lambda,t))\\
& = & d\ul{\exp}^Q_{\ul{y}(t)} (\beta_R \uleta_R) \frac{d}{d\lambda}\Big\lvert_{\lambda = 0}(\beta_R \uleta_R(s + \lambda,t))\\
& = & \beta_R d\ul{\exp}^Q_{\ul{y}(t)} (\beta_R\uleta_R) (\partial_s \uleta)(s - 2R,t).
\eea
The identity
\bea
\ul{\exp}^Q_{\ul{y}(t)} \uleta(s-2R+\lambda, t) & = & \ulv(s-2R+\lambda),
\eea  
implies that
\bea 
d \ul{\exp}^Q_{\ul{y}(t)}(\uleta(s-2R,t)) (\partial_s \uleta)(s-2R,t) & = & \partial_s \ulv(s-2R). 
\eea
The linear operator $d \ul{\exp}^Q_{\ul{y}}(\uleta): T_{\ul{y}}\ulM \to T_{\ul{\exp}^Q_y\uleta}\ulM$ is the identity for $\uleta = 0$, so is invertible for small $\uleta$.  For sufficiently large $R$, the exponential convergence of trajectories means that $\uleta(s-2R, t)$ is uniformly small for $s \in [3R/2,3R/2 + 1]$.  Therefore on this interval, we can write
\bea
df^R(\partial_s \ulv)(s,t) & = & \beta_R(s) d\ul{\exp}^Q_{\ul{y}(t)} (\beta_R\uleta_R) [d \ul{\exp}^Q_{\ul{y}(t)}(\uleta(s-2R,t))]^{-1} \partial_s \ulv(s-2R).
\eea
Thus there is a constant $c_2 \geq 0 $ such that 
\bea
\|df^R(\partial_s \ulv)(s,t)\|_{1,p; [3R/2,3R/2 + 1]} & \leq & c_2 \| \partial_s \ulv(s-2R, t) \|_{1,p; [3R/2,3R/2 + 1]}\\
& = & c_2 \| \partial_s \ulv(s, t) \|_{1,p; [-R/2,-R/2 + 1]}\\
& \leq & \epsilon(R) \|\partial_s \ulv(s,t)\|_{1,p}
\eea
where the last inequality and the term $\epsilon(R)$ reflects the fact that the ratio 
\[
\|\partial_s \ulv\|_{1,p; [-R/2,-R/2 + 1]}/ \|\partial_s \ulv\|_{1,p}
\]
goes to 0 as $R \to \infty$.  This proves (\ref{est2}), hence also Claim 1.

Now let $\beta: \R \to [0,1]$ be a smooth cut-off function such that 
$\beta(s) = 0$ for $s \leq -1/4$ and $\beta(s) = 1$ for $s \geq 1/4$, and $0 \leq \dot{\beta} \leq 3$.  Define a shifted and rescaled cut-off function $\beta_R(s):= \beta((s - R)/R)$.  Then $\beta_R = 0$ for $s \leq 3R/4$ and $\beta_R = 1$ for $s \geq 5R/4$, and $0 \leq \dot{\beta_R} = \dot{\beta}/R \leq 3/R$.  

\noindent {\bf Claim 2:} $\Psi:=(1 - Q_{\ulu_1^R} D_{\ulu_1^R})(1-\beta_R)\times (1 - Q_{\ul{v}^R} D_{\ul{v}^R})\beta_R: \ker \ D_R \lra \ker D_{\ul{u}_1^R}\oplus \ker D_{\ul{v}^R} $ is injective.


Let $\ulxi \in \ker \ D_R$.  Then $\ulxi = (1 - \beta_R) \ulxi + \beta_R \ulxi$.  On the support of $(1-\beta_R)\ulxi$, the linearized operators $D_R$ and $D_{u_1^R}$ coincide, so
\bea
D_{\ulu_1^R} (1- \beta_R)\ulxi & = & D_R (1-\beta_R)\ulxi
 =  -\dot{\beta_R} \ulxi + (1-\beta_R)D_R \ulxi =   -\dot{\beta_R} \ulxi.
\eea
Hence $\|D_{\ulu_1^R} (1- \beta_R)\ulxi\|_{0,p} \leq 3/R \|\ulxi\|_{0,p}.$
Similarly on the support of $\beta_R \ulxi$, the linearized operators $D_R$ and $D_{\ul{v}^R}$ coincide, so
\bea
D_{\ul{v}^R} \beta_R \ulxi & = & D_R \beta_R \ulxi  =  \dot{\beta_R} \ulxi + \beta_R D_R \ulxi =  \dot{\beta_R} \ulxi,\\
\implies
\|D_{\ul{v}^R} \beta_R\ulxi\|_{0,p} & \leq & 3/R \|\ulxi\|_{0,p}.
\eea
Let $c_1$ and $c_2$ be uniform bounds for the right inverses $Q_{\ulu_1^R}$ and $Q_{\ul{v}^R}$ respectively.  Then we have:
\[
(1 - Q_{\ulu_1^R} D_{\ulu_1^R})(1-\beta_R)\times (1 - Q_{\ul{v}^R} D_{\ul{v}^R}) (\ulxi)  =  [(1-\beta_R)\ulxi + Q_{\ulu_1^R} (\dot{\beta_R} \ulxi), \beta_R \ulxi - Q_{\ul{v}_R} ( \dot{\beta_R} \ulxi)]
\]
Combine the identity $(1-\beta_R)\ulxi + \beta_R \ulxi  =  \ulxi$ with the estimates
\[
\|Q_{\ulu_1^R} (\dot{\beta_R} \ulxi)\|_{1,p} \leq  3 c_1 / R \|\ulxi\|_{0,p}\ \ \ \ 
\|Q_{\ul{v}^R}(\dot{\beta_R}\ulxi)\|_{1,p}  \leq  3c_2 / R \|\ulxi\|_{0,p}
\]
to get that for $R$ sufficiently large, there is a constant $c \geq 0$ such that 
\[
\| (1 - Q_{\ulu_1^R} D_{\ulu_1^R})(1-\beta_R)\times (1 - Q_{\ul{v}^R} D_{\ul{v}^R}) (\ulxi)\|_{1,p} \geq c \|\ulxi\|_{0,p}.
\] 
This proves Claim 2, completing the proof of Lemma \ref{lemma_glued_image}. 
\end{proof}

\subsection{Surjectivity of the gluing map}\label{surjectivity_section}

The final step is to prove the surjectivity of the gluing maps near the broken tuples.  Our goal will be to show that for sufficiently small $\epsilon$ and sufficiently large $R$, the gluing map associated to the given tuple surjects onto $\M_{d,0}(\ul{x}_0, \ul{x}_1, \ldots, \ul{x}_d)^1 \cap U_\epsilon$, where the neighborhoods $U_\epsilon$ were defined in Section \ref{gnbds}.  We will  prove surjectivity separately for the different types of gluing constructions.

\subsection*{Type 1}

\begin{proposition}
Let 
\bea
(r_1, \ulu_1) &\in& \M_{d-e+1, 0}(\ul{x}_0, \ldots, \ul{x}_{i}, \ul{y}, \ul{x}_{i+e+1}, \ldots, \ul{x}_d)^0 \ \mbox{and}\\ 
(r_2, \ulu_2) &\in& \M_{e}(\ul{y}, \ul{x}_{i+1}, \ldots, \ul{x}_{i+e})^0
\eea 
be regular, and let $U_\epsilon$ be a Gromov neighborhood of the pair.  Given $\delta > 0$, there is an $\epsilon > 0$ such that the following holds.  If $(r,\ulu) \in U_{\epsilon} \cap \M_{d, 0}(\ul{x}_0, \ldots, \ul{x}_d)^1$, then there is a pre-glued curve $(r_R, \ulu_R)$ and a $(\rho, \ul{\xi}) \in T_{r_R} \RR^{d, 0} \times \Omega^0(\sS_{r_R}, \ulu_R^*TM)$ such that $\exp_{r_R}\rho = r$, $\ul{\exp}^Q_{\ulu_R}\ulxi = \ulu$, $|\rho| + \|\ul{\xi}\|_{1,p} < \delta$, and $(\rho, \ulxi) \in \im \ Q_R$.   
\end{proposition}

\begin{proof}
We will prove it by contradiction.  Suppose there were a $\delta > 0$, and sequences $\epsilon_\nu \to 0$, $\delta_\nu \leq \epsilon_\nu \to 0$, $R_\nu = -\log(\delta_\nu) \to \infty$, and $(r_\nu, \ulu_\nu) \in  \M_{d, 0}(\ul{x}_0,\ldots,\ul{x}_d)^1$, and $\widetilde{r}_{1, \nu} \to r_1$, $\widetilde{r}_{2,\nu} \to r_2$ such that $r_\nu = \widetilde{r}_{1, \nu} \#_{\delta_\nu} \widetilde{r}_{2,\nu}$, with the following properties as $\nu \to \infty$: $E(\ulu_\nu) \to E(\ulu_1) + E(\ulu_2)$, and writing $(r^{R_\nu}, \ulu^{R_\nu}) $ for the preglued curve constructed with gluing length $R_\nu$, 
\begin{equation}\label{goal}
\inf \{ |\rho| + \|\ulxi\|_{1,p} \big\lvert (r_\nu,\ulu_\nu) = (\exp_{r^{R_\nu}} \rho, \ul{\exp}^Q_{\ulu^{R_\nu}} \ulxi) \} \geq \delta.  
\end{equation}
Our goal is to contradict \ref{goal}.  

Writing $\rho = \rho_g + \rho_1 + \rho_2$ where $\rho_g$ is the component of $\rho$ in the direction of the gluing parameter, it follows from the choice of gluing length $R_\nu$ that $\rho_g =0$.   For sufficiently small $\epsilon_\nu$, the condition $\exp_{r^{R_\nu} }\rho = r_\nu$ determines $\rho$ uniquely, so let us call it $\rho_\nu$.  The convergence of $\widetilde{r}_{i, \nu} $ to $r_i$ for $i=1,2$ implies that $|\rho_\nu| \to 0$.  

By assumption, $\ulu_\nu \to \ulu_1$ uniformly on compact subsets of $\sS_{r_1}$ and $\ulu_\nu \to \ulu_2$ uniformly on compact subsets of $\sS_{r_2}$; moreover since the maps are pseudoholomorphic, convergence on these compact subsets is uniform in all derivatives.  The preglued maps $\ulu^{R_\nu}$ have the same convergence properties, so for large $\nu$ there is a unique section $\ul{\xi}_\nu \in \Omega^0(\sS_{r^{R_\nu}}, (\ulu^{R_\nu})^*T\ulM)$ such that $\ul{\exp}^Q_{\ulu^{R_\nu}}\ul{\xi}_\nu = \ulu_\nu$.  So it is enough to show that $\|\ul{\xi}_\nu\|_{1,p} < \delta$ for sufficiently large $\nu$, contradicting (\ref{goal}).  Equivalently, we will show that the $L^p$ norms of $\ul{\xi}_\nu, \nabla_s \ul{\xi}_\nu$ and $\nabla_t \ul{\xi}_\nu$ can be made arbitrarily small by taking $\nu$ sufficiently large.   

It follows from the uniform convergence in all derivatives on compact subsets that on such subsets of $\sS_{r_1} \cup \sS_{r_2}$, the $L^p$ norms of $\ul{\xi}_\nu, \nabla_s \ul{\xi}_\nu$ and $\nabla_t \ul{\xi}_\nu$ all go to zero as $\nu \to \infty$.  We can choose these compact subsets to be such that their complement is on the striplike ends and neck.  Hence,  without loss of generality, it suffices to prove that the $L^p$ norms  of $\ul{\xi}_\nu, \nabla_s \ul{\xi}_\nu$ and $\nabla_t \ul{\xi}_\nu$ converge to zero along the striplike ends and neck of the preglued surfaces.   The exponential convergence of $\ulu_\nu$ as well as $\ulu_{R_\nu}$ along the striplike ends means that the $L^p$ norms of $\ulxi_\nu$ and its first derivatives can be made arbitrarily small too; so the essential thing to prove is that the $L^p$ norm of $\ul{\xi}_\nu$ along the neck of the preglued surface can be made arbitrarily small with sufficiently large $\nu$.   

The neck consists of two finite strips of length $R_\nu$ identified along an end to form a single strip, 
\[
  [-R_\nu, R_\nu]\times[0,1]  \cong [0,R_\nu]\times [0,1] \cup [0,R_\nu]\times [0,1] / \sim, 
\] 
where $\sim$ is the identification of $(R_\nu, 1-t)$ of the first strip with $(R_\nu, t)$ of the second, for $t \in [0,1]$.  

Let $\epsilon_0 > 0$ be given.  Fix $R > 0$ large enough that $\lim_{\nu \to \infty} E(\ulu_\nu; [R,R_\nu] \times  [0,1] \cup [R,R_\nu]\times [0,1] /\sim) < \epsilon_0$.  Without loss of generality we can assume that $\epsilon_0>0$ is small enough that $|\partial_s \ulu_\nu|$ satisfies, by Proposition \ref{exp_decay_strip},  $|\partial_s \ulu_\nu|  \leq  c e^{- \kappa^2 s}$
for all $\nu$ and for all $s \in [R, R_\nu]$, for some $c, \kappa > 0$.  Since $\ulu_\nu$ satisfies Floer's inhomogeneous pseudoholomorphic equation on this strip, we deduce
\begin{equation}\label{one}
|\partial_t \ulu_\nu(s,t) - X_{\ulH_t}(\ulu_\nu(s,t))| \leq  c e^{- \kappa^2 s}.
\end{equation}
Let $\phi_t$ be the flow of the Hamiltonian vector field $X_{]ulH_t}$, and consider the function $\widetilde{\ulu_\nu}:= \phi_{1-t} (\ulu_\nu(s,t))$.  Then
\bea
\partial_t \widetilde{\ulu_\nu} = (\phi_{1-t})_*(\partial_t \ulu_\nu - X_{\ulH_t}(\ulu_\nu))
\eea
and (\ref{one}) implies 
\[
\dist(\widetilde{\ulu_\nu}(s,1) - \widetilde{\ulu_\nu}(s,0))  \leq   \int_0^1 | \partial_t \widetilde{\ulu_\nu}| \ dt \leq  \widetilde{c} e^{- \kappa^2 s}.
\]
Since $\widetilde{\ulu_\nu}(s,1) \in \ul{L}_i$ and $\widetilde{\ulu}_\nu(s,0) \in \phi_1(\ul{L}_{i+e})$, this means that both are very close to an intersection $\ul{p} \in \phi_1(\ul{L}_{i+e}) \cap \ul{L}_{i}$.  The assumption of transverse intersection implies that there is a constant $a > 0$ such that 
\bea
\dist(\widetilde{\ulu}_\nu(s,0), \ul{p}) &\leq& a \ \dist(\widetilde{\ulu}_\nu(s,0), \ul{L}_1)\\
& \leq &a\ \dist(\widetilde{\ulu}_\nu(s,0), \widetilde{\ulu}_\nu(s,1)) \\
& \leq & a\ \widetilde{c}\ e^{-\kappa^2 s}.
\eea
Now for every other $t$, 
\bea
\dist(\widetilde{\ulu}_\nu(s,t), p) & \leq & \dist(\widetilde{\ulu}_\nu(s,t), \widetilde{\ulu}_\nu(s,0))  + \dist(\widetilde{\ulu}_\nu(s,0), p)\\
& \leq & \widetilde{c} e^{- \kappa^2 s} (t) + a\ \widetilde{c}\ e^{-\kappa^2 s}\leq  b \ e^{- \kappa^2 s}.
\eea
In terms of the original function $\ulu_\nu$, and writing $\phi_{1-t} \ul{x}(t) = p$, this estimate translates into 
\[
\dist(\ulu_\nu(s,t), \ul{x}(t)) \leq \widetilde{b} e^{- \kappa^2 s}.
\]
By construction, the preglued curves $\ulu^{R_\nu}$ satisfy a similar inequality, and so 
\bea
\dist(\ulu_\nu(s,t), \ulu^{R_\nu}(s,t)) & \leq & \dist(\ulu_\nu(s,t) , \ulx(t)) + \dist(\ulx(t), \ulu^{R_\nu}(s,t)) \leq  C\ e^{- \kappa^2 s}\\
\implies
|\ulxi_\nu(s,t)| & \leq & C^\prime \ e^{- \kappa^2 s}.
\eea
Taking the $L^p$ norm on a strip $[R,R_\nu]\times [0,1]$ gives
\bea
\int_R^{R_\nu} \int_0^1 |\ulxi_\nu|^p \ ds\ dt & \leq &  C^\prime \int_R^{R_\nu} e^{- p\kappa^2 s} \ ds =  \frac{C^\prime}{p\kappa^2} (e^{-p\kappa^2 R} - e^{-p\kappa^2 R_\nu}),
\eea
and this can be made arbitrarily small by choosing $R$ large enough.   By symmetry the same estimate holds for the strip $[R,R_\nu]\times [0,1]$ on the other side of the neck.  Thus, the $L^p$ norm of $\ulxi$ on the neck can be made arbitrarily small as $\nu \to \infty$.

Now we consider the $L^p$ norms of $\nabla_s \ulxi_\nu$ and $\nabla_t \ulxi_\nu$.  First note that since $\ulu_\nu$ converges uniformly in all its derivatives on compact subsets of $\sS_{r_1}$ and $\sS_{r_2}$ to the limits $\ulu_1$ and $\ulu_2$, we see that on such compact subsets we have uniform estimates for $|\nabla_s \ulxi_\nu| \to 0$ and $|\nabla_t \ulxi| \to 0$.  On the striplike ends of $\sS_{r_\nu}$, the exponential convergence of $\ulu_\nu$ and $\ulu^{R_\nu}$ to the same limits mean that the $L^p$ norms here can be made arbitrarily small.  Therefore what we need to show is that the $L^p$ norms of $\nabla_s \ulxi_\nu$ and $\nabla_t \ulxi_\nu$ on the neck, which varies in length with $\nu$, can be made arbitrarily small with large $\nu$.

Write $\ul{\exp}^Q: T\ulM \to \ulM$, and consider $ d\ul{\exp}^Q : T(T\ulM) \to T\ulM$.  At a fixed point $(p,\xi) \in T\ulM$ we can take a tangent vector $(\zeta, \eta), \zeta \in T_p \ul{M}, \eta \in T_p \ul{M}$, and write 
\[
d\ul{\exp}^Q_{(p,\xi)}(\zeta, \eta) = D_1 \ul{\exp}^Q_{(p,\xi)}(\zeta) + D_2 \ul{\exp}^Q_{(p,\xi)} (\eta)
\]
where $D_1 \ul{\exp}^Q_{p,\xi}$ corresponds to varying the basepoint $p$ while keeping all else fixed, and $D_2 \ul{\exp}^Q_{p,\xi}$ corresponds to fixing the basepoint $p$ and varying the tangent vector $\xi$.  In particular $D_1 \ul{\exp}^Q_{p,0}$ and $D_2 \ul{\exp}^Q_{(p,0)}$ are the identity, so for small $\xi$ they are invertible.  Hence
\begin{eqnarray}
 \partial_s \ulu_\nu   =  \partial_s \ul{\exp}^Q_{\ulu^{R_\nu}} \ulxi_\nu =  (D_1 \ul{\exp}^Q)_{(\ulu^{R_\nu},\ulxi_\nu)} \partial_s \ulu^{R_\nu} +( D_2\ul{\exp}^Q)_{(\ulu^{R_\nu}, \ulxi_\nu)} (\nabla_s  \ulxi_\nu) \nonumber\\
( D_2 \ul{\exp}^Q)_{(\ulu^{R_\nu}, \ulxi_\nu)}^{-1}(\partial_s \ulu_\nu -  (D_1 \ul{\exp}^Q)_{(\ulu^{R_\nu},\ulxi_\nu)} \partial_s \ulu^{R_\nu})  =  \nabla_s \ulxi_\nu \label{nabla_s}\\
 ( D_2 \ul{\exp}^Q)_{(\ulu^{R_\nu}, \ulxi_\nu)}^{-1}(\partial_t \ulu_\nu -  (D_1 \ul{\exp}^Q)_{(\ulu^{R_\nu},\ulxi_\nu)} \partial_t \ulu^{R_\nu})  =  \nabla_t \ulxi_\nu.\label{nabla_t}
\end{eqnarray}

First we analyze (\ref{nabla_s}).  The operators $( D_2 \ul{\exp}^Q)_{(\ulu^{R_\nu}, \ulxi_\nu)}^{-1}$ and $(D_1\ul{\exp}^Q)_{(\ulu^{R_\nu},\ulxi_\nu)}$ can be uniformly bounded for $\|\ulxi_\nu\|_\infty < \delta$, so by Proposition \ref{exp_decay_strip} we get on each strip $[R,R_\nu]\times [0,1]$ on either side of the neck,
\bea
|\nabla_s \ulxi_\nu| & \leq & c_1 |\partial_s \ulu_\nu| + c_2 |\partial_s \ulu^{R_\nu}| \leq A e^{-\kappa^2 s} 
\eea
for some constant $A > 0$, and therefore 
\bea
\int_{R}^{R_\nu} \int_0^1 |\nabla_s \ulxi_\nu|^p \ ds \ dt & \leq & \frac{A}{p\kappa^2} (e^{-p\kappa^2 R} - e^{-p\kappa^2 R_\nu})
\eea
which can be made arbitrarily small by choosing $R$ large enough.  The same estimate holds by symmetry on the other side of the neck.  From this we conclude that $\|\nabla_s \xi_\nu\|_{L^p}$ can be made arbitrarily small for large $\nu$. 

Now we analyze (\ref{nabla_t}). We can write  
\bea
\partial_t \ulu_\nu & = & J_t(\ulu_\nu) (\partial_s \ulu_\nu) + X_{\ulH_t}(\ulu_\nu)\\
\partial_t \ulu^{R_\nu} & = & J_t(\ulu^{R_\nu}) (\partial_s \ulu^{R_\nu}) + X_{\ulH_t}(\ulu^{R_\nu}) + E_\nu(s,t)
\eea
where $E_\nu(s,t)$ is an error term that is supported only on the compact interval $s \in [R_\nu/2, R_\nu/2 + 1]$ of each of the two strips making up the neck, with $|E_\nu(s,t)| \leq \delta_\nu \to 0$ as $\nu \to \infty$.  Together with (\ref{nabla_t}) this yields a pointwise estimate
\bea
|\nabla_t \ulxi_\nu | & = & | ( D_2 \ul{\exp}^Q)_{(\ulu^{R_\nu}, \ulxi_\nu)}^{-1}(\partial_t \ulu_\nu -  (D_1 \ul{\exp}^Q)_{(\ulu^{R_\nu},\ulxi_\nu)} \partial_t \ulu^{R_\nu}) | \\
& \leq & c|\partial_t \ulu_\nu -  (D_1 \ul{\exp}^Q)_{(\ulu^{R_\nu},\ulxi_\nu)} \partial_t \ulu^{R_\nu}|\\
& \leq & c_1 |\partial_s \ulu_\nu | + c_2 | \partial_s \ulu^{R_\nu}|  + c_3|X_{\ulH_t}(\ulu_\nu) - (D_1 \ul{\exp}^Q)_{(\ulu^{R_\nu},\ulxi_\nu)}X_{\ulH_t}(\ulu^{R_\nu})|
  + c_4| E_\nu(s,t)|\\
& \leq & c_5 ( |\partial_s \ulu_\nu | + | \partial_s \ulu^{R_\nu}| + \dist(\ulu_\nu, \ulu^{R_\nu}) + |E_\nu(s,t)|).
\eea
From this, applying the estimates for $|\partial_s \ulu_\nu|, |\partial_s \ulu^{R_\nu}|$ and $\dist(\ulu_\nu, \ulu^{R_\nu})$ and $|E_\nu(s,t)|$ we get 
\bea
\int_R^{R_\nu} \int_0^1 |\nabla_t \ulxi_\nu |^p \ ds \ dt & \leq & c_6 \int_R^{R_\nu} \int_0^1 \left( |\partial_s \ulu_\nu |^p + | \partial_s \ulu^{R_\nu}|^p\right. \\
& & \left.+ \dist(\ulu_\nu, \ulu^{R_\nu})^p + |E_\nu(s,t)|^p\right) \ ds \ dt\\
& \leq & c_7\int_R^{R_\nu}  e^{-\kappa^2 s} \ ds + \int_{R_\nu/2-1}^{R_\nu/2} \  \delta_\nu^p \ ds\\
& \leq & c_8 (e^{-\kappa^2 R} - e^{-\kappa^2 R_\nu}) + c_7 \delta_\nu^p,
\eea
and it is clear that this can be made arbitrarily small by taking $R$ large enough.  

This provides a contradiction to (\ref{goal}).  Hence, given $\delta > 0$, there is an $\epsilon > 0$ such that whenever $(r,\ulu) \in \M_{d,0}(\ul{x}_0, \ldots, \ul{x}_d) \cap U_\epsilon$, there is a gluing length $R$ such that $r = \exp_{r_R} \rho, \ulu = \ul{\exp}^Q_{\ulu_R} \ulxi$, with $|\rho| + \|\ulxi\|_{1,p} \leq \delta$, and with $\rho_g = 0$.  By (\ref{image_TR1}), this implies that $(\rho, \ulxi) \in \im \ Q_R$.
\end{proof}

\subsection*{Surjectivity for Type 2}

\begin{proposition}
Let $(r_0, \ulu_0), (r_1, \ulu_1), \ldots, (r_k, \ulu_k)$ be regular, and let $U_\epsilon$ be a Gromov neighborhood of the tuple.  Given $\delta > 0$, there is an $\epsilon > 0$ such that the following holds.  If $(r,\ulu) \in U_{\epsilon} \cap \M_{d, 0}(\ul{x}_0, \ldots, \ul{x}_d)^1$, then there is a pre-glued curve $(r_R, \ulu_R)$ and a $(\rho, \ul{\xi}) \in T_{r_R} \RR^{d, 0} \times \Omega^0(\sS_{r_R}, \ulu_R^*TM)$ such that $\exp_{r_R}\rho = r$, $\ul{\exp}^Q_{\ulu_R}\ulxi = \ulu$, $|\rho| + \|\ul{\xi}\|_{1,p} < \delta$, and $(\rho, \ulxi) \in \im \ Q_R$.   
\end{proposition}

\begin{proof} Again observe that for sufficiently small $\epsilon > 0$, if $\dist_{\RR}(r, r_0 \#_0 \{r_1, \ldots, r_k\} )< \epsilon$, the local charts near the boundary provide a unique way of writing $r = \tilde{r}_0 \#_{\delta_r} \{\tilde{r}_1, \ldots, \tilde{r}_k\}$ with $0 \leq \delta_r \leq \epsilon$ and with each $\tilde{r}_i$ in an $\epsilon$-neighborhood of $r_i$.  So now suppose that there were a $\delta > 0$ and sequences $\epsilon_\nu \to 0$, $(r_\nu, \ulu_\nu) \in \M_{d,0}(\ul{x}_0, \ldots, \ul{x}_d)^1$, and $0 < \delta_\nu \leq \epsilon_\nu$, $R_\nu = -\log \delta_\nu \to \infty$ such that $r_\nu = \tilde{r}_{0, \nu} \#_{\delta_\nu} \{ \tilde{r}_{1, \nu}, \ldots, \tilde{r}_{k,\nu}\}$, and $\dist_{\ul{M}}(\ulu(z), \ulu_i(z)) \leq \epsilon_\nu$ on all compact subsets of $\sS_{r_i}$, but 
\begin{equation}\label{goal2}
\inf \{ |\rho| + \|\ulxi\|_{1,p} \big\lvert r_\nu = \exp_{R_\nu}\rho, \ulu_\nu = \ul{\exp}^Q_{\ulu_{R_\nu}} \ulxi \} \geq \delta. 
\end{equation}
Then the convergence would be uniform in all derivatives on those compact subsets, and we could choose the compact subsets to be large enough that their complements comprise the striplike ends and the $k$ necks of the glued surfaces $\sS_{r_0 \#_{\delta_\nu} \{r_1, \ldots, r_k\}} = \sS_{r_{R_\nu}}$.  However on these striplike ends and these necks, the energy of $\ulu_\nu$ must approach 0, and the same exponential decay arguments would imply that $\ulu_\nu = \ul{\exp}^Q_{\ulu_{R_\nu}} \ulxi_\nu$ for some $\ulxi_\nu \in \Omega^0(\sS_{r_{R_\nu}}, \ulu_{R_\nu}^*T\ul{M})$ with $\|\ulxi_\nu\|_{1,p} \to 0$ as $\nu \to \infty$. Since $|\rho_\nu| \to 0$ also, we would get a contradiction to (\ref{goal2}).  Hence given $\delta > 0$ we could find an $\epsilon > 0$ such that $(r, \ulu) \in \M_{d,0}(\ul{x}_0, \ldots, \ul{x}_d)^1 \cap U_\epsilon$ could be written as $r = \exp_{r_R} \rho, \ulu = \ul{\exp}^Q_{\ulu_R} \ulxi$ for some preglued curve $(r_R, \ulu_R)$, with $|\rho| + \|\ulxi\|_{1,p} \leq \delta$.  Moreover, this $\rho$  is such that $\rho_g = 0$, so it follows from (\ref{image_TR2}) that $(\rho, \ulxi) \in \im \ Q_R$.
\end{proof}

\subsection*{Surjectivity for Type 3}  

\begin{proposition}\label{surj_3}
Let $(r_0, \ulu_0)$ and $\ul{v}$ be regular, and $U_\epsilon$ a Gromov neighborhood of the pair.  Given $\delta > 0$, there is an $\epsilon > 0$ such that the following holds.  If $(r,\ulu) \in U_{\epsilon} \cap \M_{d, 0}(\ul{x}_0, \ldots, \ul{x}_d)^1$, then $(r,\ulu)$ is in the image of the gluing map.     
\end{proposition}

\begin{proof}
In this case it suffices to prove the following:\\

\noindent {\bf Claim 1:} {\it Given $R_1 >> 0$, and $\delta > 0$, there is an $\epsilon > 0$ such that the following holds.  If $(r, \ulu) \in U_{\epsilon} \cap \M_{d,0}(\ul{x}_0, \ldots, \ul{x}_d)^1$, then there is an $R \geq R_1$, and a preglued curve $(r_R, \ulu_R)$ and a $(\rho, \ul{\xi}) \in T_{r_R} \RR^{d,0} \times \Omega^0(\sS_{r_R}, \ulu_R^*T\ul{M})$ such that $\exp_{r_R} \rho = r, \ul{\exp}^Q_{\ulu_R}\ulxi = \ulu$, and $|\rho| + \|\ul{\xi}\|_{1,p} < \delta$.} \\

\noindent To see how Claim 1 implies Proposition \ref{surj_3}, the argument is as follows.  From Section \ref{gluing_map} we know that the image of the gluing map is contained in the one dimensional component of the moduli space of pseudoholomorphic quilted disks, $\M_{d,0}(\ul{x}_0, \ldots, \ul{x}_d)^1$.   The image of $[R_0, \infty)$ is a connected component of this one-dimensional manifold.  The implicit function theorem also tells us that in a local trivialization about a preglued curve $(r_R, \ulu_R)$, we get a local chart for $\M_{d,0}(\ul{x}_0, \ldots, \ul{x}_d)^1$.   This chart contains $g(R)$, so the piece of the manifold $\M_{d,0}(\ul{x}_0, \ldots, \ul{x}_d)^1$ covered by the chart intersects the image of the gluing map.  So we want to show that if $R$ is sufficiently big, and $\delta$ is sufficiently small, then the piece of the one-manifold determined by the local chart about $(r_0, \ulu_R)$ is contained in the image of the gluing map.  

The preglued curves $\ulu_R$ are defined on the same domain $\sS_{r_0}$, but by construction any two of them will differ by a translation in the $s$ direction sufficiently far along the striplike end $Z_i$.  The magnitude of the distance between these translations depends on distances between points in $\ul{v}$, and the size of the difference in gluing lengths.  Since $\ul{v}$ is non-constant we can chooose $\delta > 0$ small enough that for any $R_1$, there will eventually be an $R^\prime > R_1$ such that the preglued curves $\ulu_{R}$ for $R \geq R^\prime$ can not be written $\ulu_{R} = \ul{\exp}^Q_{\ulu_{R_1}} \ulxi$ with $\|\ulxi\|_{1,p} \leq \delta$.  In particular, for $R \geq R^\prime$ the preglued curves $(r_0, \ulu_R)$ are not in a $\delta$-neighborhood of the local trivialization about $(r_0, \ulu_{R_1})$, and similarly $(r_0, \ulu_{R_1})$ is not in a $\delta$-neighborhood of the local trivialization about $(r_0, \ulu_{R})$.  So considering the gluing map $g: [R_0, \infty) \to \M_{d,0}(\ul{x}_0, \ldots, \ul{x}_d)^1$, by making $\delta > 0$ smaller if necessary we can assume that the hypotheses of the implicit function theorem are satisfied.  Then we can fix an $R_1 > R_0$ such that the respective $\delta$-neighborhoods of $(r_0, \ulu_{R_1})$ and $(r_0, \ulu_{R_0})$ are disjoint.  If we suppose, as in Claim 1, that $(r, \ulu) \in \M_{d,0}(\ul{x}_0, \ldots, \ul{x}_d)^1$ is such that $r = \exp_{r_0}\rho, \ulu = \ul{\exp}^Q_{\ulu_R} \ulxi$ for some $R \geq R_1$, and $|\rho|+\|\ulxi\|_{1,p} < \delta$, then $(r,\ulu)$ is in the local chart around $(r_0, \ulu_R)$.  But since $R \geq R_1$ we see that $g(R_0)$ is not in this chart, and we can choose an $R_2 >> R$ large enough that $g(R_2)$ is also not in that chart, but by the connectedness of the image of the gluing map this means that the whole chart is contained in the image of the gluing map.  \\

\noindent{\it Proof of Claim 1} For the sake of contradiction suppose that the assertion were false.  Then there would be some $\delta > 0$, and sequences $\epsilon_\nu \to 0$, $R_\nu = -\log \epsilon_\nu \to \infty$, $\tau_\nu \geq 2 R_\nu \to \infty$, and $(r_\nu, \ulu_\nu) \in \M_{d,0}(\ul{x}_0, \ldots, \ul{x}_d)^1$ such that 
\begin{itemize}
\item $\dist_{\RR}(r_\nu, r_0) < \epsilon_\nu$, 
\item $|E(\ulu_0) + E(\ul{v}) - E(\ulu_\nu)| < \epsilon_\nu$, 
\item $\dist_{\ul{M}}(\ulu_\nu(z), \ulu_0(z)) < \epsilon_\nu$ for all $z \in \sS_{r_0}^{R_\nu}$, and 
\item $\dist(\ulu_\nu(s+\tau_\nu,t), \ul{v}_(s,t)) < \epsilon_\nu$ for all $s \in [-R_\nu, R_\nu]$, 
\end{itemize}
and yet for every $\nu$,
\begin{equation} \label{contra}
\inf \{ |\rho| + \|\ulxi\|_{1,p} \big\lvert r_\nu = \exp_{r_0} (\rho), \ulu_\nu = \ul{\exp}^Q_{\ulu_R} (\ulxi) \} \geq \delta.
\end{equation}   
For large $\nu$, the condition $\exp_{r_0}\rho = r_\nu$ uniquely determines $\rho=:\rho_\nu$, and the convergence implies that $|\rho_\nu| \to 0$.  So the quantity $|\rho|$ becomes insignificant in (\ref{contra}).  We will arrive at a contradiction by showing that the norms $\|\ul{\xi}\|_{1,p}$ in (\ref{contra}) must also go to $0$ for large $\nu$.  The assumptions show that $\ulu_\nu$ converges to $\ulu_0$ uniformly on compact subsets of $\sS_{r_0}$, and since both are pseudoholomorphic the convergence is uniform in all derivatives.  On the striplike end $Z_i$, $\ulu_\nu(s+\tau_n, t)$ converges uniformly on compact subsets of $\R \times [0,1]$ to $\ul{v}(s,t)$; and since they are pseudoholomorphic curves the convergence is uniform in all derivatives.  Moreover, the preglued curves $(r_0, \ulu_{\tau_\nu})$ converge in the same way.   For each $\nu$ the energy of $\ulu_\nu$ restricted to the subsets $[R(\epsilon_\nu), \tau_n - R(\epsilon_\nu)]\times [0,1]$ and $s \geq \tau_n + R(\epsilon_\nu)$ of the striplike end $Z_i$ goes to zero.  Thus, the proof reduces to the same calculations as done for Type 1.  That is, the uniform estimates of convergence on those compact subsets of $\sS_{r_0}$, combined with exponential decay estimates based on the vanishing energy of the strips in the complement of those compact subsets, show that for sufficiently large $\nu$ there is a unique section $\ulxi_\nu \in \Omega^0(\sS_{r_0}, \ulu_{\tau_\nu})$ for which $\ul{\exp}^Q_{\ulu_{\tau_\nu}} \ulxi_\nu = \ulu_\nu$, and $\|\ulxi_\nu\|_{1,p} \to 0$, contradicting \ref{contra}. 
\end{proof}

\appendix

\section{ Sobolev embeddings.}

We collect relevant $W^{1,p}$ embedding statements for domains in $\R^2$, to verify that for each quilted surface $S$ constructed in Section \ref{quilt_disks}, there exists a constant $c_p$ such that
\[
\|f\|_{L^\infty(S)} \leq c_p(S) \|f\|_{W^{1,p}(S)}
\]
and that for the families constructed in Chapter 4, there is a uniform bound $c_p(S) \leq c_0$.  The following theorem is a consequence of general Sobolev estimates (see Theorem B.1.11, \cite{mcd-sal}).

\begin{appthm}\label{thm_1p}
Let $S \subset \R^2$ be a compact Lipschitz domain.   Let $u \in C^{\infty}(S)$.  Then there is a constant $c$, depending only on $p$, such that 
\[
\sup\limits_{S} |u(s,t)| \leq c \|u\|_{W_{std}^{1,p}}.
\]
Here the $W_{std}^{1,p}$ norm refers to the standard volume form $ds\wedge dt$ on $\R^2$. 
\end{appthm}
For a general volume form $\dvol_S$, we have the following consequence of Theorem \ref{thm_1p}.  
\begin{appcor}\label{1p_inside}
Let $S \subset \R^2$ be a compact Lipschitz domain, and $\dvol_S = f(s,t) ds\wedge dt$
a volume form on $S$.   Let $u \in C^{\infty}(S)$.   Then there is a constant $c = c(p)$ (in fact it is the same constant as in Theorem \ref{thm_1p}) such that
\[
\sup\limits_{S} |u(s,t)| \leq \frac{c}{(f_{min})^{1/p}} \|u\|_{W^{1,p}(S)}
\]
for all $z \in S$, where $\| \cdot \|_{W^{1,p}(S)}$ denotes the $W^{1,p}$ norm defined by the volume form $\dvol_S$, and $f_{min} = \min\limits_{ S} f(s,t)$.    
\end{appcor}
\begin{proof} Since $f(s,t)ds\wedge dt$ is a volume form, $f (s, t) > 0$ for all $(s,t)\in S$.  Moreover, $S$ is compact so $f$ achieves a minimum $f_{min} > 0$.   Hence, using the constant $c$ of Theorem \ref{thm_1p},
\bea
\sup\limits_{(s,t)\in S} |u(s,t)| & \leq & c\|u\|_{W^{1,p}}\\
& = & c \left( \iint\limits_S( |u|^p + |du|^p) ds\wedge dt  \right)^{1/p}\\
& = & c \left( \iint\limits_S (|u|^p + |du|^p )\frac{f(s,t)}{f(s,t)} ds\wedge dt  \right)^{1/p}\\
& \leq &c \left( \iint\limits_S (|u|^p + |du|^p) \frac{f(s,t)}{f_{min}} ds\wedge dt  \right)^{1/p}\\
& = & \frac{c}{f_{min}} \left( \iint\limits_S (|u|^p + |du|^p f(s,t)) ds\wedge dt  \right)^{1/p}\\
& = &  \frac{c}{f_{min}} \|u\|_{W^{1,p}(S)}.
\eea

\end{proof}

Similar estimates hold for unbounded domains $\R^2$ whose geometry satisfies a {\em cone condition}; see \cite{adams}, Chapter 5.    

\begin{definition}
A domain $\Omega \subset \R^2$ satisfies the {\em cone condition} if there is a finite cone $C = C(r_c,\theta_c)$ such that each $x\in \Omega$ is the vertex of a finite cone $C_x$ contained in $\Omega$  and congruent to $C$. 
\end{definition}
\begin{appthm}\label{1p_cone}
Let $\Omega \subset \R^2$.  Suppose that $\Omega$ satisfies the cone condition for some finite cone $C = C(r_c, \theta_c)$, and let $p > 2$.  Then there is a constant $ c = c(r_c, \theta_c, p) >0$ such that  for every $f \in C^\infty(S) \cap W^{1,p}(\Omega)$, and every $x \in \Omega$,
\[
|f(x)| \leq c \|f\|_{W^{1,p}}.
\]
\end{appthm}

\begin{appthm}
Let $S = S_1 \cup \ldots \cup S_l$ be a surface defined by a union of open sets $S_i$, such that each $S_i$ is one of the following types:
\begin{enumerate}
 \item The closure of $S_i$ is diffeomorphic to a compact Lipschitz domain $\widetilde{S}_i \subset \R^2$.
 
 \item The closure of $S_i$ is diffeomorphic to a domain $\widetilde{S}_i \subset \R^2$ that satisfies the cone condition, for some cone $C_i =  (r_i, \theta_i)$, {\em and} the volume form on $S$ restricted to $S_i$ is the pull-back of the standard volume form on $R^2$.
 
 \end{enumerate}
Then there is a constant $c$, depending on $p$, the cones in the cone condition, and the volume form $\dvol_S$ restricted to the $S_i$'s of type (a), such that
 \[
 \sup\limits_S |u| \leq c \| u\|_{W^{1,p}(S)}
 \]
for all $u \in C^\infty(S) \cap W^{1,p}(S)$.
\end{appthm}

\begin{proof}
Let $\{\rho_i\}$ be a partition of unity subordinate to the cover $S_1 \cup \ldots \cup S_l$ of $S$.  Then
$u  = \sum\limits_{i=1}^l \rho_i u$, and each $\rho_i u \in C^\infty(S_i) \cap W^{1,p}(S_i)$.  By Corollary \ref{1p_inside} and Theorem \ref{1p_cone}, there is a constant $c_i$ which depends on $p$ and, in the case of the subsets of type (a), the volume form $\dvol_S$ restricted to those components, such that 
\[
\sup\limits_S |\rho_i u| = \sup\limits_{S_i} |\rho_i u| \leq c_i \| \rho_i u\|_{W^{1,p}(S_i)} =  c_i \| \rho_i u\|_{W^{1,p}(S)} \leq c_i \|u\|_{W^{1,p}(S)}.
\]
Hence,
$\sup\limits_S |u| \leq  \sup\limits_S \sum\limits_{i=1}^l |\rho_i u|
 \leq  \sum\limits_{i=1}^l \sup\limits_S|\rho_i u|
 \leq  \left( \sum\limits_{i=1}^l c_i\right) \| u\|_{W^{1,p}(S)}
 =  c\| u\|_{W^{1,p}(S)}.$

\end{proof}

\section{Exponential decay}

Results on exponential decay for Floer trajectories with small energy follow from two things: a convexity estimate for the energy density of a trajectories in a sufficiently small neighborhood of a generalized intersection point, and a mean-value inequality that converts $L^2$ energy density estimates to pointwise estimates.   The results we collect here are based on the convexity estimates in \cite{robbin-salamon}, and mean-value inequality in \cite{wehrheim}.  

Without loss of generality, we only need to consider solutions to
\begin{eqnarray}
\partial_s u + J_t \partial_t u = 0, \label{floer_traj}\\
u(s,0) \subset L_0, u(s,1) \subset L_1, \nonumber
\end{eqnarray}
where $L_0$ and $L_1$ are transversely intersecting Lagrangians.  This is because a solution of the inhomogeneous equation,
\begin{eqnarray}
\partial_s u + J_t (\partial_t u - X_{H_t}(u)) = 0\\
u(s,0) \subset L_0, u(s,1) \subset L_1,\nonumber
\end{eqnarray}
can be translated into a solution of type (\ref{floer_traj}) by setting $\tilde{u}(s,t) = \phi_{1-t} (u(s,t))$, where $\phi_t$ is the time $t$ flow of the Hamiltonian vector field $X_{H_t}$, and $\tilde{J}_t := (\phi_{1-t}^{-1})^* J_t$, which satisfies
\begin{eqnarray*}
\partial_s \tilde{u} + \tilde{J}_t \partial_t \tilde{u} = 0,\\
\tilde{u}(s,0) \subset \phi_1(L_0), \tilde{u}(s,1) \subset \phi_0(L_1) = L_1,
\end{eqnarray*}
and by assumption the Hamiltonian perturbation is such that $\phi_1(L_0)$ intersects $L_1$ transversely. 

Consider a solution $u: I \times [0,1] \to M$ of (\ref{floer_traj}) where $I = [-T, T],  [T,\infty)$ or $(-\infty, -T]$  for some $T >0$.  We assume that $I$ is fixed.  The {\em energy density} of $u$ on this strip is the function 
\begin{equation}
e(s,t) := \omega( \partial_s u(s,t), J_t \partial_s u(s,t) ) = | \partial_s u(s,t)|_{J_t}^2 =  | \partial_t u(s,t)|_{J_t}^2
\end{equation}
Let $\triangle = \partial^2_s + \partial_t^2$ be the Laplace operator on $I \times [0,1]$. 
\begin{applem}\label{mean_est}
There are constants $A_1, a, B_1, b , C_1, c \geq 0$, which depend on $M, \omega, J_t$ and the Lagrangians $L_0, L_1$, such that if $u: I\times [0,1] \to M$ satisfies (\ref{floer_traj}), then its energy density, $e$, satisfies 
\begin{eqnarray}
\triangle e & \geq  & -A_1 e - a e^2 \label{laplacian}\\
\frac{\partial e}{\partial t}\Big\lvert_{t=0} & \geq & - B_1 e - b e^{3/2}\label{outernorm1}\\
\frac{\partial e}{\partial t}\Big\lvert_{t=1} & \leq & C_1 e + c e^{3/2}.\label{outernorm2}
\end{eqnarray}
 
\end{applem}
\begin{proof}

First we prove \eqref{laplacian}. We abbreviate $\xi = \partial_s u, \eta = \partial_t u,$; so $\xi + J_t(u) \eta = 0$.  
Consider the ($t$-dependent) Riemannian metric $\langle \cdot, \cdot \rangle_{J_t}$ on $M$, and let $\nabla$ be its  Levi-Civita connection.  Then
\bea
\partial_s \langle \xi, \xi \rangle_{J_t} & = & 2 \langle \nabla_s \xi, \xi\rangle_{J_t} \\
\partial_t \langle \xi, \xi \rangle_{J_t} & = & 2 \langle \nabla_t \xi, \xi \rangle_{J_t} + \omega(\xi, \dot{J_t} \xi)
=   2 \langle \nabla_t \xi, \xi \rangle_{J_t} + \langle J_t \xi, \dot{J_t} \xi\rangle_{J_t} 
\eea
from which we get that
\bea
\triangle e(s,t) & = &  2 \partial_s  \langle \nabla_s \xi, \xi\rangle_{J_t} + 2  \partial_t \langle \nabla_t \xi, \xi \rangle_{J_t} + \partial_t \langle  \dot{J_t} \xi, J_t \xi\rangle_{J_t} \\
& = & 2 \langle \nabla_s \nabla_s \xi, \xi\rangle_{J_t} + 2 \langle  \nabla_s \xi, \nabla_s \xi\rangle_{J_t} \\
& & + 2 \langle \nabla_t \nabla_t \xi, \xi \rangle_{J_t} + 2 \langle \nabla_t \xi,  \nabla_t \xi \rangle_{J_t} + 2 \omega( \nabla_t \xi, \dot{J_t} \xi)\\
& & + \langle \nabla_t ( \dot{J_t} \xi ),  J_t \xi\rangle_{J_t} + \langle \dot{J_t} \xi, \nabla_t ( J_t \xi) \rangle_{J_t} + \omega( \dot{J_t} \xi, \dot{J_t} J_t \xi)\\
& = &  2 \langle \nabla_s \nabla_s \xi, \xi\rangle_{J_t} + 2 \langle  \nabla_s \xi, \nabla_s \xi\rangle_{J_t} \\
& &  + 2 \langle \nabla_t \nabla_t \xi, \xi \rangle_{J_t} + 2 \langle \nabla_t \xi,  \nabla_t \xi \rangle_{J_t} + 2 \langle J_t \nabla_t \xi, \dot{J_t} \xi\rangle_{J_t}\\
& & + \langle \nabla_t ( \dot{J_t} \xi ),  J_t \xi\rangle_{J_t} + \langle \dot{J_t} \xi, \nabla_t ( J_t \xi) \rangle_{J_t} - \langle \dot{J_t} \xi, \dot{J_t} \xi\rangle_{J_t}\\
& = & 2 \langle (\nabla_s^2 + \nabla_t^2) \xi, \xi\rangle_{J_t} + 2  \langle  \nabla_s \xi, \nabla_s \xi\rangle_{J_t} + 2 \langle \nabla_t \xi,  \nabla_t \xi \rangle_{J_t} \\
& & + \langle (\nabla_t \dot{J_t} )\xi + \dot{J_t}\nabla_t \xi ,  J_t \xi\rangle_{J_t} +  \langle \dot{J_t} \xi, (\nabla_t J_t) \xi + J_t \nabla_t \xi \rangle_{J_t} - \langle \dot{J_t} \xi, \dot{J_t} \xi\rangle_{J_t}
\eea
Since $\nabla_t J_t = \nabla_\eta J + \dot{J}$ and $\nabla_t \dot{J_t} = \nabla_\eta \dot{J_t} + \ddot{J_t}$, there is a constant $c_2>0$ that depends on $J_t$ and $\omega$, and exists because of the compactness of $M$, such that the operator norms
\bea
\|\ddot{J_t}\|, \|\dot{J_t}\| & \leq & c_2\\
\|\nabla_tJ_t\| & \leq & c_2( |\eta|_{J_t} + 1) = c_2 (|\xi|_{J_t} + 1)\\
\|\nabla_t \dot{J_t}\| & \leq & c_2( |\eta|_{J_t} + 1) = c_2( |\xi|_{J_t} + 1) 
\eea
where the equalities follow from the identity $\eta = J_t \xi$, with $J_t$ an isometry for the metric $\langle \cdot, \cdot\rangle_{J_t}$.
We temporarily drop the subscript $J_t$ from our notation for the metric for the next few calculations, since there is no danger of confusion. So we can estimate
\bea
|\langle (\nabla_t \dot{J_t} )\xi, J_t \xi\rangle| & \geq & - c_2 (|\xi| + 1) |\xi|^2 \\
|\langle \dot{J_t}\nabla_t \xi ,  J_t \xi\rangle| & \geq & - c_2 |\nabla_t \xi| |\xi|\\
|\langle \dot{J_t} \xi, (\nabla_t J_t) \xi\rangle| & \geq & - c_2^2 (|\xi|+1)|\xi|^2\\
|\langle \dot{J_t} \xi, J_t \nabla_t \xi \rangle| & \geq & -c_2 |\xi| |\nabla_t \xi|\\
|\langle \dot{J_t} \xi, \dot{J_t} \xi\rangle| & \geq & -c_2^2|\xi|^2
\eea
\bea
\implies \triangle e(s,t) & \geq & -2 |(\nabla_s^2 + \nabla_t^2) \xi| |\xi| + 2  | \nabla_s \xi|^2 + 2 |\nabla_t \xi|^2 \\
& & - (c_2 + c_2^2)|\xi|^3 - (c_2 +2c_2^2)|\xi|^2 - 2c_2 |\nabla_t\xi| |\xi|.
\eea
It remains to estimate $|(\nabla_s^2 + \nabla_t^2)\xi|$. Since we are using a Levi-Civita connection, $\nabla_t \xi = \nabla_s \eta$. Write $R(\cdot, \cdot)\cdot$ for the curvature tensor associated to $\nabla$. Then   
\bea
\nabla_s \nabla_s \xi + \nabla_t \nabla_t \xi & = & \nabla_s \nabla_s \xi + \nabla_t \nabla_s \eta\\
& = &   \nabla_s \nabla_s \xi + (\nabla_t \nabla_s \eta - \nabla_s \nabla_t \eta) + \nabla_s \nabla_t \eta\\
& = & \nabla_s (\nabla_s \xi + \nabla_t \eta) + R(\eta, \xi)\eta.
\eea
Analyzing $\nabla_s \xi + \nabla_t \eta$:
\bea
\nabla_s \xi + \nabla_t \eta & = & \nabla_s (-J_t\eta ) + \nabla_t (J_t \xi )\\
& = & -(\nabla_s J_t)\eta - J_t \nabla_s \eta   + (\nabla_t J_t)\xi + J_t \nabla_t \xi \\
&= &-(\nabla_s J_t)\eta  + (\nabla_t J_t)\xi \\
\implies \nabla_s (\nabla_s \xi + \nabla_t \eta) & = & -(\nabla_s^2 J_t) \eta - (\nabla_s J_t)\nabla_s \eta  \\
& & + (\nabla_s\nabla_t J_t)\xi + (\nabla_t J_t)\nabla_s \xi.
\eea
Putting things together gives
\bea
\nabla_s^2 \xi + \nabla_t^2\xi & = & -(\nabla_s^2 J_t) \eta - (\nabla_s J_t)\nabla_s \eta  \\
& & + (\nabla_s\nabla_t J_t)\xi + (\nabla_t J_t)\nabla_s \xi + R(\eta,\xi)(\eta) .
\eea
Since $\nabla_s J_t = \nabla_\xi J_t$ and $\nabla_t J_t = \nabla_\eta J_t + \dot{J_t}$, there is a constant $c_3$ depending only on $\omega, J_t$, and the fact that $M$ is compact, such that we can estimate operator norms
\bea
\|\nabla_s J_t\| & \leq & c_3 |\xi| \\
\|\nabla_s (\nabla_s J_t)\| & \leq & c_3(|\nabla_s \xi| + |\xi|^2)\\
\|\nabla_s (\nabla_t J_t)\| & \leq & c_3(|\nabla_s \eta| + |\xi||\eta| + |\xi|)\\
& = & c_3(|\nabla_t \xi| + |\xi|^2 + |\xi|).
\eea
Also by the compactness of $M$ there is a uniform constant $c_4 > 0$ such that 
$|R(X,Y)Z| \leq c_4 |X||Y||Z|$ for all $X,Y, Z \in T_p M$, and all $p \in M$.  Hence, 
\bea
|\nabla_s^2 \xi + \nabla_t^2\xi| & \leq & c_3 (|\nabla_s \xi| + |\xi|^2) |\eta| + c_3 |\xi| |\nabla_s \eta|  + c_3 (|\nabla_t \xi| + |\xi|^2 + |\xi|)|\xi| \\
& & + c_2 (|\xi| +1) |\nabla_s \xi| + c_4|\eta|^2|\xi|\\
& = & c_3 |\nabla_s \xi| |\xi| + c_3 |\xi|^3 + c_3 |\xi||\nabla_t \xi| + c_3 |\nabla_t \xi||\xi| + c_3 |\xi|^3 + c_3|\xi|^2\\
& & + c_2 |\xi||\nabla_s \xi| + c_2|\nabla_s \xi| + c_4|\xi|^3\\
&  = & c_2 |\nabla_s \xi| + ( c_3 + c_2 ) |\nabla_s \xi | |\xi| + 2c_3 |\nabla_t \xi| |\xi| + (2c_3 + c_4 ) |\xi|^3 + c_3 |\xi|^2
\eea
leading to the estimate
\bea
\triangle e & \geq & -2 c_2 |\nabla_s \xi| |\xi| - 2 ( c_3 + c_2 ) |\nabla_s \xi | |\xi|^2 - 4 c_3 |\nabla_t \xi| |\xi|^2 - 2 (2c_3 + c_4 ) |\xi|^4 - 2 c_3 |\xi|^2 \\
& & + 2 |\nabla_s \xi|^2 + 2 |\nabla_t \xi|^2 - (c_2 + c_2^2)|\xi|^3 - (c_2 + 2c_2^2)|\xi|^2 - 2c_2 |\nabla_t \xi| |\xi|\\
& \geq & 2 |\nabla_s \xi|^2 + 2 |\nabla_t \xi|^2 - c |\nabla_s \xi| (|\xi|^2 + |\xi|) - c |\nabla_t \xi| (|\xi|^2 + |\xi|) - c (|\xi|^4 + |\xi|^3 + |\xi|^2
\eea
where $c > 0$ is chosen large enough to absorb the other constants. We apply the estimate $2xy \leq x^2 + y^2$ to the mixed terms,
\bea
|\nabla_s \xi| (|\xi|^2 + |\xi|) & \leq & \frac{\epsilon^2}{2}  |\nabla_s \xi|^2 + \frac{1}{2 \epsilon^2} (|\xi|^2 + |\xi|)^2\\
|\nabla_t \xi| (|\xi|^2 + |\xi|) & \leq & \frac{\epsilon^2}{2}  |\nabla_t \xi|^2 + \frac{1}{2 \epsilon^2} (|\xi|^2 + |\xi|)^2.
\eea
Choosing $\epsilon > 0$ sufficiently small that $B\epsilon^2/2 < 2 c_1$ leads to an inequality
\bea
\triangle e(s,t)  \geq  - \widetilde{A} (|\xi|^4 + |\xi|^3 + |\xi|^2)
\eea 
where $\widetilde{A}$ depends only on $J_t, M, \omega$.  The term $|\xi|^3$ can be absorbed into the other terms, in the sense that we can find constants $\tilde{A}_1 \geq 0, \tilde{a}\geq 0$ such that 
\bea
\widetilde{A} (|\xi|^2 + |\xi|^3 + |\xi|^4)  \leq  \widetilde{A}_1 |\xi|^2 + \widetilde{a}|\xi|^4
\leq  A_1 |\xi|_{J_t}^2 + a |\xi|_{J_t}^4
 =  A_1 e + a e^2.
\eea 
Thus $\triangle e \geq - A_1 e - a e^2$, proving (\ref{laplacian}).

Now we prove (\ref{outernorm1}), (\ref{outernorm2}).   In each case we fix a reference Riemannian metric $\langle \cdot, \cdot \rangle$ on $M$ that renders $L_0$ (respectively $L_1$ ) totally geodesic, and write $\nabla$ for its associated Levi-Civita connection.  Since the manifold $M$ is compact, there are uniform constants $c_1, c_2 > 0$ such that for all $p \in M, X \in T_p M$, and for all $t \in [0,1]$,
 \[
 c_1 |X|_{J_t} \leq |X| \leq c_2 |X|_{J_t}.
 \] 
 Then in either case we have
\bea
\frac{\partial e}{\partial t} & = & \partial_t \omega( \partial_s u, J_t(u) \partial_s u)\\
& = & 2 \omega( \nabla_t \partial_s u, J_t(u) \partial_s u) + \omega(\partial_s u, (\nabla_t J_t) \partial_s u) + (\nabla_{\partial_t u} \omega)( \partial_s u, J_t(u) \partial_s u).
\eea 
For the first term,
\bea
\nabla_t \partial_s u & = & \nabla_s \partial_t u\\
& = & \nabla_s (J_t (u) \partial_s u)\\
& = & (\nabla_{\partial_s u} J_t) \partial_s u + J_t(u) \nabla_s \partial_s u\\
\implies |\omega( \nabla_t \partial_s u, J_t(u) \partial_s u)| & \leq & |\omega( \nabla_{\partial_s u} J_t) \partial_s u, J_t (u) \partial_s u) | + |\omega( J_t(u) \nabla_s \partial_s u, J_t(u) \partial_s u)|\\
& \leq & c_{3} |\partial_s u|^3 + |\omega( \nabla_s \partial_s u, \partial_s u)|.
\eea
Evaluating at $t = i$ where $i$ is $ 0$ or $1$, the boundary conditions on $u$ imply that $\partial_s u (s, i) \in T_{u(s,i)} L_i$,  and since $L_i$ is totally geodesic  $\nabla_s \partial_s u(s,i) \in T_{u(s,i)} L_i$. Thus, at $t = i$,  $\omega(\nabla_s \partial_s u, \partial_s u) = 0$.  
Since $\nabla_t J_t = \nabla_{\partial_t u} J_t + \dot{J}_t$, there is a constant $c_{4} \geq 0$ such that $\| \nabla_t J_t\| \leq c_{4} (|\partial_t u| + 1) $.  With this, the second term can be estimated by
\bea
|\omega(\partial_s u, (\nabla_t J_t) \partial_s u)| & \leq & c_{5} |\partial_s u|^2 (|\partial_t u| +1)\\
& \leq & c_{6} (|\partial_s u|^3 + |\partial_s u|^2)\\
& \leq & c_{7}(|\partial_s u|_{J_t}^3 + |\partial_s u|^2_{J_t})\\
& = & c_{7}(e^{3/2} + e).
\eea
To estimate the last term, by compactness there is a uniform constant $c_8>0$ such that
\bea
|(\nabla_{\partial_t u} \omega)( \partial_s u, J_t(u) \partial_s u)| & \leq & c_8 |\partial_t u| |\partial_s u|^2\\
& \leq & c_2c_8|\partial_t u|_{J_i} |\partial_s u|_{J_i} ^2\\
& = & c_2c_8|\partial_s u|_{J_i} ^3 = c_2c_8 e^{3/2}.
\eea
Combining all estimates gives 
$
\left|\frac{\partial e} {\partial t} \right|_{t=i}  \leq  D e + d e^{3/2},
$
proving (\ref{outernorm1}), (\ref{outernorm2}).

\end{proof}

We use the mean-value inequalities of \cite{wehrheim} for $n = 2$. \footnote{ In \cite{wehrheim} the Laplace operator is $- \partial_s^2 - \partial_t^2$, , so at first sight some inequalities appear reversed from those in \cite{wehrheim}.}

\begin{appthm}[Theorem 1.1, \cite{wehrheim}]
There exist constants $C, \mu > 0$ such that the following holds.  Let $B_r(0) \subset \R^2$ be the ball of radius $0 < r \leq 1$. Suppose that the non-negative function $e \in C^2(B_r(0), [0,\infty))$ satisfies for some $A_0, A_1, a \geq 0$,

\[
\triangle e \geq - A_0 - A_1 e - a e^2 \ \ \  \mbox{and} \ \ \ \int_{B_r(0)} e \leq \mu a^{-1}.
\]
Then
\[
e(0) \leq C A_0 r^2 + C(A_1 + r^{-2}) \int_{B_r(0)} e.
\]
\end{appthm}

\begin{appthm}[Theorem 1.3, \cite{wehrheim}]
There exists a constant $C$,and for all $a, b \geq 0$ there exists $\mu(a,b) >0$ such that the following holds: Consider the (partial) ball $D_r(y) \subset \bH^2$ for some $r > 0$ and $y \in \bH^2$.  Suppose that $e \in C^2(D_r(y), [0,\infty))$ satisfies for some $A_0, A_1, B_0, B_1 \geq 0$
\[
\left\{ \begin{array}{ll}
\triangle e & \geq - A_0 - A_1 e - ae^2,\\
\frac{\partial}{\partial \nu}\big\lvert_{\partial \bH^2} e& \leq B_0 + B_1e + be^{3/2}
\end{array}\right. \ \ \ \mbox{and} \ \ \ \int_{D_r(y)} e \leq \mu(a,b).
\]
Then 
\[
e(y) \leq C A_0 r^2 + C B_0 r + C(A_1 + B_1^2 + r^{-2}) \int_{D_r(y)} e.
\]
\end{appthm}

We apply these theorems to the estimates of Lemma \ref{mean_est}. In the following we write $I_{1/2} = \{ s\in I | [s -1/2,s+1/2]\subset I\}$, i.e., all points in the interval $I$ which are at least distance 1/2 from the endpoint(s) of $I$.  

\begin{appcor}
Suppose that $u: I \times [0,1] \to M$ is a solution of (\ref{floer_traj}). Then there is a $\delta > 0$ such that if $E(u) < \delta$, then there is a constant $C$ depending on $M, J_t, L_0, L_1$ such that for every $s \in I_{1/2}$, 
\begin{equation} \label{mean_value}
e(s,t) \leq C \int_{s-1/2}^{s+1/2} \int_0^1 e \ dt\ ds. 
\end{equation}
\end{appcor}
\begin{proof}
This will follow from the theorems and Lemma \ref{mean_est}, since the derivative in the outer normal direction at $t=0$ is precisely $-\frac{\partial e}{\partial t}\Big\lvert_{t=0}$, while at $t=1$ it is precisely $\frac{\partial e}{\partial t}\Big\lvert_{t=1}$.

So given the estimates of Lemma \ref{mean_est}, fix $\delta$ small enough that for each $s \in I_{1/2}$, the uniform inequality
\[
\int_{D_{1/2}(s,t)} e \leq E(u) \leq \delta
\]
is enough to imply that the mean-value inequalities of the Theorems hold; the choice of such a $\delta$ therefore depends on the constants $a, b$ and $c$ in the estimates of Lemma \ref{mean_est}. Then it follows directly from the mean-value inequalities of the Theorems that there exists a uniform constant $C$ such that 
\[
e(s,t) \leq C\int_{D_{1/2}(s,t)} e \leq C \int_{s-1/2}^{s+1/2} \int_0^1 e \ dt\ ds.
\]

\end{proof}

We now study solutions of (\ref{floer_traj}) with energy $E(u) < \delta$. If $\delta$ is small enough, then the mean-value inequality provides a uniform pointwise bound on the sizes of $\partial_t u$ and $\partial_s u$ for $(s,t) \in I_{1/2}\times [0,1]$, 
\begin{equation}\label{apriori}
|\partial_t u|_{J_t}^2 = |\partial_s u|_{J_t}^2 \leq C^\prime \delta.
\end{equation} 
In particular, for a fixed $s \in I_{1/2}$, the path $\gamma_s: [0,1] \to {M}$ defined by $\gamma_s(t) : = u(s,t)$ satisfies
\begin{eqnarray}
\dist_{M}(\gamma_s(0), \gamma_s(1)) & \leq & C_1 \int_0^1 \|\dot{\gamma_s}\|_{J_t} \ dt \nonumber\\
& \leq & C_1 ( \int_0^1 \|\partial_t u\|_{J_t}^2 \ dt)^{1/2} \nonumber \\
& \leq & C_1 (C^\prime \delta)^{1/2}. 
\end{eqnarray}
By the transversality of the intersection $L_0 \cap L_1$, and the compactness of $M$,  intersection points are isolated in $M$ so if $\delta$ is small enough, all paths $\gamma_s$ are close to the {\em same} intersection point $p \in L_0\cap L_1$. Thus, given any neighborhood of this $p \in L_0 \cap L_1$, one can choose $\delta > 0$ small enough that each path $\gamma_s$, $s \in I_{1/2}$, is contained in that neighborhood.  

Define $f: I_{1/2} \to \R$ by 
\begin{equation}
f(s) = \frac{1}{2} \int_0^1 | \partial_t u |_{J_t}^2 \ dt.
\end{equation}
It follows from \cite{robbin-salamon} that given an intersection point $p \in L_0 \cap L_1$, there is a neighborhood $U$ of $p$ such that the the function $f(s)$ satisfies a convexity estimate
\begin{equation}\label{convexity_est}
\ddot{f}(s) \geq \kappa^2 f(s)
\end{equation}
for some $\kappa > 0$.  Therefore, choose $\delta$ small enough that all paths $\gamma_s$ are contained in such a neighborhood.  Combining this convexity estimate with the mean value inequality again, we can prove the following standard exponential decay results for strips.

\begin{appprop}
\label{exp_conv}
Let $u: [0,\infty) \times [0,1] \lra M$ be a solution of (\ref{floer_traj}), such that $E(u) < \infty$.  Then there exist constants $\kappa > 0$ and $A >0$ such that 
\begin{equation}
|\partial_s u(s, t)| \leq A \ e^{-\kappa s}.
\end{equation}  
\end{appprop}

\begin{proof}
We can safely ignore a compact subset $[0,T] \times [0,1]$ of the strip, which can be bounded by some fixed constant.  So choosing sufficiently large $T$ we can assume without loss of generality assume that the strip is of the form $[T, \infty)\times [0,1]$,  the energy of $u$ restricted to this strip is less than $\delta$ for which the convexity estimate (\ref{convexity_est}) holds.   Note that the fact that $E(u) < \infty$ implies that $\partial_s u \to 0$ as $s \to \infty$, so in particular $f(s) \to 0$ as $s\to \infty$.  Then for $s \geq T$, we have that 
\[
\ddot{f}(s) \geq \kappa^2 f(s). 
\]
This convexity estimate on $f(s)$ implies (explained, for instance, in \cite{robbin-salamon}) an inequality $f(s) \leq  c e^{-\kappa s}$ for some $c > 0$, i.e., 
\[
f(s) = \int_0^1 |\partial_s u(s,t)|_{J_t}^2 \ dt \leq c e^{-\kappa s}. 
\] 
Now by (\ref{mean_value}) we get
\bea
|\partial_s u|_{J_t}^2  \leq  C \int_{s-1/2}^{s+1/2} f(s) \ ds
\leq  C \ c\ e^{- \kappa (s - 1/2)}
 =  A\ e^{-\kappa s}.
\eea

\end{proof}

\begin{appprop}\label{exp_dist}
There is  a $\delta > 0$ so that the following holds.  For any solution $v: [-\rho, \rho]\times[0,1] \to M$ of (\ref{floer_traj}) with $E(v) < \delta$, there is a $\kappa > 0$ such that 
\begin{equation}
E(v; [-\rho + T, \rho - T]\times[0,1]) \leq e^{-\kappa T} E(v)
\end{equation}
for all $ 1 \leq T \leq \rho/2$. 
\end{appprop}

\begin{proof}
Take $\delta$ to small enough that $E(v) < \delta$ implies the {\em a priori} estimate (\ref{apriori})
for all $ -\rho +1 \leq s \leq \rho -1$, as well as the convexity estimate (\ref{convexity_est}).  
Let us write 
\[
E(T) := E(v; [-\rho + T, \rho - T]\times[0,1]).
\]
Then in terms of $f$, we have 
\bea
E(T) = \int_{-\rho+T}^{\rho-T} f(s) ds
\eea
is a monotone decreasing function of $T$.  Taking derivatives with respect to $T$,
\bea
& & {E}^\prime(T)  =  -f(\rho - T) - f(-\rho +T),\\
& & {E}^{\prime\prime}(T)  =  \dot{f}(\rho - T) - \dot{f}(-\rho+T)
 =  \int_{-\rho+T}^{\rho-T} \ddot{f}(s) ds
 \geq  \triangle^2 \int_{-\rho+T}^{\rho-T} f(s) ds
 =  \kappa^2 E(T).
\eea
Thus $e^{-\kappa T}({E}^\prime(T) + \kappa E(T))$ is monotone increasing.  An inequality ${E}^\prime(T) + \kappa E(T) > 0$ would imply that 
\[
e^{-\kappa T} ({E}^\prime(T) + \kappa E(T)) > \alpha > 0,
\]
so that  ${E}^\prime(T) + \kappa E(T)  >  e^{\kappa T} \alpha$ would imply 
imply$ (e^{\kappa T} E(T))^\prime > e^{2\kappa T} \alpha$, 
i.e., $E(T)$ grows exponentially, which is impossible. Hence
${E}^\prime(T) + \kappa E(T) \leq 0$, so $e^{\kappa T} E(T)$ is monotone decreasing, and 
$e^{\kappa T} E(T) \leq E(0) = E(v)$
implies $E(T) \leq e^{-\kappa T} E(v)$.

\end{proof}

Applying the mean-value inequality we get the following description of the behavior of long pseudo-holomorphic strips with small energy.

\begin{appprop}\label{exp_decay_strip}
There is  a $\delta > 0$ so that the following holds.  For any solution $v: [-\rho, \rho]\times[0,1] \to M$ of (\ref{floer_traj}) with $E(v) < \delta$, there is a $\kappa > 0$ and $A > 0$ which depend only on $M, \omega, J_t$, $L_0$ and $L_1$, such that 
\begin{equation}
|\partial_s v|_{J_t}^2 \leq A \delta e^{-\kappa |s|}
\end{equation}
for all $s \in [-\rho+1, \rho-1]$.
\end{appprop}

\begin{proof}
This is an application of the mean-value inequality to the previous lemma.  By (\ref{mean_value}) we have, for each $s \in [-\rho+1/2, \rho-1/2]$,
\bea
|\partial_s u(s,t) |_{J_t}^2 & \leq & C \int_{-|s|+1/2}^{|s|-1/2} \ \int_0^1 \ |\partial_s u|_{J_t}^2 \ ds\ dt 
 =  C E(|s|-1/2)\\
& \leq & C E(v) e^{- \kappa(|s| -1/2)}
 \leq  C \delta e^{\kappa/2} e^{-\kappa |s|}
 =:  A \delta e^{-\kappa |s|}.
\eea

\end{proof}

\def\cprime{$'$}


\end{document}